%% file: sharp_santalo_paper_2026-09-29_1623.tex
\documentclass[11pt]{article}

\usepackage[margin=1in]{geometry}
\usepackage{amsmath,amsthm,amssymb,mathtools}
\usepackage{booktabs,array}
\usepackage{microtype}
\usepackage{xcolor}
\usepackage[hidelinks]{hyperref}
\usepackage{tikz}
\usetikzlibrary{arrows.meta,calc,positioning}
\usepackage{placeins}

\newtheorem{theorem}{Theorem}[section]
\newtheorem{proposition}[theorem]{Proposition}
\newtheorem{lemma}[theorem]{Lemma}
\newtheorem{corollary}[theorem]{Corollary}
\theoremstyle{definition}

\theoremstyle{remark}
\newtheorem{remark}[theorem]{Remark}

\newcommand{\R}{\mathbb R}
\newcommand{\vol}{\operatorname{vol}}
\newcommand{\Argmax}{\operatorname*{arg\,max}}
\newcommand{\Leg}{\mathcal L}

\newcommand{\dd}{d}

\definecolor{revisionred}{RGB}{180,0,0}

\title{{A sharp Santal\'o inequality in a slab\\with centrally symmetric sections}}
\author{Shiri Artstein-Avidan\thanks{School of Mathematical Sciences, Tel Aviv University, Tel Aviv, Israel.}}
\date{}

\begin{document}
\maketitle

\begin{abstract}
	
	We  prove   a sharp  Santal\'o-type inequality for the product $V(\varphi)V({\cal L}\varphi)$ where $\varphi$ is an even geometric convex function on $\R^n$, ${\cal L}$ is the Legendre transform, and $V$ is given by $V(\varphi)=\int(1+\varphi)^{-(n+1)}$. 
	For this size functional $V(\Leg\varphi)= V({\cal A}\varphi)$ where ${\cal A}$ is functional polarity, so the same inequality holds for both dualities. Through a projective correspondence, the problem is equivalent  to a new  Santal\'o-type inequality with respect to standard volume and standard polarity, on  a class of not necessarily centered convex bodies. This class consists of convex bodies  $K\subset\mathbb R^n\times[-1,1]$ containing $\{0\}\times[-1,1]$, for which every horizontal section is centrally symmetric about the distinguished vertical axis.  We determine the sharp upper bound for $\operatorname{vol}_{n+1}(K)\operatorname{vol}_{n+1}(K^\circ)$ in this class and classify all equality cases.  
	When $n=1$, the extremizers are the disk and its
	horizontal linear images. When $n\ge2$, the ball is no longer a
	maximizer, and the extremizer is unique up to horizontal linear
	transformations and the reflection $t\mapsto-t$.
	The symmetric  Santal\'o inequality applied to horizontal sections, together with monotone transport, reduces the problem to a maximization problem for increasing curves in the square $[-1,1]^2$. We solve this problem  using a global potential when $n=1$ and two   Hamilton-Jacobi branches when $n\ge2$. 	
	Finally, in dimension one we provide a sharp family of functional
	 Santal\'o  inequalities for the Legendre duality, interpolating between the volume $V$ and
	the exponential volume
	$\int \exp(-\varphi)$, 
	with centered quadratics as the only equality cases.
\end{abstract}

\section{Introduction}

For a convex body $K\subset\R^m$ containing the origin, its polar is
\[
 K^\circ=\{y\in\R^m:\langle x,y\rangle\le1\text{ for every }x\in K\}.
\]
The classical Blaschke-Santal\'o inequality~\cite{Santalo1949,MeyerPajor1990,Schneider2014,ArtsteinAvidanGiannopoulosMilmanbook} identifies ellipsoids as the
maximizers of the volume product after the appropriate centering.  This inequality has many variants, one of the influential variants being the functional form  \cite{BallThesis1986} with the non-symmetric counterpart  \cite{ArtsteinKlartagMilman2004, Lehec, FradeliziMeyer2007}, where the objects are convex functions, the duality is the Legendre transform $\Leg$, and the volume of a convex function $\varphi$ is given 
by $\int \exp(-\varphi)$.  Another functional variant, for which the sharp constant and extremizers are not known in general, is with the same volume measuring, but a different duality, the so-called ``polarity of functions'' $\cal A$ on the class of even geometric convex functions, that is, nonnegative even convex functions which vanish at the origin \cite{Artstein-Slomka}, see also \cite{GSS}.

In this note we discuss a variant of the  Blaschke-Santal\'o inequality which can be presented in two equivalent forms. 
In the functional form, we consider even geometric
convex functions on $\R^n$ and replace the exponential volume by
\[
V(\varphi)=\int_{\R^n}(1+\varphi(x))^{-(n+1)}\dd x.
\]
We determine the sharp upper bound for
$V(\varphi)V(\Leg\varphi)$, assuming that both volumes are positive
and finite, and classify all equality cases. This gives the sharp
form, at exponent $n+1$, of an inequality of Bobkov
\cite{Bobkov2010}. A particular feature of this volume is that
\[
V(\Leg\varphi)=V(\mathcal A\varphi),
\]
where $\mathcal A$ is functional polarity. Thus the same sharp
inequality holds for both dualities.

The equivalent geometric problem concerns usual volume and usual
polarity, but a different class of convex bodies. Fix a centered
vertical segment of length $2$ and its polar slab. We consider
convex bodies satisfying
\[
\{0\}\times[-1,1]\subset K\subset\R^n\times[-1,1],
\qquad
(x,t)\in K\Longleftrightarrow(-x,t)\in K.
\]
Thus every horizontal section is centrally symmetric about the
vertical axis, but the body itself need not be centrally symmetric
or have its centroid at the origin. This class is invariant under
polarity.

The connection between the two problems is given by a projective
transformation of the epigraph of a geometric convex function.
It produces a body in the above class whose volume is a fixed
dimensional multiple of $V(\varphi)$, and carries the Legendre
transform to ordinary polarity. This explains both the exponent
$n+1$ in the functional volume and the normalization in the
geometric problem. The correspondence is described in
Section~\ref{sec:functional}.

We determine the sharp upper bound for
$\vol_{n+1}(K)\vol_{n+1}(K^\circ)$ in this class and classify all
equality cases. The answer depends on the dimension. For $n=1$,
namely for bodies in $\R^2$, the disk and its horizontal linear
images are the only maximizers. For every $n\ge2$, the ball is
no longer a maximizer. The extremizer is unique up to horizontal
linear transformations and the reflection $t\mapsto-t$.
After such a linear transformation, it is a body of revolution
with one apex and an $n$-dimensional face at the opposite end of
the slab. The sharp constant and the profile of this body are
described in Section~\ref{sec:statements}.

In the functional formulation, the distinction is between
quadratic and nonquadratic extremizers. In dimension one, the
equality functions are precisely $\varphi(x)=cx^2/2$, $c>0$,
whereas in every dimension $n\ge2$ the quadratic functions are
no longer extremal. In dimension one we also prove a sharp
family of inequalities interpolating between the volume $V$
and the exponential volume $\int\exp(-\varphi)$. The equality
functions remain the centered quadratics throughout this
family.

The proof is carried out  in the geometric setting.
Brunn's principle and the symmetric Blaschke-Santal\'o inequality
applied to horizontal sections reduce the problem to concave
profiles. Monotone transport then bounds the volume product
in terms of an integral over an increasing curve in the square
$[-1,1]^2$. We bound this integral by constructing potentials
whose increments dominate the integrand. For $n=1$, a single
potential gives the sharp bound. For $n\ge2$, we need to use two
Hamilton-Jacobi branches.

The two-branch construction reduces the bound to the
maximization of a scalar function. The magical point is
that the critical-point equation for this function is also
the condition that makes the two characteristic curves meet properly.
Thus the maximizing parameter gives not only an upper bound
but also a curve attaining it. From this curve we reconstruct
the extremal profile and obtain the equality classification.

The precise statements are given in Section~\ref{sec:statements}.
The geometric proof is developed in
Sections~\ref{sec:transport}-\ref{sec:equality-profiles}, from the transport
reduction to the classification of equality bodies.
Section~\ref{sec:functional} establishes the functional
correspondence, and Section~\ref{sec:legendre-interpolation}
proves the one-dimensional interpolation.
Section~\ref{sec:earlier-functional-santalo} compares our results
with earlier functional inequalities. The appendix proves
continuity of the scalar bound, a derivative identity for it, endpoint exclusion as well as uniqueness of its
maximizing parameter.

\section{Precise statements and extremizers}\label{sec:statements}


Throughout, $n\ge1$ is the dimension of a horizontal section, so the bodies
under consideration live in $\R^{n+1}$.  We write
\[
 \kappa_n=\vol_n(B_2^n).
\]

\subsection{The profile problem}

For a convex body
$K\subset\R^n\times[-1,1]$ with $\{0\}\times [-1,1]\subset K$ write
\[
K_t=\{x\in\R^n:(x,t)\in K\}.
\]
Assume that every $K_t$ is centrally symmetric about the origin and define
its volume radius by
\begin{equation*}
	r(t)=\left(\frac{\vol_n(K_t)}{\kappa_n}\right)^{1/n}.
\end{equation*}
By Brunn concavity principle,  $r:[-1,1]\to[0,\infty)$ is a nonzero upper semicontinuous  concave function, 
which we call a ``profile function''.  
For $|s|\le1$, the horizontal section of the polar is
\begin{equation}\label{eq:polar-section}
	(K^\circ)_s
	=\bigcap_{-1< t<1}(1-st)K_t^\circ.
\end{equation}
The endpoint constraints are already contained by closure, so it is enough to intersect over the interior levels.

We define, for any  nonzero concave profile $r:[-1,1]\to[0,\infty)$, the 
  polar profile 
\begin{equation*}
 r^\circ(s)
 =\inf_{\{t\in(-1,1):\,r(t)>0\}}
   \frac{1-st}{r(t)},
 \qquad -1\le s\le1.
\end{equation*}
The endpoint values may equivalently be defined by one-sided limits; they do
not affect any integral below.  Note that  by the symmetric Blaschke-Santal\'o inequality in $\R^n$~\cite{MeyerPajor1990,Schneider2014}
\begin{align*}
	\vol_n((K^\circ)_s)
	&\le\inf_t\vol_n((1-st)K_t^\circ)\le\kappa_n\inf_t\left(\frac{1-st}{r(t)}\right)^n
	=\kappa_n(r^\circ(s))^n.
\end{align*}

Thus
\[
\vol_{n+1}(K)
=\kappa_n\int_{-1}^1r(t)^n\dd t,
\qquad
\vol_{n+1}(K^\circ)
\le\kappa_n\int_{-1}^1(r^\circ(s))^n\dd s.
\]
For the body of revolution
\[
K(r)=\{(x,t):\|x\|_2\le r(t),\ -1\le t\le1\},
\]
one has $K(r)^\circ=K(r^\circ)$, and both volume
formulas are equalities.


 \subsection{The sharp constants}

 We next  define the scalar quantity which gives the sharp constant.
 For $0<a\le1$, define 
 \begin{equation}\label{eq:Sigma-def}
 	\Sigma_n(a)
 	=
 	2p_a\log\frac1a
 	+
 	\int_{-a}^{1}
 	\frac{(1-w)^n}
 	{\sqrt{4p_a^2+w(1-w)^n}+2p_a}\dd w\quad {\text{where}}  \quad 
 		p_a=\frac12\sqrt{a(1+a)^n}.
 \end{equation}
 At $a=0$ we set
 \begin{equation} 
 	\Sigma_n(0)
 	=
 	\int_{-1}^{1}(1-t^2)^{n/2}\dd t
 	=
 	\frac{\kappa_{n+1}}{\kappa_n}.
 \end{equation}
In Appendix~\ref{sec:proofofLemmaSigma-cont} we prove 
that  $\Sigma_n$ is continuous on $[0,1]$  
 and therefore we can set
 \begin{equation}\label{eq:M-def}
 	M_n=\max_{0\le a\le1}\Sigma_n(a).
 \end{equation}
In Section~\ref{sec:sharpness} we show that the bound $M_n$
is attained by a closed calibrated path, and in
Appendix~\ref{sec:unique-maximizer} we prove that the
corresponding maximizing parameter is unique. More precisely:

 \begin{proposition}[The scalar maximizer]\label{prop:scalar-maximizer}
 	For $n=1$, $\Sigma_1$ is strictly decreasing on $[0,1]$. Thus its
 	unique maximizer is
$
 	a_1=0$ and $M_1=\Sigma_1(0)=\frac{\pi}{2}$.
 	For every $n\ge2$, there is a unique $a_n\in(0,1)$ such that
$ 	M_n=\Sigma_n(a_n)$. 
 	Moreover, $\Sigma_n$ is strictly increasing on $[0,a_n]$ and strictly
 	decreasing on $[a_n,1]$.
 \end{proposition}

 For $n=1$, let
 \[
 r_1(t)=\sqrt{1-t^2}.
 \]
 For $n\ge2$, let $T_n$ denote the  increasing graph
 associated with the parameter $a_n$, constructed in
 Section~\ref{sec:sharpness}, choosing the orientation which crosses
 the line $t=0$ before the line $s=0$. For $-1<t<1$ let
 \begin{equation}\label{eq:rn-definition}
 	r_n(t)
 	=
 	(1-tT_n(t))^{1/2}T_n'(t)^{1/(2n)},
 \end{equation}
 and extend it to the endpoints by its one-sided limits.
 Section~\ref{sec:equality-profiles} shows that $r_n$ is a concave
 profile.
 
 \begin{theorem}[Sharp profile inequality]\label{thm:profile}
 	For every  nonzero concave profile\footnote{Throughout, profiles are taken to be upper semicontinuous on
 		$[-1,1]$, so their endpoint values agree with their one-sided limits} $r:[-1,1]\to[0,\infty)$, 
 	\begin{equation}\label{eq:profile-main}
 	 \int_{-1}^1r(t)^n\dd t \int_{-1}^1(r^\circ(s))^n\dd s
 	 \le M_n^2.
 	\end{equation}
 	Equality holds if and only if
 	\[
 	r(t)=c\,r_n( t) \quad \text{or}\quad r(t)=c\,r_n(- t) 
 	\]
 	for some $c>0$.
 	For $n=1$ the maximizing profile is even and the two choices   coincide.
 \end{theorem}

\subsection{The geometric form}
The geometric consequence is the following.
\begin{theorem}[Sharp convex body inequality]\label{thm:body-main}
	Let $K\subset\R^n\times[-1,1]$ be a convex body satisfying
	\[
	\{0\}\times[-1,1]\subset K,
	\qquad
	(x,t)\in K\Longleftrightarrow(-x,t)\in K.
	\]
	Then
	\begin{equation}\label{eq:body-bound-main}
		\vol_{n+1}(K)\vol_{n+1}(K^\circ)
		\le \kappa_n^2M_n^2.
	\end{equation}
	Equality holds if and only if, for some $A\in GL(n)$ and
	$\varepsilon\in\{-1,1\}$,
	\begin{equation}\label{eq:equality-body-form}
		K=
		\left\{
		(Ax,t):
		\|x\|_2\le r_n(\varepsilon t),\ -1\le t\le1
		\right\}.
	\end{equation}
\end{theorem}

For $n=1$ this gives
\[
\operatorname{area}(K)\operatorname{area}(K^\circ)\le\pi^2,
\]
with equality precisely for the horizontal linear images of the disk.
For every $n\ge2$, the ball is not a maximizer. In the orientation
$\varepsilon=1$, the equality body has an apex at $t=-1$ and a
positive-radius face at $t=1$; the other orientation is obtained by
reflection in the horizontal hyperplane.

\subsection{The functional form}

The inequality can equivalently be presented in a functional form. 
For a proper lower semicontinuous convex function
$\varphi:\R^n\to[0,\infty]$ with $\varphi(0)=0$, let
\[
\Leg\varphi(y)
=
\sup_{x\in\R^n}
\bigl(\langle x,y\rangle-\varphi(x)\bigr)
\]
and set
\begin{equation}\label{eq:functional-volume}
	V(\varphi)
	=
	\int_{\R^n}\frac{\dd x}{(1+\varphi(x))^{n+1}}.
\end{equation}
Let
\[
q(x)=\frac{\|x\|_2^2}{2}.
\]

\begin{theorem}[Sharp functional inequality]\label{thm:functional-main}
	Let $\varphi:\R^n\to[0,\infty]$ be proper, lower semicontinuous,
	even and convex, with $\varphi(0)=0$, and assume that
	$V(\varphi)$ and $V(\Leg\varphi)$ are positive and finite. Then
	\begin{equation}\label{eq:functional-bound}
		V(\varphi)V(\Leg\varphi)
		\le
		\frac{(n+1)^2}{2^{n+2}}\kappa_n^2M_n^2.
	\end{equation}
%
	For $n=1$ this becomes
	\[
	V(\varphi)V(\Leg\varphi)
	\le\frac{\pi^2}{2},
	\]
	with equality if and only if
	\[
	\varphi(x)=\frac{c x^2}{2},
	\qquad c>0.
	\]
	
	For $n\ge2$, equality holds precisely for the functions whose
	projective bodies, under the correspondence of
	Section~\ref{sec:functional}, are the equality bodies in
	Theorem~\ref{thm:body-main}. Equivalently, up to an invertible
	linear change of variables, there are two equality functions,
	corresponding to the profiles $r_n(t)$ and $r_n(-t)$; they are
	Legendre dual to one another.
\end{theorem}

\begin{remark}
	The Legendre transform and functional polarity are two
	order-reversing involutions on the class of lower
	semicontinuous geometric convex functions, and both
	preserve evenness
	\cite{ArtsteinAvidanMilman2009,ArtsteinAvidanMilman2011}.
Up to the corresponding linear changes, they are unique. The Legendre
transform is given by 
\[
 \Leg\varphi(y)=\sup_{x\in\R^n}
 \bigl(\langle x,y\rangle-\varphi(x)\bigr),
\]
and the polarity transform by
\[
 \mathcal A\varphi(y)=
 \begin{cases}
 \displaystyle
 \sup_{\{x:\,\varphi(x)>0\}}
 \dfrac{\langle x,y\rangle-1}{\varphi(x)},
 &0\ne y\in\{\varphi=0\}^{\circ},\\[3mm]
 0,&y=0,\\
 +\infty,&y\notin\{\varphi=0\}^{\circ},
 \end{cases}
\]
with the convention that the supremum of the empty set is zero.
Section~\ref{sec:functional} also shows that
\begin{equation}\label{eq:A-L-volume}
	V(\mathcal A\varphi)=V(\Leg\varphi),
\end{equation}
where $\mathcal A$ is functional polarity. Consequently,
Theorem~\ref{thm:functional-main} gives the identical sharp
inequality, with the same equality cases, when $\Leg\varphi$ is
replaced by $\mathcal A\varphi$.
Geometrically, $V$ is a constant multiple of the volume of a diagonal section
of the homogeneous epigraph cone (see \cite{ArtsteinICM})
\begin{equation*}
C_\varphi
=\overline{\bigl\{(sx,s,sz):s>0,\ (x,z)\in\operatorname{epi}\varphi\bigr\}}
\subset\R^n\times\R\times\R.
\end{equation*}
\end{remark}

\input{sharp_santalo_optimal_profiles_tikz_fragment_2026-09-01.tex}

\input{sharp_santalo_optimal_paths_tikz_fragment_2026-09-01.tex}

\FloatBarrier

 \subsection{The one-dimensional interpolation}

 In the case $n=1$, the extremal bodies are disks and the extremal functions
 are quadratic.  We interpolate between the rational functional volume $V$ from \eqref{eq:functional-volume} used in
 Theorem~\ref{thm:functional-main} and the classical exponential functional
 Blaschke-Santal\'o inequality
 \cite{BallThesis1986,ArtsteinKlartagMilman2004,FradeliziMeyer2007}.
 
 For $1\le\beta<\infty$, set
 \[
 V_\beta(\varphi)
 =
 \int_{\R}
 \left(1+\frac{\varphi(x)}{\beta}\right)^{-\beta-1}\dd x,
 \]
 and set
 \[
 V_\infty(\varphi)=\int_{\R}e^{-\varphi(x)}\dd x.
 \]
 Thus $V_1=V$ in dimension one.
 
 \begin{theorem}[Sharp one-dimensional Legendre interpolation]
 	\label{thm:legendre-interpolation}
 	Let $1\le\beta\le\infty$, and let
 	$\varphi:\R\to[0,\infty]$ be proper, lower semicontinuous, even
 	and convex, with $\varphi(0)=0$. Assume that
 	$V_\beta(\varphi)$ and $V_\beta(\Leg\varphi)$ are positive and
 	finite. Then
 	\begin{equation}\label{eq:legendre-interpolation}
 		V_\beta(\varphi)V_\beta(\Leg\varphi)
 		\le V_\beta(q)^2,
 		\qquad q(x)=\frac{x^2}{2}.
 	\end{equation}
 	For $1\le\beta<\infty$,
 	\begin{equation}\label{eq:legendre-constant}
 		V_\beta(q)^2
 		=
 		2\beta\pi
 		\left(
 		\frac{\Gamma\!\left(\beta+\frac12\right)}
 		{\Gamma(\beta+1)}
 		\right)^2,
 	\end{equation}
 	while
 	\[
 	V_\infty(q)^2=2\pi.
 	\]
 	Equality holds if and only if
 	\[
 	\varphi(x)=\frac{cx^2}{2}
 	\qquad\text{for some }c>0.
 	\]
 \end{theorem}
 
 The endpoint $\beta=1$ is Theorem~\ref{thm:functional-main} in
 dimension one, while $\beta=\infty$ is Ball's even functional
 Santal\'o inequality \cite{BallThesis1986}.
 Section~\ref{sec:legendre-interpolation} gives a direct proof of
 the full family.

\section{Monotone transport and the path problem}\label{sec:transport}

The elementary inequality behind the profile duality is
\begin{equation}\label{eq:polar-pointwise}
 r(t)r^\circ(s)\le1-ts,
 \qquad -1\le t,s\le1.
\end{equation}
For convenience, set
\begin{equation*}
A(r):= \int_{-1}^1r(t)^n\dd t,
	\qquad
 B(r): = \int_{-1}^1(r^\circ(s))^n\dd s. 
\end{equation*}
The profile duality inequality combined with monotone transport produces the following bound. 

\begin{lemma}[Transport reduction]\label{lem:transport}
Let $r$ be as in Theorem~\ref{thm:profile}.  There is an increasing,
absolutely continuous map $T:[-1,1]\to[-1,1]$ with $T(-1)=-1$ and
$T(1)=1$ such that
\begin{equation*}
 \sqrt{A(r)B(r)}
 \le I_n(T):=
 \int_{-1}^{1}(1-tT(t))^{n/2}\sqrt{T'(t)}\dd t.
\end{equation*}
\end{lemma}

\begin{proof}
Set $f=r^n$, $g=(r^\circ)^n$, $A=\int f$, and $B=\int g$.  Let $T$ be the
increasing transport from $f(t)\dd t/A$ to $g(s)\dd s/B$. More precisely, 
since $r$ and $r^\circ$ are nonzero nonnegative concave functions,
both $f$ and $g$ are continuous and strictly positive on $(-1,1)$.
Define
\[
F(t)=\frac1A\int_{-1}^t f(\tau)\dd\tau,
\qquad
G(s)=\frac1B\int_{-1}^s g(\sigma)\dd\sigma.
\]
Then $F$ and $G$ are continuous increasing bijections from $[-1,1]$
onto $[0,1]$, and are $C^1$ with strictly positive derivative in
$(-1,1)$. Hence
\[
T=G^{-1}\circ F
\]
is continuous and increasing on $[-1,1]$, is $C^1$ on $(-1,1)$,
and satisfies
\begin{equation}\label{eq:transport-Jacobian}
	\frac{g(T(t))}{B} T'(t)=\frac{f(t)}A,
	\qquad -1<t<1.
\end{equation}
In particular $T'(t)>0$ in the interior. Since $T$ is locally
absolutely continuous on $(-1,1)$, monotone, and continuous at the
endpoints, it is absolutely continuous on $[-1,1]$.

By \eqref{eq:polar-pointwise} at $(t, T(t))$ we see
$f(t)g(T(t))\le(1-tT(t))^n$.  Substituting
in \eqref{eq:transport-Jacobian} gives
\[
 \sqrt{\frac BA}\,\frac{f(t)}{\sqrt{T'(t)}}
 \le(1-tT(t))^{n/2}.
\]
Multiplication by $\sqrt{T'(t)}$ and integration proves the claim.
\end{proof}

\begin{remark}
The only inequality in the proof is
\eqref{eq:polar-pointwise}. Consequently, equality in
Lemma~\ref{lem:transport} holds if and only if
\[
r(t)r^\circ(T(t))=1-tT(t)
\]
for almost every $t\in(-1,1)$.
\end{remark}

We next enlarge the class of admissible paths by forgetting that the
transport path is a graph.
   A
\emph{causal curve} is an absolutely continuous curve
\[
 \gamma(u)=(t(u),s(u))\in[-1,1]^2,
 \qquad 0\le u\le1,
\]
whose two coordinates are nondecreasing and whose endpoints are
$(-1,-1)$ and $(1,1)$.  Define
\begin{equation*}
 I_n(\gamma)
 =\int_0^1(1-t(u)s(u))^{n/2}
       \sqrt{t'(u)s'(u)}\dd u.
\end{equation*}
For the graph
$\gamma(t)=(t,T(t))$, this agrees with $I_n(T)$. Thus an upper bound for $I_n$ on all causal curves gives, in
particular, an upper bound for the original profile problem. 

The following observation is the ``calibration principle'' used throughout the proof and explains the role of a potential.

\begin{lemma}[Calibration principle]\label{lem:calibration-principle}
Let $U\subset\R^2$ be open and let $\Phi\in C^1(U)$ satisfy
\begin{equation*}
 \Phi_t\ge0,
 \qquad \Phi_s\ge0,
 \qquad
 4\Phi_t\Phi_s\ge(1-ts)^n.
\end{equation*}
Along every portion of a causal curve contained in $U$,
\begin{equation}\label{eq:calibration-bound}
 \frac{d}{du}\Phi(t(u),s(u))
 \ge(1-t(u)s(u))^{n/2}\sqrt{t'(u)s'(u)}
\end{equation}
for almost every $u$.
\end{lemma}

\begin{proof}
Since $t', s'\ge 0$ almost everywhere, the arithmetic-geometric mean inequality implies
\[
 \Phi_t t'+\Phi_s s'
 \ge2\sqrt{\Phi_t\Phi_s t's'}
 \ge(1-ts)^{n/2}\sqrt{t's'}.
\]
\end{proof}

Thus, on a region where $\Phi_t,\Phi_s\ge0$, a solution of the
Hamilton-Jacobi equation
\[
4\Phi_t\Phi_s=(1-ts)^n
\]
bounds the action of every causal arc by the corresponding increment of
$\Phi$.  We say that a causal curve\footnote{The adjective \emph{causal}
	is borrowed from Lorentzian geometry: the quadratic form
	$(\tau,\sigma)\mapsto\tau\sigma$ is indefinite, and the condition
	$\tau,\sigma\ge0$ selects one of the two cones on which it is
	non-negative.  The terminology is only mnemonic; no result from
	Lorentzian geometry is used below.}
is \emph{calibrated by $\Phi$} on an interval if equality  in \eqref{eq:calibration-bound} holds almost
everywhere in the interval.
In this case   calibration is equivalent to
\begin{equation}\label{eq:AMGM-equality}
	\Phi_t(t(u),s(u))\,t'(u)
	=
	\Phi_s(t(u),s(u))\,s'(u)
	\qquad\text{for almost every }u.
\end{equation}
A calibrated arc therefore has action equal to the increment of
$\Phi$ along it. Different portions of a curve may be calibrated by
different potentials; if the potentials agree at the switching
points, the corresponding increments telescope.

\section{The planar case: one global potential}\label{sec:n1}

When $n=1$, the path problem admits a single global calibration.
Consider
\begin{equation}\label{eq:planar-potential}
	\Phi_1(t,s)
	=\frac12\left(\arcsin t+s\sqrt{1-t^2}\right).
\end{equation}
For $-1<t<1$,
\[
(\Phi_1)_t=\frac{1-ts}{2\sqrt{1-t^2}},
\qquad
(\Phi_1)_s=\frac{\sqrt{1-t^2}}2.
\]
Thus both partial derivatives are nonnegative and
\[
4(\Phi_1)_t(\Phi_1)_s=1-ts.
\]
The function $\Phi_1$ itself extends continuously to the closed square.
Applying Lemma~\ref{lem:calibration-principle} away from the sides
$t=\pm1$ and then passing to the limit gives, for every causal curve,
\begin{equation}\label{eq:n1-path-bound}
	I_1(\gamma)
	\le \Phi_1(1,1)-\Phi_1(-1,-1)
	=\frac{\pi}{2}.
\end{equation}

For a graph $s=T(t)$, the same computation reads
\begin{align*}
 \sqrt{(1-tT(t))T'(t)}
 &\le\frac12\left(
 \frac{1-tT(t)}{\sqrt{1-t^2}}
 +\sqrt{1-t^2}\,T'(t)\right)=\frac{d}{dt}\Phi_1(t,T(t)).
\end{align*}
Equality in \eqref{eq:n1-path-bound} for such a graph therefore implies
equality in the arithmetic-geometric mean inequality almost everywhere,
namely
\begin{equation}\label{eq:n1-equality-ode}
	(1-t^2)T'(t)=1-tT(t)
	\qquad\text{for almost every }-1<t<1.
\end{equation}

 Every absolutely continuous solution of \eqref{eq:n1-equality-ode} is
 of the form
 \[
 T(t)=t+C\sqrt{1-t^2}.
 \]
 If $C>0$, then $T'$ is negative near $t=1$, while if $C<0$, then
 $T'$ is negative near $t=-1$. Hence an increasing transport must have
 $C=0$. Thus the diagonal
 \[
 T(t)=t
 \]
 is the unique maximizing transport graph.

Combining \eqref{eq:n1-path-bound} with Lemma~\ref{lem:transport} gives
$A(r)B(r)\le\pi^2/4$. 
Suppose equality holds. Then the transport graph is $T(t)=t$, and
equality holds in the pointwise polar estimate along this graph.
 From  \eqref{eq:transport-Jacobian}, $\frac{(r^\circ(t))^n}{B}
 =
 \frac{r(t)^n}{A}.$
 Since $n=1$, there is a constant $\lambda>0$ such that
 \[
 r^\circ(t)=\lambda r(t).
 \]
 Equality in \eqref{eq:polar-pointwise} then gives
 \[
 \lambda r(t)^2=1-t^2.
 \]
 Consequently, for some $c>0$,
 \[
 r(t)=c\sqrt{1-t^2},
 \qquad
 r^\circ(t)=c^{-1}\sqrt{1-t^2}.
 \]
 This proves the $n=1$ part of Theorem~\ref{thm:profile}.
 It also agrees with
the scalar definition of $M_1$.
The special feature of the planar case is that the Hamilton-Jacobi
equation admits a solution affine in one variable,
\[
\Phi(t,s)=A(t)+sB(t).
\]  
For such an ansatz, $4\Phi_t\Phi_s$ is affine in $s$, and hence cannot
equal $(1-ts)^n$ for all $s$ when $n>1$. This does not rule out other
global potentials, but our sharp higher-dimensional bound will instead
come from two potentials which are joined along a turning hyperbola.
The simplicity of the planar calibration is also what allows the
one-dimensional interpolation proved in
Section~\ref{sec:legendre-interpolation}.

\section{Higher dimensions: the two Hamilton-Jacobi branches}
\label{sec:noether}

We now prove the path bound needed for the higher-dimensional
profile inequality. We first construct two potentials and then
apply them to an arbitrary causal curve. This will establish
the upper bound; showing that it is attained is a separate step.
Although our main interest is $n\ge2$, the construction and
the path bound below are valid for every $n\ge1$.

\subsection{The separated Hamilton-Jacobi equation}

On each open quadrant, set
\[
 w=ts,
 \qquad
 \rho=\frac12\log\left|\frac ts\right|.
\]
Within each quadrant, $w$ specifies the hyperbola $ts=w$,
while $\rho$ records position along that hyperbola. The hyperbolic scaling
\[
 (t,s)\longmapsto(e^\lambda t,e^{-\lambda}s)
\]
fixes $w$ and sends $\rho$ to $\rho+\lambda$.  The path Lagrangian
\[
 L(t,s,t',s')=(1-ts)^{n/2}\sqrt{t's'}
\]
is invariant under this action (note that the square itself and its endpoints in particular are {\em not} preserved under this mapping). 

On smooth interior stationary arcs with
$t',s'>0$, it is the symmetry behind the conserved quantity in
Noether's principle \cite{Arnold1989}. Here we use this principle only to motivate looking for a potential with a  constant $\rho$-derivative.
 The conserved
quantity  following from this principle  will  be derived directly in graph coordinates in
\eqref{eq:first-integral-T}.

We therefore seek a potential which is a solution of
$
 4\Phi_t\Phi_s=(1-ts)^n
$
in the ``separated'' form
\[
 \Phi(w,\rho)=2p\rho+W(w),
 \qquad\text{equivalently}\qquad
 \Phi(t,s)=p\log\left|\frac ts\right|+W(ts).
\]
A direct calculation gives
\[
 \Phi_t=\frac pt+sW'(w),
 \qquad
 \Phi_s=-\frac ps+tW'(w). 
\]
Thus the Hamilton-Jacobi equation becomes
\begin{equation}\label{eq:Wprime}
 (W'(w))^2=\frac{4p^2+w(1-w)^n}{4w^2}.
\end{equation}
This gives the two choices of sign, which 
is the source of the two branches.

For a causal curve $\gamma = (t,s)$, define its turning parameter by
\[
 a(\gamma)=\max\left\{0,-\min_{u\in[0,1]}t(u)s(u)\right\}.
\]
Fix $a\in (0,1]$. For a curve with turning parameter  $a(\gamma) = a>0$,  the range of $w=ts$ is contained in $[-a,1]$. After interchanging $t$ and $s$ if
necessary, we may assume that the mixed-sign portion of the curve
lies in the fourth quadrant.
Because $x(1+x)^n$ is strictly increasing for $x\ge0$, the expression 
$ 4p^2+w(1-w)^n
$ is nonnegative throughout this interval if and only if $ 4p^2\ge a(1+a)^n$.
We choose the least positive value with this property, 
\begin{equation}\label{eq:p-a}
	p=p_a=\frac12\sqrt{a(1+a)^n}.
\end{equation}
With this choice, the square root vanishes at $w=-a$, so the
two derivative branches meet there. We will also choose their
additive constants so that the resulting potentials agree on
$ts=-a$.

\subsection{The two potentials}

For $0<a\le1$, with $p_a$ as in \eqref{eq:p-a}, we integrate the branch
\[
W'(w)=\frac{\sqrt{4p_a^2+w(1-w)^n}}{2w}
\]
by separating off its singular part $p_a/w$. The remainder
extends continuously across $w=0$. We integrate this remainder
across the whole interval $[-a,1]$ and choose the constant so
that $W_a(-a)=0$. Thus we define
\begin{equation}\label{eq:W-def}
	W_a(w)
	=
	p_a\log\frac{|w|}{a}
	+
	\frac12\int_{-a}^{w}
	\frac{(1-v)^n}
	{\sqrt{4p_a^2+v(1-v)^n}+2p_a}\dd v,
	\qquad w\in[-a,1]\setminus\{0\}.
\end{equation}
The integrand is continuous on $[-a,1]$, including at $v=0$
and $v=-a$. Only the explicit logarithmic term is singular at
$w=0$. Direct differentiation gives
\begin{equation}\label{eq:W-derivative}
	W_a(-a)=0,
	\qquad
	W_a'(w)=\frac{\sqrt{4p_a^2+w(1-w)^n}}{2w},
	\qquad -a<w<1,\quad w\ne0.
\end{equation}
Thus $W_a$ and $-W_a$ give the two choices in
\eqref{eq:Wprime}.

For $ts\ne0$ and $-a\le ts\le1$, set
\[
\Phi_{a,\pm}(t,s)
=
p_a\log\left|\frac ts\right|\pm W_a(ts).
\]
For $ts\ne0$ and $-a\le ts\le1$, their first derivatives are
given by the following formulas, with derivatives at the
boundary understood through the local $C^1$ extensions
described below:
\begin{equation}\label{eq:two-potentials-derivatives}
	\begin{aligned}
		(\Phi_{a,\pm})_t
		&=
		\frac{2p_a\pm\sqrt{4p_a^2+ts(1-ts)^n}}{2t},\\
		(\Phi_{a,\pm})_s
		&=
		\frac{-2p_a\pm\sqrt{4p_a^2+ts(1-ts)^n}}{2s},
	\end{aligned}
	\qquad
	4(\Phi_{a,\pm})_t(\Phi_{a,\pm})_s=(1-ts)^n.
\end{equation}

Within the square $[-1,1]^2$, we use $\Phi_{a,-}$ where
$s<0$ and $ts\ge-a$, and $\Phi_{a,+}$ where $t>0$ and
$ts\ge-a$. The square root in
\eqref{eq:two-potentials-derivatives} is at least $2p_a$
when $ts\ge0$, and at most $2p_a$ when $-a\le ts\le0$.
The derivative formulas therefore show that both partial
derivatives are nonnegative on the respective domains away
from the axes. The same holds at the relevant axis crossings
by the extensions below.
Indeed, the logarithmic singularities cancel in the combinations
needed to cross the axes: 
\begin{align*}
	\Phi_{a,-}(t,s)
	&=
	p_a\log\frac{a}{s^2}
	-
	\frac12\int_{-a}^{ts}
	\frac{(1-v)^n}
	{\sqrt{4p_a^2+v(1-v)^n}+2p_a}\dd v,\\
	\Phi_{a,+}(t,s)
	&=
	p_a\log\frac{t^2}{a}
	+
	\frac12\int_{-a}^{ts}
	\frac{(1-v)^n}
	{\sqrt{4p_a^2+v(1-v)^n}+2p_a}\dd v.
\end{align*}
These formulas extend $\Phi_{a,-}$ smoothly across $t=0$
when $s<0$, and $\Phi_{a,+}$ smoothly across $s=0$ when
$t>0$. In particular, the nonnegativity of the partial
derivatives and the Hamilton-Jacobi identity hold there as well.

The potentials also have local $C^1$ extensions at the turning
boundary $ts=-a$. To see this, extend the continuous integrand
in \eqref{eq:W-def} continuously beyond the endpoints of
$[-a,1]$. Its primitive is then $C^1$, and the preceding
formulas give the required extensions. This justifies the
chain rule along admissible absolutely continuous arcs,
including arcs which touch or follow the turning boundary. These extensions are used only to justify the chain rule;
the Hamilton-Jacobi identity is asserted on the original
admissible domains.

On the common turning boundary, where $ts=-a$ and $t>0>s$,
we have
\begin{equation}\label{eq:potentials-match}
	\Phi_{a,-}(t,s)=\Phi_{a,+}(t,s),
\end{equation}
because $W_a(-a)=0$. Their first derivatives agree there as well:
\[
(\Phi_{a,-})_t=(\Phi_{a,+})_t=\frac{p_a}{t},
\qquad
(\Phi_{a,-})_s=(\Phi_{a,+})_s=-\frac{p_a}{s}.
\]

Finally, the corner values give
\[
\Phi_{a,+}(1,1)-\Phi_{a,-}(-1,-1)
=
2W_a(1)
=
\Sigma_n(a),
\]
where the last equality follows from \eqref{eq:Sigma-def}.
Thus the scalar quantity introduced in Section~\ref{sec:statements}
is precisely the endpoint difference of these two potentials.
In the next subsection we use them to bound the action of a causal
curve with turning parameter $a$.

\begin{figure}[t]
\centering
\begin{tikzpicture}[scale=2.55,>=Latex]
  \draw[thick] (-1,-1) rectangle (1,1);
  \draw[dashed] (-1,0)--(1,0);
  \draw[dashed] (0,-1)--(0,1);
  \node[below] at (1,-1.02) {$t$};
  \node[left] at (-1.02,1) {$s$};
  \node[below left] at (-1,-1) {$(-1,-1)$};
  \node[above right] at (1,1) {$(1,1)$};

  \draw[gray,thick,domain=.25:1,samples=80]
       plot (\x,{-.25/\x});
  \node[gray,right] at (.82,-.32) {$ts=-a$};

  \draw[very thick,blue]
    (-1,-1) .. controls (-.75,-.76) and (-.28,-.62) .. (0,-.55)
    .. controls (.18,-.52) and (.37,-.51) .. (.5,-.5);
  \draw[very thick,red]
    (.5,-.5) .. controls (.64,-.36) and (.74,-.12) .. (.78,0)
    .. controls (.83,.28) and (.89,.63) .. (1,1);
  \fill (.5,-.5) circle (.025);
  \node[below right] at (.5,-.5) {$P$};
  \node[blue,align=center] at (-.46,-.18) {$\Phi_{a,-}$\\before $P$};
  \node[red,align=center] at (.53,.52) {$\Phi_{a,+}$\\after $P$};
  \draw[->,blue] (-.62,-.7)--(-.47,-.64);
  \draw[->,red] (.82,.18)--(.86,.39);
\end{tikzpicture}
\caption{A schematic causal path entering the fourth quadrant
	and reaching its minimum product $ts=-a$ at $P$.  We use $\Phi_{a,-}$ before $P$ and $\Phi_{a,+}$ after $P$.
	The former extends across $t=0$, the latter across $s=0$,
	and the two potentials agree on $ts=-a$.
	The curve is not assumed to be calibrated.}
\label{fig:two-layer}
\end{figure}
 
\FloatBarrier

\subsection{The path bound}\label{sec:two-layer}

We now return to an arbitrary causal curve $\gamma$ and use
$a=a(\gamma)$, which need not equal the maximizing parameter
$a_n$. When $a>0$, we split the curve at a point where $ts$
reaches its minimum, using the minus potential before that
point and the plus potential afterwards. We do not assume
that $ts$ is monotone on either portion. The agreement of the
potentials at the switching point makes the two increments
telescope. No equality curve is being assumed or constructed
at this stage. 

\begin{proposition}[Two-layer calibration and equality]
	\label{prop:two-layer-calibration}
	For every $n\ge1$ and every causal curve $\gamma$, with
	$a=a(\gamma)$,
	\[
	I_n(\gamma)\le\Sigma_n(a)\le M_n.
	\]
	If $a=0$, equality in the first inequality holds if and only if
	the image of $\gamma$ is the diagonal segment.
	
	If $a>0$, interchange the coordinates if necessary so that the
	mixed-sign portion lies in the fourth quadrant. Choose any
	$u_*\in(0,1)$ such that
	\[
	t(u_*)s(u_*)=-a.
	\]
	Equality in the first inequality holds if and only if the
	restriction of $\gamma$ to $[0,u_*]$ is calibrated by
	$\Phi_{a,-}$ and its restriction to $[u_*,1]$ is calibrated by
	$\Phi_{a,+}$.
	
	Finally, $I_n(\gamma)=M_n$ if and only if equality holds in the
	first inequality and $
	a(\gamma)\in\Argmax_{[0,1]}\Sigma_n.$  
\end{proposition}

\begin{proof}
	Suppose first that $a>0$, and choose the orientation and $u_*$
	as in the statement. Monotonicity of the coordinates gives
	\[
	s(u)\le s(u_*)<0\quad(0\le u\le u_*),
	\qquad
	t(u)\ge t(u_*)>0\quad(u_*\le u\le1).
	\]
	Also $ts\ge-a$ along the whole curve. Thus each portion lies
	in the domain of the corresponding potential and stays away
	from the axis on which that potential is singular. The local
	$C^1$ extensions constructed above justify the chain rule on
	both portions, including at points of $ts=-a$.
	
	Writing $\Phi$ for the potential used on the relevant portion,
	the Hamilton-Jacobi identity gives, almost everywhere,
	\[
	\frac{\dd}{\dd u}\Phi(\gamma(u))
	-(1-ts)^{n/2}\sqrt{t's'}=
	\Phi_t t'+\Phi_s s'
	-2\sqrt{\Phi_t\Phi_s\,t's'}=
	\left(\sqrt{\Phi_t t'}-\sqrt{\Phi_s s'}\right)^2.
	\]
	The potentials agree at $\gamma(u_*)$, and their corner values
	satisfy
	\[
	\Phi_{a,+}(1,1)-\Phi_{a,-}(-1,-1)=\Sigma_n(a).
	\]
	Integration therefore gives the exact gap identity
	\begin{align}
		\Sigma_n(a)-I_n(\gamma)
		&=
		\int_0^{u_*}
		\left(
		\sqrt{(\Phi_{a,-})_t t'}
		-
		\sqrt{(\Phi_{a,-})_s s'}
		\right)^2 
	 +
		\int_{u_*}^{1}
		\left(
		\sqrt{(\Phi_{a,+})_t t'}
		-
		\sqrt{(\Phi_{a,+})_s s'}
		\right)^2.
		\label{eq:two-layer-gap}
	\end{align}
	Here all partial derivatives are evaluated along $\gamma$.
	Both integrands are nonnegative. This proves the bound and,
	by \eqref{eq:AMGM-equality}, the stated equality criterion.
	
	Suppose now that $a=0$. Then $ts\ge0$ throughout the curve.
	By monotonicity, the curve consists of an initial portion in
	the third quadrant, a possible portion on the coordinate axes,
	and a terminal portion in the first quadrant. Setting $p=0$ in \eqref{eq:Wprime} on $w>0$ gives the primitive
	\[
	W_0(w)=\frac12\int_0^w
	\frac{(1-v)^{n/2}}{\sqrt v}\dd v,
	\qquad 0\le w\le1.
	\]
	On the portions with $ts>0$, use $-W_0(ts)$ in the third
	quadrant and $W_0(ts)$ in the first quadrant. These potentials
	have nonnegative partial derivatives and satisfy
	\[
	4\Phi_t\Phi_s=(1-ts)^n.
	\]
	Apply the calibration on the portions where $ts\ge\varepsilon$
	and then let $\varepsilon\downarrow0$. Since $W_0(0)=0$ and
	$W_0$ is continuous, the initial and terminal portions each
	have action at most $W_0(1)$. The portions on the coordinate
	axes contribute zero, since one coordinate is constant there.
	Consequently,
	\[
	I_n(\gamma)\le2W_0(1)=\Sigma_n(0).
	\]
	
	If equality holds, the calibration gap vanishes on each open
	same-sign portion, and hence
	\[
	s\,t'=t\,s'
	\qquad\text{almost everywhere where }ts>0.
	\]
	Thus $s/t$ is constant on each such portion. The initial corner
	$(-1,-1)$ and the terminal corner $(1,1)$ force both constants
	to be $1$. By continuity, both portions meet the axes at the
	origin, and monotonicity forces any intervening portion to
	remain there. The image of $\gamma$ is therefore the diagonal.
	
	Conversely, along a diagonal curve $\gamma(u)=(h(u),h(u))$,
	\[
	I_n(\gamma)
	=
	\int_0^1(1-h(u)^2)^{n/2}h'(u)\dd u
	=
	\int_{-1}^1(1-v^2)^{n/2}\dd v
	=
	\Sigma_n(0).
	\]
	
	The inequality $\Sigma_n(a)\le M_n$ follows from the definition
	of $M_n$. By Proposition~\ref{prop:scalar-maximizer}, equality
	holds in this last inequality exactly when $a=a_n$.
\end{proof}

\subsection{The form of equality curves}

The equality criterion in Proposition \ref{prop:two-layer-calibration} is conditional: it does not yet show
that a curve attaining $\Sigma_n(a)$ exists for a prescribed
$a$. We next describe the form that such a curve, when it exists, must have. In the next section we will complete the argument in showing that for $a$ which is critical for $\Sigma_n$, such a curve indeed exists. 

There is no equality curve with $a=1$. Indeed, a causal curve
with this parameter must pass through $(1,-1)$ or $(-1,1)$,
and monotonicity forces it to follow the corresponding two
sides of the square. Its action is zero, whereas
$\Sigma_n(1)>0$.

Suppose therefore that $0<a<1$ and
$I_n(\gamma)=\Sigma_n(a)$, with the fourth-quadrant orientation.
On each calibrated portion, the equality
\[
\Phi_t t'=\Phi_s s'
\]
implies
\[
(t',s')
=
\lambda(u)\bigl(\Phi_s(\gamma(u)),\Phi_t(\gamma(u))\bigr),
\qquad \lambda(u)\ge0.
\]
Here $\Phi$ is the corresponding potential. Its partial
derivatives have positive sum on the relevant domain, so we
may take
\[
\lambda(u)=
\frac{t'(u)+s'(u)}
{\Phi_t(\gamma(u))+\Phi_s(\gamma(u))}.
\]
The multiplier is allowed to vanish.

From
\eqref{eq:two-potentials-derivatives}, 
\[
t(\Phi_{a,-})_t+s(\Phi_{a,-})_s
=
-\sqrt{4p_a^2+ts(1-ts)^n},\quad 
t(\Phi_{a,+})_t+s(\Phi_{a,+})_s
=
\sqrt{4p_a^2+ts(1-ts)^n}.
\]
Therefore, considering   $w(u)=t(u)s(u)$ and using the above identity, 
\begin{equation}\label{eq:w'neg}
	w'
	=
	-\lambda
	\sqrt{4p_a^2+w(1-w)^n}\le0
	\qquad\text{on the first portion},
\end{equation}
and
\begin{equation}\label{eq:w'pos}
	w'
	=
	\lambda
	\sqrt{4p_a^2+w(1-w)^n}\ge0
	\qquad\text{on the second portion}.
\end{equation}
We see that for an equality curve, the product first decreases to
$-a$ and then increases. The set of parameters at which
$w=-a$ is an interval, possibly reduced to one point.

On a compact portion away from the turning hyperbola,
the characteristic vector field is smooth and nonzero,
and $\lambda$ is integrable. With
\[
\tau(u)=\int_{u_0}^u\lambda(v)\dd v,
\]
uniqueness for the differential equation gives locally
$\gamma(u)=\eta(\tau(u))$, where
$\eta'=(\Phi_s(\eta),\Phi_t(\eta))$.
Thus zeros of $\lambda$ affect the parametrization,
but not the characteristic image.
In other words, away from the interval where $w=-a$, the characteristic vector field is
smooth, and the oriented image of the curve follows an
incoming characteristic of $\Phi_{a,-}$ and an outgoing
characteristic of $\Phi_{a,+}$, respectively. 



On $ts=-a$, the
two characteristic directions agree and are tangent to the
hyperbola, since
\[
\Phi_t=\frac{p_a}{t},
\qquad
\Phi_s=-\frac{p_a}{s},
\qquad
s\Phi_s+t\Phi_t=0.
\]
Consequently, the image of an equality curve consists of an
incoming characteristic of $\Phi_{a,-}$, a possible increasing arc of
$ts=-a$, and an outgoing characteristic of $\Phi_{a,+}$. The hyperbola arc
may reduce to a point.

Conversely, any causal curve with this calibrated structure
makes both gaps in \eqref{eq:two-layer-gap} vanish and
therefore attains $\Sigma_n(a)$.

We have now proved the upper bound in the profile problem.
Indeed, for the transport map $T$ associated with any admissible
profile $r$,
\[
\sqrt{A(r)B(r)}
\le I_n(T)
\le \Sigma_n(a(T))
\le M_n.
\]
Thus $A(r)B(r)\le M_n^2$. What remains is to show that this
bound is attained.

For the upper bound, the parameter $a$ was read from an
arbitrary curve. For sharpness, we reverse the direction:
we start with the maximizing parameter $a_n$ and seek a
curve attaining $M_n$. The agreement of the two potentials
on $ts=-a_n$ does not by itself imply that their
characteristics meet. In Section~\ref{sec:sharpness} we prove
that, at this parameter, the characteristic of $\Phi_{-}$ from $(-1,-1)$
and the characteristic of $\Phi_{+}$  ending at $(1,1)$ do meet directly.
In Section~\ref{sec:equality-profiles} we then reconstruct a
concave profile for which the transport inequality is also
an equality.

\section{Closing the characteristics and attaining the path bound}
\label{sec:sharpness}

The preceding section gives $I_n(\gamma)\le M_n$ for every
causal curve. We now show that this bound is attained.
For a fixed $0<a<1$, we follow the characteristic of
$\Phi_{a,-}$ from $(-1,-1)$ to its first contact with
$ts=-a$, and the characteristic of $\Phi_{a,+}$ backwards
from $(1,1)$ to the same hyperbola. The question is whether
their contact points agree.

We will describe the separation of these points by a function
$\Psi_n(a)$ and show that they agree exactly when $\Psi_n(a)=0$.
The key identity is
\[
\Sigma_n'(a)=2p_a'\Psi_n(a).
\]
We will also show that the maximizing parameter (for $n\ge 2$) lies in $(0,1)$.
Thus the critical-point condition for the scalar bound forces
the two characteristics to join. At this parameter, they form
an increasing curve in the square with action $M_n$.
The reconstruction of a concave profile attaining the original
inequality is carried out in Section~\ref{sec:equality-profiles}.

The characteristic construction and the derivative identity
in the first three subsections are valid for every $n\ge1$.
In the final subsection we assume $n\ge2$ and show that
the maximizing parameter lies in $(0,1)$.

\subsection{The characteristic equations}

Fix $0<a<1$, and let $\Phi$ denote either $\Phi_{a,-}$ or
$\Phi_{a,+}$ on its corresponding domain. On an open
characteristic arc where $-a<ts<1$, both partial derivatives
are positive. The arc can therefore be written locally as
a smooth strictly increasing graph $s=T(t)$, and by \eqref{eq:AMGM-equality}
calibration
is equivalent to
\begin{equation}\label{eq:characteristic-equation}
T'(t)=\frac{\Phi_t(t,T(t))}{\Phi_s(t,T(t))}.
\end{equation}
We first derive a first integral of this equation, and then
a second-order equation which will be used both to join the
arcs smoothly and to reconstruct a concave profile.

The Hamilton-Jacobi identity in
\eqref{eq:two-potentials-derivatives}, together with 
\eqref{eq:characteristic-equation} gives
\[
\Phi_t(t,T(t))
=\frac{(1-tT)^{n/2}}2\sqrt{T'},
\qquad
\Phi_s(t,T(t))
=\frac{(1-tT)^{n/2}}{2\sqrt{T'}}.
\]
On the other hand, \eqref{eq:two-potentials-derivatives}
gives, for either choice of potential,
\[
t\Phi_t(t,T(t))-T(t)\Phi_s(t,T(t))=2p_a.
\]
Combining these identities yields
\begin{equation}\label{eq:first-integral-T}
(1-tT)^{n/2}\frac{tT'-T}{\sqrt{T'}}=4p_a.
\end{equation}
This derives directly the first integral suggested by the
hyperbolic symmetry. Since the derivation uses only first
derivatives of the potential, the same identity holds along
a calibrated graph on the turning hyperbola $tT=-a$.

To obtain the second-order equation, we differentiate the
Hamilton-Jacobi identity where the potential is smooth.
Indeed, along the characteristic graph,
\[
\frac{\dd}{\dd t}\Phi_s(t,T(t))
=
\Phi_{st}(t,T(t))
+
\Phi_{ss}(t,T(t))\,T'(t),
\]
and \eqref{eq:characteristic-equation} gives
\[
T'(t)=\frac{\Phi_t(t,T(t))}{\Phi_s(t,T(t))}.
\]
Therefore 
\[
\frac{\dd}{\dd t}\Phi_s(t,T(t))
=
\frac{\partial_s(\Phi_t\Phi_s)(t,T(t))}
{\Phi_s(t,T(t))}
=
-\frac{nt(1-tT)^{n-1}}{4\Phi_s(t,T(t))}.
\]
Substituting the expression for $\Phi_s(t,T(t))$ gives
\[
\frac{\dd}{\dd t}
\left(\frac{(1-tT)^{n/2}}{2\sqrt{T'}}\right)
+
\frac{nt}{2}(1-tT)^{n/2-1}\sqrt{T'}=0.
\]
This is the Euler-Lagrange equation for the path functional
in graph form. Simplifying, we obtain
\begin{equation}\label{eq:geodesic-equation}
\frac{T''}{T'}
=
n\frac{tT'-T}{1-tT}.
\end{equation}

\begin{remark}[Convexity of the characteristic arcs]
	\label{rem:characteristic-convexity}
	Equations \eqref{eq:first-integral-T} and
	\eqref{eq:geodesic-equation} imply
	\[
	T''
	=
	\frac{4np_a\,(T')^{3/2}}
	{(1-tT)^{n/2+1}}>0.
	\]
	Thus each incoming or outgoing characteristic is strictly
	convex as a graph on its open arc away from $tT=-a$.
	This does not apply to a possible segment of the turning
	hyperbola. Such a segment is calibrated and satisfies
	\eqref{eq:first-integral-T}, but it does not satisfy
	\eqref{eq:geodesic-equation}: indeed, $T(t)=-a/t$ has
	$T''(t)=-2a/t^3<0$ in the fourth quadrant.
\end{remark}

To locate the contact points with $ts=-a$, we now describe
the characteristics in the coordinates introduced in
Section~\ref{sec:noether}:
\[
w=tT(t),
\qquad
\rho=\frac12\log\left|\frac{t}{T(t)}\right|.
\]
Away from the coordinate axes,
\[
\frac{\dd w}{\dd t}
=
\frac{t\Phi_t+T\Phi_s}{\Phi_s},
\qquad
\frac{\dd\rho}{\dd t}
=
\frac{T-tT'}{2w}
=
-\frac{p_a}{w\Phi_s},
\]
where the partial derivatives are evaluated at $(t,T(t))$.
By \eqref{eq:two-potentials-derivatives},
\[
t\Phi_t+T\Phi_s
=
\begin{cases}
-\sqrt{4p_a^2+w(1-w)^n},&\Phi=\Phi_{a,-},\\
\sqrt{4p_a^2+w(1-w)^n},&\Phi=\Phi_{a,+}.
\end{cases}
\]
Since $\Phi_s>0$, the product $w$ strictly decreases along
the incoming characteristic and strictly increases along
the outgoing characteristic, as long as $-a<w<1$.
We may therefore use $w$ as a parameter on each open arc.
Dividing the preceding derivative formulas gives
\begin{equation}\label{eq:rho-w-characteristic}
\frac{\dd\rho}{\dd w}
=
\begin{cases}
\displaystyle
\frac{p_a}{w\sqrt{4p_a^2+w(1-w)^n}},
&\text{on the incoming characteristic},\\[1.2em]
\displaystyle
-\frac{p_a}{w\sqrt{4p_a^2+w(1-w)^n}},
&\text{on the outgoing characteristic}.
\end{cases}
\end{equation}
This formula is used for $-a<w<1$, $w\ne0$, separately on
the two sides of an axis crossing. The coordinate $\rho$
diverges there, although the characteristic itself is smooth.
The parameter $w$ is not used on a segment of the turning
hyperbola, where it is constant.

\subsection{The contact points and the closing condition}

We next construct the incoming and outgoing characteristic arcs
and locate their contact points with $ts=-a$. Each arc is
stopped at its first contact with this hyperbola when followed
from its prescribed corner. We will then measure the separation
between the two contact points.

For the moment, we follow each characteristic from its
prescribed corner until it reaches $ts=-a$, without requiring
it to remain in the square. The same formulas define
$\Phi_{a,-}$ wherever $s<0$ and $-a\le ts\le1$, and
$\Phi_{a,+}$ wherever $t>0$ and $-a\le ts\le1$.
When the two contact points agree, we will show that both
arcs lie entirely in $[-1,1]^2$. 

\begin{lemma}[Construction of the two characteristic arcs]
	\label{lem:characteristic-arcs}
	For every $0<a<1$, the incoming characteristic of $\Phi_{a,-}$
	from $(-1,-1)$ reaches $ts=-a$ at a finite point $P_-(a)$.
	Its oriented image, stopped at this first contact, is unique.
	Both coordinates strictly increase along the arc, which lies
	first in the third quadrant and then in the fourth quadrant,
	crossing $t=0$ exactly once. It admits an absolutely continuous
	parametrization up to both endpoints.
	
	The outgoing characteristic of $\Phi_{a,+}$ is obtained by
	following it backwards from $(1,1)$ to its first contact
	$P_+(a)$ with $ts=-a$, and then taking the forward orientation.
	It has the same properties, crossing $s=0$ exactly once.
	It is the image of the incoming arc under
	$(t,s)\mapsto(-s,-t)$, with the orientation reversed.
\end{lemma}

\begin{proof}
	Along an incoming characteristic, $s<0$ and $t=w/s$, so
	\[
	\rho=\frac12\log|w|-\log(-s).
	\]
	Consequently, considering $s$ as a function of $w$ where $w$ goes from $1$ to $-a$, \eqref{eq:rho-w-characteristic} gives
	\begin{align*}
	\frac{\dd}{\dd w}\log(-s)
	&=
	\frac1{2w}
	-
	\frac{p_a}{w\sqrt{4p_a^2+w(1-w)^n}}\\
	&=
	\frac{(1-w)^n}
	{2\sqrt{4p_a^2+w(1-w)^n}
		\bigl(\sqrt{4p_a^2+w(1-w)^n}+2p_a\bigr)}.
	\end{align*}
	The last expression extends continuously across $w=0$.
	At $w=-a$, its only singularity is of order
	$(w+a)^{-1/2}$, since
	\[
	\left.\frac{\dd}{\dd w}
	\bigl(4p_a^2+w(1-w)^n\bigr)\right|_{w=-a}
	=
	(1+a)^{n-1}\bigl(1+(n+1)a\bigr)>0.
	\]
At the initial corner $(t,s)=(-1,-1)$ we have $w=1$,
so the initial condition in the $w$-parametrization is
$s_-(1)=-1$.
This leads to
	\[
	s_-(w)
	=
	-\exp\left[
	-\frac12\int_w^1
	\frac{(1-v)^n}
	{\sqrt{4p_a^2+v(1-v)^n}
		\bigl(\sqrt{4p_a^2+v(1-v)^n}+2p_a\bigr)}
	\dd v
	\right],
	\qquad
	t_-(w)=\frac{w}{s_-(w)}.
	\]
	These formulas are defined for $-a\le w\le1$, with the integral
	at $w=-a$ understood as a convergent improper integral.
	In particular, $s_-(w)$ stays negative and bounded away from
	zero. We have
	\[
	(t_-(1),s_-(1))=(-1,-1),
	\qquad
	t_-(0)=0,\quad s_-(0)<0.
	\]
	
	For $-a<w<1$, differentiation gives
	\begin{align*}
	s_-'(w)
	&=
	\frac{s_-(w)(1-w)^n}
	{2\sqrt{4p_a^2+w(1-w)^n}
		\bigl(\sqrt{4p_a^2+w(1-w)^n}+2p_a\bigr)}
	<0,\\
	t_-'(w)
	&=
	\frac{\sqrt{4p_a^2+w(1-w)^n}+2p_a}
	{2\sqrt{4p_a^2+w(1-w)^n}\,s_-(w)}
	<0.
	\end{align*}
	Thus both coordinates strictly increase when $w$ is traversed
	from $1$ down to $-a$. The derivatives are integrable at
	$w=-a$, so this parametrization is absolutely continuous
	on the closed interval.
	
	Moreover,
	\[
	\frac{\dd s_-}{\dd t_-}
	=
	\frac{s_-(w)^2(1-w)^n}
	{\bigl(\sqrt{4p_a^2+w(1-w)^n}+2p_a\bigr)^2}
	=
	\frac{(\Phi_{a,-})_t(t_-(w),s_-(w))}
	{(\Phi_{a,-})_s(t_-(w),s_-(w))}.
	\]
	This proves calibration, including across $t=0$ by continuity.
	The arc reaches the hyperbola for the first time at
	\[
	P_-(a)=(t_-(-a),s_-(-a)).
	\]
	Conversely, every incoming characteristic with the stated
	initial point satisfies the equation for $\log(-s)$ above
	until its first contact with the hyperbola. That equation
	determines $s$ uniquely as a function of $w$, and then $t=w/s$.
	This proves uniqueness of its oriented image.
	
	Finally,
	\[
	\Phi_{a,+}(-s,-t)=-\Phi_{a,-}(t,s).
	\]
	Reflection by $(t,s)\mapsto(-s,-t)$, together with reversal of
	orientation, therefore produces the outgoing characteristic.
	Its contact point is
	\[
	P_+(a)=(-s_-(-a),-t_-(-a)).
	\]
	The corresponding existence and uniqueness assertions follow
	from those for the incoming arc.
\end{proof}

We next define the closing mismatch geometrically:
\[
\Psi_n(a):=\rho(P_+(a))-\rho(P_-(a)).
\]
The reflection relating the two arcs changes the sign of $\rho$.
Hence
\begin{align}
\rho(P_-(a))&=-\frac12\Psi_n(a),
\label{eq:rho-turn}\quad \text{and}\quad 
\rho(P_+(a))=\frac12\Psi_n(a).
\end{align}
On the fourth-quadrant branch of $ts=-a$,
\[
t=\sqrt a\,e^\rho,
\qquad
s=-\sqrt a\,e^{-\rho}.
\]
Both coordinates increase with $\rho$, so the sign of
$\Psi_n(a)$ records the order of the two contact points.

Using
\[
\rho(P_-(a))
=
\frac12\log a-\log(-s_-(-a)),
\]
and the formula for $s_-$, we obtain
\begin{align}
\Psi_n(a)
&=
\log\frac1a
-
\int_{-a}^{1}
\frac{(1-w)^n}
{\sqrt{4p_a^2+w(1-w)^n}
	\bigl(\sqrt{4p_a^2+w(1-w)^n}+2p_a\bigr)}
\dd w
\nonumber\\
&=
\log\frac1a
+
\int_{-a}^{1}
\frac{\displaystyle
	\frac{2p_a}{\sqrt{4p_a^2+w(1-w)^n}}-1}
{w}\dd w,
\qquad 0<a<1.
\label{eq:Psi-def}
\end{align}
The first expression shows that this is an ordinary convergent
improper integral. In the second expression, obtained by
rationalization, the apparent singularity at $w=0$ is assigned
its limiting value $-1/(8p_a^2)$.

\begin{lemma}[Direct-closing criterion]
	\label{lem:closing}
	For $0<a<1$, the contact points of the two arcs in
	Lemma~\ref{lem:characteristic-arcs} agree if and only if
	\begin{equation}\label{eq:Psi-zero}
	\Psi_n(a)=0.
	\end{equation}
	In that case the arcs meet at
	\[
	P=(\sqrt a,-\sqrt a)
	\]
	and remain in $[-1,1]^2$. Together they form the graph of a
	strictly increasing absolutely continuous map
	$T:[-1,1]\to[-1,1]$, smooth on $(-1,1)$, with
	\[
	T(-1)=-1,\qquad T(1)=1,\qquad
	a(T)=a,\qquad I_n(T)=\Sigma_n(a).
	\]
\end{lemma}

\begin{proof}
	On the fourth-quadrant branch of $ts=-a$, the coordinate
	$\rho$ determines the point uniquely. Equation \eqref{eq:rho-turn} shows
	that the contact points agree exactly when $\Psi_n(a)=0$.
	Their common value of $\rho$ is then zero, giving
	$P=(\sqrt a,-\sqrt a)$.
	
	Each coordinate is strictly increasing along the incoming
	arc from $(-1,-1)$ to $P$, and along the outgoing arc from
	$P$ to $(1,1)$. Thus both arcs remain in the square and
	together form a strictly increasing graph $s=T(t)$.
	Their product decreases from $1$ to $-a$ and then increases
	to $1$, so $a(T)=a$.
	
	At $P$, the first derivatives of both potentials satisfy
	\[
	\Phi_t(P)=\Phi_s(P)=\frac{p_a}{\sqrt a}.
	\]
	The characteristic equation therefore gives
	\[
	\lim_{t\to\sqrt a-}T'(t)
	=
	\lim_{t\to\sqrt a+}T'(t)=1.
	\]
	Hence the two arcs join to a $C^1$ graph. On either side,
	\eqref{eq:geodesic-equation} gives
	\[
	T''=nT'\frac{tT'-T}{1-tT}.
	\]
	Its right-hand side is smooth near
	$(t,T,T')=(\sqrt a,-\sqrt a,1)$, since $1-tT=1+a>0$ there.
	In particular, both one-sided limits of $T''$ equal
	$2n\sqrt a/(1+a)$. The joined graph is therefore a $C^2$
	solution of this equation, and hence is smooth across the
	meeting point.
	
	The graph is continuous at $t=\pm1$ and has $T'>0$ in the
	interior. Integrating $T'$ on compact subintervals and passing
	to the endpoint limits gives
	\[
	\int_{-1}^{1}T'(t)\dd t=T(1)-T(-1)=2,
	\]
	so $T$ is absolutely continuous on $[-1,1]$.
	Finally, calibration on the two arcs and
	\eqref{eq:potentials-match} give
	\[
	I_n(T)
	=
	\Phi_{a,+}(1,1)-\Phi_{a,-}(-1,-1)
	=
	2W_a(1)=\Sigma_n(a).
	\]
\end{proof}

\begin{remark}[Equality curves with a hyperbola segment]
	\label{rem:hyperbola-bridge}
	When $\Psi_n(a)>0$, the contact points occur in the order
	$P_-(a)$ before $P_+(a)$ along the increasing branch of
	$ts=-a$. Both arcs lie in the square: for the incoming
	contact point,
	\[
	-1<s_-(-a)<0,
	\qquad
	t_-(-a)=\sqrt a\,e^{-\Psi_n(a)/2}<\sqrt a<1,
	\]
	and the outgoing arc is its reflection. Joining the contact
	points by the intervening hyperbola arc therefore gives a
	causal curve attaining $\Sigma_n(a)$, since this arc is
	calibrated by both potentials.
	
	When $\Psi_n(a)<0$, the contact points occur in the opposite
	order. The equality criterion in
	Proposition~\ref{prop:two-layer-calibration} then rules out
	a causal curve attaining $\Sigma_n(a)$. This concerns
	attainment of this particular bound, not existence of a
	maximizer for the fixed-$a$ problem, as explained in the following remark.
\end{remark}

\begin{remark}[The fixed-$a$ problem beyond the closing regime]
	The failure of the $p_a$-characteristics to close when
	$\Psi_n(a)<0$ does not mean that the relaxed problem with prescribed
	turning parameter $a$ has no maximizer.  If the momentum in the
	Hamilton-Jacobi construction is allowed to vary, $p\ge p_a$, one obtains
	a sharper family of calibration bounds.  Minimizing these bounds in $p$
	gives the exact fixed-$a$ value.
	When $\Psi_n(a)\ge0$, the minimizing momentum is $p_a$: for
	$\Psi_n(a)>0$ the maximizing curve contains the hyperbola bridge described
	above, while for $\Psi_n(a)=0$ the two characteristic arcs meet smoothly.
	When $\Psi_n(a)<0$, the minimizing momentum is strictly larger than
	$p_a$; the corresponding two characteristic arcs meet at
	$(\sqrt a,-\sqrt a)$ with a genuine corner.  In particular, in this last
	case the true fixed-$a$ maximum is strictly smaller than $\Sigma_n(a)$.
	We do not pursue this auxiliary fixed-parameter problem further here.
\end{remark}

We have reduced direct closing to the scalar equation
$\Psi_n(a)=0$. The next subsection relates this equation
to the critical-point equation for $\Sigma_n$.

\definecolor{curvefive}{RGB}{148,103,189}
\begin{figure}[t]
	\centering
	\begin{minipage}[t]{0.48\textwidth}
		\centering
		\begin{tikzpicture}[x=2.25cm,y=2.25cm,>=Latex]
			\draw[blue] (-1,-1) rectangle (1,1);
			\draw[blue!55,dashed] (-1,0)--(1,0);
			\draw[blue!55,dashed] (0,-1)--(0,1);
			\draw[curvefive!45,domain=.25:1,samples=80]
			plot (\x,{-.25/\x});
			\draw[revisionred,very thick]
			(-1,-1) .. controls (-.65,-.92) and (-.16,-.82) .. (0,-.78)
			.. controls (.16,-.76) and (.28,-.73) .. (.352,-.710);
			\draw[revisionred,very thick,densely dashed,domain=.352:.710,samples=50]
			plot (\x,{-.25/\x});
			\draw[revisionred,very thick]
			(.710,-.352) .. controls (.74,-.24) and (.77,-.10) .. (.79,0)
			.. controls (.84,.28) and (.91,.67) .. (1,1);
			\fill[revisionred] (.352,-.710) circle (.022);
			\fill[revisionred] (.710,-.352) circle (.022);
			\node[blue,below left] at (-1,-1) {$(-1,-1)$};
			\node[blue,above right] at (1,1) {$(1,1)$};
			\node[black,below] at (.53,-1.04) {$\Psi_n(a)>0$};
			\node[revisionred,scale=.8,below left] at (.352,-.710) {$P_-$};
			\node[revisionred,scale=.8,above right] at (.710,-.352) {$P_+$};
		\end{tikzpicture}
		
		\small Noncritical fixed-$a$ equality path: the dashed part is the
		increasing arc of $ts=-a$.
	\end{minipage}\hfill
	\begin{minipage}[t]{0.48\textwidth}
		\centering
		\begin{tikzpicture}[x=2.25cm,y=2.25cm,>=Latex]
			\draw[blue] (-1,-1) rectangle (1,1);
			\draw[blue!55,dashed] (-1,0)--(1,0);
			\draw[blue!55,dashed] (0,-1)--(0,1);
			\draw[curvefive!45,domain=.25:1,samples=80]
			plot (\x,{-.25/\x});
			\draw[revisionred,very thick]
			(-1,-1) .. controls (-.66,-.91) and (-.17,-.78) .. (0,-.72)
			.. controls (.20,-.66) and (.39,-.61) .. (.5,-.5);
			\draw[revisionred,very thick]
			(.5,-.5) .. controls (.62,-.38) and (.71,-.18) .. (.75,0)
			.. controls (.81,.29) and (.90,.68) .. (1,1);
			\fill[revisionred] (.5,-.5) circle (.024);
			\node[blue,below left] at (-1,-1) {$(-1,-1)$};
			\node[blue,above right] at (1,1) {$(1,1)$};
			\node[black,below] at (.53,-1.04) {$\Psi_n(a)=0$};
			\node[revisionred,scale=.8,below right] at (.5,-.5) {$P_-=P_+$};
		\end{tikzpicture}
		
		\small Critical path: the bridge collapses and the two characteristic arcs
		close directly.
	\end{minipage}
	\caption{
		Schematic equality configurations for the bound
		$I_n(\gamma)\le\Sigma_n(a)$ obtained from the $p_a$-potentials.
		When $\Psi_n(a)>0$, the characteristic arcs are joined by
		an increasing arc of $ts=-a$; when $\Psi_n(a)=0$, they meet
		smoothly at $(\sqrt a,-\sqrt a)$. When $\Psi_n(a)<0$,
		this bound is not attained. }
	\label{fig:bridge-versus-direct}
\end{figure}

\subsection{The derivative of the scalar function}
\label{sec:derivSigma}

The preceding subsection identified $\Psi_n(a)$ as the signed
separation of the two contact points on $ts=-a$. We now show
that this same quantity determines the derivative of the
calibration bound. Thus the condition that the characteristics
close is also the critical-point condition for $\Sigma_n$.

\begin{lemma}[Derivative of $\Sigma_n$]
	\label{lem:Sigma-derivative}
	For every $n\ge1$, the function $\Sigma_n$ is continuously
	differentiable on $(0,1)$ and
	\begin{equation}\label{eq:Sigma-derivative}
	\Sigma_n'(a)=2p_a'\Psi_n(a),
	\qquad 0<a<1.
	\end{equation}
	Since $p_a'>0$, the critical-point equation
	$\Sigma_n'(a)=0$ is equivalent to the direct-closing
	condition \eqref{eq:Psi-zero}.
\end{lemma}
The proof is a technical differentiation procedure which includes a delicate justification of differentiation at a
moving square-root singularity, and is deferred to
Appendix~\ref{sec:Sigma-derivative-proof}.

By Lemma~\ref{lem:Sigma-derivative}, at every interior critical point
of $\Sigma_n$, the incoming and outgoing characteristics meet
at the same point of $ts=-a$ and join smoothly.
For $n\ge2$, to obtain a curve attaining $M_n$, it remains to show that
the maximum of $\Sigma_n$ is attained in $(0,1)$.

\subsection{Attainment of the path bound}

In this section we show that the endpoints of the interval $[0,1]$ are not maximizing parameters for $\Sigma_n, n\ge 2$.
With this in hand, the derivative identity shows that at a maximizing
parameter, the two characteristics meet and produce a curve
attaining $M_n$.

\begin{proposition}[Attainment of the path bound]
	\label{prop:path-bound-attained}
	For every $n\ge2$, every maximizer of $\Sigma_n$ on $[0,1]$
	lies in $(0,1)$. At any such parameter $a$, the incoming and
	outgoing characteristics meet at $(\sqrt a,-\sqrt a)$ and
	form the graph of a strictly increasing absolutely continuous
	map $T:[-1,1]\to[-1,1]$, smooth on $(-1,1)$, with
	\[
	T(-1)=-1,\qquad T(1)=1,\qquad
	a(T)=a,\qquad I_n(T)=M_n.
	\]
	Consequently,
	\[
	\max_{\gamma\text{ causal}} I_n(\gamma)=M_n.
	\]
\end{proposition}

\begin{proof}
	By Lemma~\ref{lem:Sigma-continuity}, $\Sigma_n$ attains its
	maximum on $[0,1]$.  Appendix~\ref{sec:endpoint-exclusion}
	shows that, for $n\ge2$, neither endpoint is a maximizing
	parameter.  Hence every maximizer $a$ lies in $(0,1)$.
	
	Lemma~\ref{lem:Sigma-derivative} then gives
	\[
	0=\Sigma_n'(a)=2p_a'\Psi_n(a).
	\]
	Since $p_a'>0$, we have $\Psi_n(a)=0$.  By
	Lemma~\ref{lem:closing}, the incoming and outgoing
	characteristics meet at $(\sqrt a,-\sqrt a)$ and form a
	smooth increasing graph $T$ in the square, with
	\[
	I_n(T)=\Sigma_n(a)=M_n.
	\]
	Together with Proposition~\ref{prop:two-layer-calibration},
	this proves
	\[
	\max_{\gamma\text{ causal}} I_n(\gamma)=M_n.
	\]
\end{proof}

We write $T_n$ for the graph corresponding to the unique
maximizing parameter $a_n$; uniqueness of the parameter is
proved in Appendix~\ref{sec:unique-maximizer}.
The argument above establishes attainment without using
that uniqueness.

We have so far attained the sharp bound in the class of causal
curves. To complete the proof of Theorem \ref{thm:profile} 
 it remains to
reconstruct from the maximizing curve $T_n$ an admissible concave profile for which
the transport inequality is also an equality. This is done
in the next Section~\ref{sec:equality-profiles}.

 \section{Equality profiles}\label{sec:equality-profiles}
 
 The preceding section constructs a smooth increasing graph
 attaining the path bound. We now show that it comes from an
 admissible concave profile and that equality holds in the
 transport reduction. We then prove that every equality
 profile arises in this way.

 The first three subsections concern $n\ge2$, since the
 planar profile problem was settled in Section~\ref{sec:n1}.
 The final subsection completes the geometric theorem
 for every $n\ge1$.

 \subsection{Reconstruction and attainment}
 
 Let $T=T_n$ be the graph constructed at the maximizing
 parameter $a_n$, with its mixed-sign portion in the fourth
 quadrant. Define, for $-1<t<1$,
 \begin{equation}\label{eq:r-balanced}
 	r(t)=(1-tT(t))^{1/2}T'(t)^{1/(2n)}.
 \end{equation}
 This is the formula forced by simultaneous equality in
 \eqref{eq:polar-pointwise} and
 \eqref{eq:transport-Jacobian} when the horizontal scale is
 chosen so that $A(r)=B(r)$. We must verify that it defines
 a concave profile and that its polar has the required
 contact with $r$.
 
 The function $r$ is smooth and positive on $(-1,1)$.
 Differentiating its logarithm and using
 \eqref{eq:geodesic-equation} gives
 \begin{equation}\label{eq:log-r}
 	\frac{r'}r
 	=
 	-\frac{T+tT'}{2(1-tT)}
 	+\frac1{2n}\frac{T''}{T'}
 	=
 	-\frac{T}{1-tT}.
 \end{equation}
 Put $D=r-tr'$, the value at zero of the tangent line to
 $r$ at $t$. Then
 \begin{equation}\label{eq:contact-identities}
 	D=\frac{r}{1-tT}>0,
 	\qquad
 	T=-\frac{r'}D,
 	\qquad
 	T'=-\frac{rr''}{D^2}.
 \end{equation}
 The last identity follows by differentiating the second,
 using $D'=-tr''$. Since $T'>0$, it follows that $r''<0$.
 
 Thus $r$ is strictly concave in the interior. It is
 nonnegative and bounded above by any fixed tangent line,
 so concavity gives finite nonnegative one-sided limits at
 both endpoints. We define $r(\pm1)$ by these limits.
 This extends $r$ to an admissible profile on $[-1,1]$.
 
 We next identify its polar profile. For $t,u\in(-1,1)$,
 the tangent-line inequality gives
 \[
 r(u)
 \le r(t)+r'(t)(u-t)
 =D(t)(1-uT(t)).
 \]
 Therefore
 \[
 \frac{1-uT(t)}{r(u)}\ge\frac1{D(t)},
 \]
 with equality when $u=t$. Taking the infimum over $u$
 shows that 
 \[
r^\circ(T(t))
 =
 \frac1{D(t)}
 =
 \frac{1-tT(t)}{r(t)}.
 \]
 In particular, the pointwise polar inequality is an
 equality along the entire graph of $T$.
 
 Using \eqref{eq:r-balanced}, we obtain
 \begin{equation}\label{eq:r-q-densities}
 	r(t)^n=(1-tT(t))^{n/2}\sqrt{T'(t)},
 	\qquad
 	(r^\circ(T(t)))^n
 	=
 	\frac{(1-tT(t))^{n/2}}{\sqrt{T'(t)}}.
 \end{equation}
 Consequently,
 \[
 A(r)=I_n(T),
 \qquad
 B(r)
 =
 \int_{-1}^1(r^\circ(T(t)))^nT'(t)\dd t
 =
 I_n(T).
 \]
 The change of variables is first applied on compact
 subintervals of $(-1,1)$ and then extended to the endpoints
 by monotone convergence. Since $I_n(T)=M_n$, we have
 \[
 A(r)=B(r)=M_n,
 \qquad
 A(r)B(r)=M_n^2.
 \]
 This proves attainment of the profile bound.
 
 The profile just constructed is $r_n$ from
 \eqref{eq:rn-definition}. We call this choice of horizontal
 scale the \emph{balanced normalization}: the two profile
 integrals are equal. Indeed, multiplication of a profile
 by $c>0$ multiplies its polar by $c^{-1}$, so the two
 integrals are multiplied by $c^n$ and $c^{-n}$, respectively.

 \subsection{The equality classification}
 
 Now let $r$ be any profile satisfying
 $A(r)B(r)=M_n^2$, and let $T$ be its monotone transport map.
 Equality must hold throughout
 \[
 \sqrt{A(r)B(r)}
 \le I_n(T)
 \le \Sigma_n(a(T))
 \le M_n.
 \]
 In particular,
 \[
 I_n(T)=M_n,
 \qquad
 a(T)=a_n,
 \]
 where we use the uniqueness of the scalar maximizer from
 Proposition~\ref{prop:scalar-maximizer}.
 
 Reflection of the profile, $r(t)\mapsto r(-t)$, reflects
 its polar in the same way and replaces the transport by
 $t\mapsto-T(-t)$. Thus, after reflecting $r$ if necessary,
 we may assume that the mixed-sign portion of the transport
 graph lies in the fourth quadrant.
 
 By the equality criterion in
 Proposition~\ref{prop:two-layer-calibration} and the
 subsequent description of equality curves, the graph
 consists of an incoming characteristic of $\Phi_{a_n,-}$,
 a possible arc of $ts=-a_n$, and an outgoing characteristic
 of $\Phi_{a_n,+}$. Lemma~\ref{lem:characteristic-arcs}
 determines the incoming and outgoing arcs uniquely.
 Since $\Psi_n(a_n)=0$, their contact points agree, and the
 intervening hyperbola arc reduces to one point. Hence
 $T=T_n$.
 
 Write $A=A(r)$ and $B=B(r)$. Equality in the transport
 reduction gives, almost everywhere in $(-1,1)$,
 \[
 r(t)r^\circ(T(t))=1-tT(t),
 \qquad
 (r^\circ(T(t)))^nT'(t)=\frac BA\,r(t)^n.
 \]
 Eliminating $r^\circ(T(t))$ yields
 \[
 r(t)^{2n}
 =
 \frac AB(1-tT(t))^nT'(t).
 \]
 Therefore
 \[
 r(t)
 =
 \left(\frac AB\right)^{1/(2n)}
 (1-tT_n(t))^{1/2}T_n'(t)^{1/(2n)}
 =
 c\,r_n(t)
 \]
 for some $c>0$. Continuity gives this identity throughout
 the open interval, and the endpoint convention for profiles
 extends it to $[-1,1]$.
 
 Restoring the possible reflection, every equality profile
 has the form
 \[
 r(t)=c\,r_n(t)
 \qquad\text{or}\qquad
 r(t)=c\,r_n(-t),
 \qquad c>0.
 \]
 Conversely, multiplication by a positive constant and
 reflection preserve the product $A(r)B(r)$, so all these
 profiles attain equality. This completes the proof of
 Theorem~\ref{thm:profile}.
 
 \subsection{The profile equation and polarity}
 
 We return to the balanced profile $r=r_n$, with $T=T_n$
 and $a=a_n$, and record its differential equation and
 endpoint behaviour.

 \begin{proposition}[Equation, endpoints, and polarity of the equality profile]
 	\label{prop:equality-profile-properties}
 	Let $n\ge2$, and let $r_n$ be the balanced equality profile
 	in the orientation fixed above. Set
 	\[
 	k_n=4p_{a_n}=2\sqrt{a_n(1+a_n)^n}.
 	\]
 	Then on $(-1,1)$,
 \begin{equation}\label{eq:profile-ODE}
 	-r_n''=r_n^{n-1}(r_n-tr_n')^{n+2},
 		\end{equation}
 	and
 \begin{equation}\label{eq:profile-first-integral}
  		\frac{r_n'}{(r_n-tr_n')^{n+1}}+tr_n^n=k_n.
  	\end{equation}
 	Moreover,
 	\[
 	r_n(-1)=0,\qquad r_n'(-1+)=k_n^{-1/n},
 	\qquad
 	r_n(1)=k_n^{1/n},\qquad r_n'(1-)=-\infty,
 	\]
 	and
 	\[
 	r_n^\circ(t)=r_n(-t),\qquad -1\le t\le1.
 	\]
 	The reflected profile satisfies the same second-order equation,
 	with the sign of the constant in the first integral reversed.
 \end{proposition}

 \begin{proof}
 	For brevity, write $r=r_n$, $T=T_n$, $a=a_n$ and $k=k_n$.
 	With $D=r-tr'$, equations \eqref{eq:contact-identities}
 	and \eqref{eq:r-balanced} give
 	\[
 	T'=\frac{r^{2n}}{(1-tT)^n}=r^nD^n,
 	\qquad
 	T'=-\frac{rr''}{D^2}.
 	\]
 	Equating these expressions gives \eqref{eq:profile-ODE}.
 	Likewise, the path first integral
 	\eqref{eq:first-integral-T} becomes
 	\[
 	k
 	=
 	(1-tT)^{n/2}\frac{tT'-T}{\sqrt{T'}}
 	=
 	tr^n-\frac{T}{D^n}
 	=
 	tr^n+\frac{r'}{D^{n+1}},
 	\]
 	which proves \eqref{eq:profile-first-integral}.
 	
 	Near $t=-1$, we use the incoming characteristic.
 	Rationalizing the quotient of the derivatives in
 	\eqref{eq:two-potentials-derivatives} gives
 	\[
 	\frac{T'(t)}{(1-tT(t))^n}
 	=
 	\frac{T(t)^2}
 	{\bigl(\sqrt{4p_a^2+tT(t)(1-tT(t))^n}+2p_a\bigr)^2}
 	\longrightarrow\frac1{k^2}.
 	\]
 	In particular, $T'(t)\to0$. Since $T(-1)=-1$, this implies
 	$T(t)+1=o(1+t)$, and hence
 	\[
 	1-tT(t)\sim1+t.
 	\]
 	Formula \eqref{eq:r-balanced} now gives
 	\[
 	r(t)\sim k^{-1/n}(1+t).
 	\]
 	Moreover, \eqref{eq:log-r} yields
 	\[
 	r'(t)=-\frac{r(t)T(t)}{1-tT(t)}
 	\longrightarrow k^{-1/n}.
 	\]
 	This proves the assertions at $t=-1$.
 	
 	Near $t=1$, the outgoing characteristic similarly gives
 	\[
 	(1-tT(t))^nT'(t)
 	=
 	\frac{
 		\bigl(\sqrt{4p_a^2+tT(t)(1-tT(t))^n}+2p_a\bigr)^2}
 	{t^2}
 	\longrightarrow k^2.
 	\]
 	Since
 	\[
 	r(t)^n=(1-tT(t))^{n/2}\sqrt{T'(t)},
 	\]
 	we obtain $r(t)^n\to k$, so $r(1)=k^{1/n}$.
 	Finally, $T(t)\to1$ and $1-tT(t)\to0$, and
 	\eqref{eq:log-r} gives $r'(t)\to-\infty$.
 	
 	Under the reflection $r(t)\mapsto r(-t)$, the second-order
 	equation is unchanged, while the left-hand side of
 	\eqref{eq:profile-first-integral} changes sign.

To show the polarity property of the balanced equality profile we note that 
 by the construction in Lemma~\ref{lem:characteristic-arcs},
 the outgoing arc is the reflection of the incoming arc
 under $(t,s)\mapsto(-s,-t)$. Since the two arcs meet, the
 whole graph is invariant under this reflection. Equivalently,
 \[
 T(-T(t))=-t.
 \]
 Differentiating gives
 \[
 T'(-T(t))T'(t)=1.
 \]
 Thus, using \eqref{eq:r-balanced} and
 \eqref{eq:r-q-densities},
 \begin{align*}
 	r(-T(t))
 	&=
 	(1-tT(t))^{1/2}T'(-T(t))^{1/(2n)}\\
 	&=
 	(1-tT(t))^{1/2}T'(t)^{-1/(2n)}\\
 	&=
 	r^\circ(T(t)).
 \end{align*}
 Since $T$ maps $(-1,1)$ onto itself, and both profiles
 are continuous at the endpoints, we conclude that
 \[
 r_n^\circ(s)=r_n(-s),
 \qquad -1\le s\le1.
 \]
 In particular, the two orientations are polar to one
 another in the balanced normalization. They are distinct
 when $n\ge2$, since $r_n(-1)=0$ and $r_n(1)>0$.
  \end{proof}

\begin{figure}[t]
\centering
\begin{tikzpicture}[scale=2.1]
  \draw[blue][->] (-1.25,0)--(1.3,0) node[right] {$x$};
  \draw[blue][->] (0,-1.25)--(0,1.3) node[above] {$t$};
  \draw[dashed][blue] (-1.05,1)--(1.05,1);
  \draw[dashed][blue] (-.25,-1)--(.25,-1);
  \draw[thick][revisionred]
    (0,-1) .. controls (.42,-.75) and (.88,-.08) .. (.78,.56)
            .. controls (.74,.82) and (.58,1) .. (.50,1);
  \draw[thick][revisionred]
    (0,-1) .. controls (-.42,-.75) and (-.88,-.08) .. (-.78,.56)
            .. controls (-.74,.82) and (-.58,1) .. (-.50,1);
  \draw[thick][revisionred] (-.50,1)--(.50,1);
  \fill (0,-1) circle (.025);
  \node[left] at (-.03,-1) {$-1$};
  \node[left] at (-.03,1) {$1$};
  \node[right,align=left] at (.82,.55) {strictly concave\\profile};
  \node[above] at (0,1.02) {positive-radius face};
\end{tikzpicture}
\caption{Schematic meridian of an equality body for $n\ge2$.  The profile vanishes at $t=-1$ and has a positive value at
	$t=1$. Since $r'(1-)=-\infty$, the curved side has a horizontal
	tangent where it meets the top face in these coordinates.
}
\label{fig:profile-shape}
\end{figure}

\subsection{Completion of the geometric theorem}
\label{sec:bodies}

We now complete the proof of Theorem~\ref{thm:body-main},
for every $n\ge1$. The bound and its attainment follow from
the sectionwise estimate in Section~\ref{sec:statements}
and the profile theorem. The remaining point is to show
that all horizontal sections of an equality body are
homothetic copies of one fixed centered ellipsoid.

\begin{proof}[Proof of Theorem~\ref{thm:body-main}]
	Let $K$ satisfy the hypotheses of the theorem, and let
	\[
	r(t)=\left(\frac{\vol_n(K_t)}{\kappa_n}\right)^{1/n}
	\]
	be its volume-radius profile. Recall from
	Section~\ref{sec:statements} that
	\[
	\vol_{n+1}(K)=\kappa_n A(r),
	\qquad
	\vol_n((K^\circ)_s)
	\le\kappa_n(r^\circ(s))^n,
	\qquad -1<s<1.
	\]
	Integration and Theorem~\ref{thm:profile} give
	\[
	\vol_{n+1}(K)\vol_{n+1}(K^\circ)
	\le\kappa_n^2A(r)B(r)
	\le\kappa_n^2M_n^2.
	\]
	Equality is attained by
	\[
	K=\{(x,t):\|x\|_2\le r_n(t),\ -1\le t\le1\}.
	\]
	Invertible horizontal linear transformations and reflection
	in $t$ preserve the volume product, so every body in
	\eqref{eq:equality-body-form} is an equality case.
	
	Suppose now that equality holds. Then
	\[
	A(r)B(r)=M_n^2,
	\qquad
	\vol_{n+1}(K^\circ)=\kappa_n B(r).
	\]
	Consequently, the sectionwise estimate is an equality
	for almost every $s\in(-1,1)$. Both sides of that estimate
	are continuous on $(-1,1)$: the continuity of
	$\vol_n((K^\circ)_s)$ follows from Brunn's concavity
	principle, and $r^\circ$ is a finite concave function.
	Thus
	\[
	\vol_n((K^\circ)_s)=\kappa_n(r^\circ(s))^n
	\qquad\text{for every }-1<s<1.
	\]
	
	By the profile equality classification, $r$ is smooth
	and positive in the interior, and its transport map $T$
	is a smooth strictly increasing bijection of $(-1,1)$
	onto itself. For $n=1$, this is $T(t)=t$; for $n\ge2$,
	these properties were established in the preceding
	subsections. Moreover,
	\[
	r(t)r^\circ(T(t))=1-tT(t),
	\qquad -1<t<1.
	\]
	For each such $t$, we therefore have equality throughout
	\[
	\vol_n((K^\circ)_{T(t)})
	\le
	\vol_n((1-tT(t))K_t^\circ)
	\le
	\kappa_n\left(\frac{1-tT(t)}{r(t)}\right)^n
	=
	\kappa_n(r^\circ(T(t)))^n.
	\]
	Equality in the symmetric Blaschke-Santal\'o inequality
	\cite{MeyerPajor1990,Schneider2014} implies that $K_t$
	is a centered ellipsoid. The first comparison comes
	from an inclusion of convex bodies with equal positive
	volume, so it is also an equality of sets:
	\[
	(K^\circ)_{T(t)}=(1-tT(t))K_t^\circ.
	\]
	Using \eqref{eq:polar-section}, we obtain
	\begin{equation}\label{eq:contact-inclusion-body}
	(1-tT(t))K_t^\circ
	\subseteq
	(1-uT(t))K_u^\circ,
	\qquad -1<t,u<1.
	\end{equation}
	
	Write
	\[
	K_t=r(t)E_t,
	\qquad
	\vol_n(E_t)=\kappa_n.
	\]
	We must show that the centered ellipsoid $E_t$ does not
	depend on $t$. For fixed $t$, set
	\[
	F_t(u)=\log\frac{1-uT(t)}{r(u)}.
	\]
	The contact relation
	$r(t)r^\circ(T(t))=1-tT(t)$ shows that $F_t$ is minimized
	at $u=t$. In particular,
	\[
	F_t'(t)=0.
	\]
	Fix a compact interval $J\subset(-1,1)$. Since $r$ is
	smooth and positive, the second derivative of $F_t$
	with respect to $u$ is uniformly bounded for $t,u\in J$.
	Taylor's theorem therefore gives a constant $C_J$ such that
	\[
	0\le F_t(u)-F_t(t)\le C_J|u-t|^2,
	\qquad t,u\in J.
	\]
	Normalizing \eqref{eq:contact-inclusion-body} now yields
	\[
	E_t^\circ
	\subseteq
	e^{F_t(u)-F_t(t)}E_u^\circ
	\subseteq
	e^{C_J|u-t|^2}E_u^\circ.
	\]
	Interchanging $t$ and $u$ gives the reverse comparison.
	
	The quadratic error forces these normalized sections
	to be identical. Indeed, fix $t_0<t_1$ in $J$ and divide
	$[t_0,t_1]$ into $m$ equal subintervals. Composing the
	inclusions along the subdivision gives
	\[
	E_{t_0}^\circ
	\subseteq
	\exp\left(\frac{C_J(t_1-t_0)^2}{m}\right)
	E_{t_1}^\circ.
	\]
	Letting $m\to\infty$ and using closedness gives
	$E_{t_0}^\circ\subseteq E_{t_1}^\circ$. The reverse
	inclusion follows in the same way. Hence $E_t$ is
	constant on $J$, and therefore on $(-1,1)$.
	
	Let $E$ denote this fixed centered ellipsoid. We have
	$K_t=r(t)E$ for every interior level. The body $K$ is
	the closure of its interior-level sections: every point
	at an endpoint level can be approached along the segment
	joining it to $(0,0)$. Since $r$ is continuous on
	$[-1,1]$, it follows that
	\[
	K=\{(x,t):x\in r(t)E,\ -1\le t\le1\}.
	\]
	Finally, the profile equality classification gives
	$r(t)=c\,r_n(\varepsilon t)$ for some $c>0$ and
	$\varepsilon\in\{-1,1\}$. Writing $cE=AB_2^n$ with
	$A\in GL(n)$ gives exactly \eqref{eq:equality-body-form}.
\end{proof}

\begin{remark}
	The above direct section-wise argument is why no separate Steiner symmetrization~\cite{MeyerPajor1989,MeyerPajor1990}
	step is needed for Theorem~\ref{thm:body-main}, although one can indeed begin by taking Steiner symmetrizations with respect
	to horizontal directions (namely with respect to hyperplanes which include the special vertical one). Each of these symmetrizations  does not decrease the volume product (since the corresponding hyperplane passes through the center of mass), and taking a limit reduces the convex bodies question to the one for profiles. We find the current technique better suited for our goals as the equality cases are simpler to handle this way. 
\end{remark}

\section{Functional reformulation}\label{sec:functional}

In this section we prove Theorem~\ref{thm:functional-main} by establishing
the projective correspondence between geometric convex
functions satisfying the conditions in the theorem, and  convex bodies of the form we considered above. Under this
correspondence, the functional volume becomes ordinary volume
up to a dimensional constant, and the Legendre transform
becomes ordinary polarity. Functional polarity gives the same
polar body followed by reflection in the last coordinate.
Thus the geometric theorem gives both functional inequalities
and their equality cases.

 \subsection{The projective correspondence and volume}

As in \cite{ArtsteinICM} we  associate with a geometric convex function $\varphi$ the
homogeneous epigraph cone
\begin{equation*}
 C_\varphi
 =\overline{\bigl\{(sx,s,sz):s>0,\ (x,z)\in\operatorname{epi}\varphi\bigr\}}
 \subset\R^n\times\R\times\R.
\end{equation*}
Its section by $s=1$ is the epigraph of $\varphi$. We use instead the section
by the diagonal hyperplane
\[
 H=\{(\xi,s,z):s+z=1\}.
\]
Using the affine coordinates
\[
 (\xi,s,z)\in H
 \longmapsto (\sqrt2\,\xi,z-s)\in\R^n\times\R,
\]
we denote the corresponding image of $C_\varphi\cap H$ by $K_\varphi$.
For $(x,z)\in\operatorname{epi}\varphi$, the ray through
$(x,1,z)$ meets $H$ after multiplication by $(1+z)^{-1}$.
In the chosen coordinates, this gives the projective map
\begin{equation}\label{eq:projective-map}
  P(x,z)
=
\left(\frac{\sqrt2\,x}{1+z},\frac{z-1}{z+1}\right),
\qquad z\ge0,
\end{equation}
and hence
\[
K_\varphi
=
\overline{ P(\operatorname{epi}\varphi)}.
\]
 The diagonal-section description shows that $K_\varphi$ is closed and
convex, lies in $\R^n\times[-1,1]$, and contains
$\{0\}\times[-1,1]$.  If $\varphi$ is even, 
the cone is invariant under
$(\xi,s,z)\mapsto(-\xi,s,z)$, so
every horizontal section of $K_\varphi$ is
centrally symmetric.  

To show that $K_\varphi$ is a convex body, we use the mapping $P$ to compute the volume of $K_\varphi$. The map
$P$ is a homeomorphism from
$\R^n\times[0,\infty)$ onto $\R^n\times[-1,1)$.
Since $\operatorname{epi}\varphi$ is closed, its image is
relatively closed in the latter set. Thus the closure
defining $K_\varphi$ can add points only on the hyperplane
$t=1$, which has zero $(n+1)$-dimensional volume.
The Jacobian of $P$ is
\[
|\det DP(x,z)|
=
\frac{2^{n/2+1}}{(1+z)^{n+2}}.
\]
By change of variables, 
\begin{equation}\label{eq:body-functional-volume}
\vol_{n+1}(K_\varphi)
=
2^{n/2+1}
\int_{\R^n}
\left(
\int_{\varphi(x)}^\infty
\frac{\dd z}{(1+z)^{n+2}}
\right)\dd x = c_nV(\varphi),
\qquad
c_n=\frac{2^{n/2+1}}{n+1}.
\end{equation}

In particular, under the assumption $0<V(\varphi)<\infty$ (present in Theorem \ref{thm:functional-main}), the closed convex set
$K_\varphi$ has positive finite volume, and therefore is not merely a closed convex set but is a convex body.  
In particular, we see that it satisfies all the requirements in 
Theorem~\ref{thm:body-main}.

Conversely, let $K$ satisfy the hypotheses of
Theorem~\ref{thm:body-main}, and set
\[
E=\{(x,z)\in\R^n\times[0,\infty):
 P(x,z)\in K\}.
\]
The set $E$ is closed by continuity of $P$, and convex since $P$ is fractional linear. To see that $E$ is an epigraph, observe that
for $(x,z)\in E$ and $h\ge0$,
\[
P(x,z+h)
=
\frac{1+z}{1+z+h}P(x,z)
+
\frac{h}{1+z+h}(0_n,1)
\in K.
\]
 Thus $E$ is upward closed and is
the epigraph of a lower semicontinuous convex function
$\varphi:\R^n\to[0,\infty]$, with $\varphi(x)=+\infty$
when the corresponding vertical fiber is empty.
Since $  P(0_n,0)=(0_n,-1)\in K$, we have
$\varphi(0)=0$. The horizontal
symmetry of $K$ implies that $\varphi$ is even.

By construction,
\[
  P(\operatorname{epi}\varphi)
=
K\cap\{t<1\}.
\]
Every point of $K$ at height $t=1$ is a limit of points
on the segment joining it to $(0,-1)$, so taking closures
gives $K_\varphi=K$. The function $\varphi$ is uniquely
determined, since $ P$ is invertible on $z\ge0$.
Finally, the same computation as above gives
\[
0<V(\varphi)
=c_n^{-1}\vol_{n+1}(K)<\infty.
\]
The corresponding positivity and finiteness of
$V(\Leg\varphi)$ will follow from the polarity identity
proved below.

\subsection{The two dualities}

Legendre transform corresponds to
ordinary polarity. Indeed, for $x,y\in\R^n$ and $z,w\ge0$, a direct
calculation gives
\[
	\langle P(x,z), P(y,w)\rangle\le1
	\quad\Longleftrightarrow\quad
	\langle x,y\rangle\le z+w.
\]
Consequently, for $y\in\R^n$ and $w\ge0$,
\begin{align*}
	P(y,w)\in K_\varphi^\circ
	&\quad\Longleftrightarrow\quad
	w\ge\sup_{x\in\R^n}
	\bigl(\langle x,y\rangle-\varphi(x)\bigr)
	=\Leg\varphi(y).
\end{align*}
This identifies $P(\operatorname{epi}\Leg\varphi)$ with
$K_\varphi^\circ\cap\{t<1\}$. Taking closures, and using
$(0_n,-1)\in K_\varphi^\circ$ to recover its top section,
we obtain
\begin{equation}\label{eq:projective-polar}
	K_{\Leg\varphi}=K_\varphi^\circ.
\end{equation}

For the polarity transform, we can use that the cone symmetry $(s,z)\mapsto (z,s)$ corresponds to switching $\varphi$ and ${\cal J}\varphi$ (see \cite{ArtsteinICM}) and therefore, letting in $\R^n\times \R$  the linear reflection be given by $R(x,t) = (x, -t)$, we have that 
\begin{equation*}
	K_{\mathcal A\varphi}=  R K_{\Leg\varphi}
	=  R K_\varphi^\circ.
\end{equation*}
Alternatively, note that 
\[
\langle P(x,z), R P(y,w)\rangle\le1
\quad\Longleftrightarrow\quad
\langle x,y\rangle\le1+zw.
\]
Consequently, for $y\in\R^n$ and $w\ge0$,
\begin{align*}
  R P(y,w)\in K_\varphi^\circ
	&\quad\Longleftrightarrow\quad
	\langle x,y\rangle\le1+w\varphi(x)
	\quad\text{for every }x\in\operatorname{dom}\varphi\\
	&\quad\Longleftrightarrow\quad
	w\ge\mathcal A\varphi(y).
\end{align*}
Here the middle condition includes
$\langle x,y\rangle\le1$ whenever $\varphi(x)=0$,
as required by the definition of $\mathcal A$.

Reflection preserves volume, so
\[
V(\mathcal A\varphi)
=
c_n^{-1}\vol_{n+1}(K_\varphi^\circ)
=
V(\Leg\varphi).
\]
This proves \eqref{eq:A-L-volume}. In particular, both
transformed volumes are positive and finite.

\subsection{The functional inequality and its equality cases}

\begin{proof}[Proof of Theorem~\ref{thm:functional-main}]
	
By \eqref{eq:body-functional-volume},
\eqref{eq:projective-polar}, and Theorem~\ref{thm:body-main},
\begin{align*}
	V(\varphi)V(\Leg\varphi)
	&=
	\frac{1}{c_n^2}
	\vol_{n+1}(K_\varphi)\vol_{n+1}(K_\varphi^\circ)\\
	&\le
	\frac{\kappa_n^2M_n^2}{c_n^2}
	=
	\frac{(n+1)^2}{2^{n+2}}\kappa_n^2M_n^2.
\end{align*}
This proves \eqref{eq:functional-bound}.
The identity \eqref{eq:A-L-volume} gives the same inequality,
with the same equality cases, when $\Leg\varphi$ is replaced
by $\mathcal A\varphi$.

The only inequality above is the geometric inequality.
Thus equality holds if and only if $K_\varphi$ is one of
the equality bodies in Theorem~\ref{thm:body-main}.
The converse construction shows that every such body
corresponds to a unique function satisfying the hypotheses
of the functional theorem.

\end{proof}

\begin{remark}[Explicit projective form of the equality functions]
Let an equality body have horizontal sections $r(t)E$, where
$\vol_n(E)=\kappa_n$. 
At height $z\ge0$, the projective
map \eqref{eq:projective-map} gives
\[
\varphi(x)\le z
\quad\Longleftrightarrow\quad
\frac{\sqrt2\,x}{1+z}
\in r\!\left(\frac{z-1}{z+1}\right)E.
\]
This specifies the entire function in terms of its
equality profile and the ellipsoid $E$. In particular,
its positive sublevel sets are homothetic centered
ellipsoids, and
\[
\{\varphi=0\}=\frac{r(-1)}{\sqrt2}E.
\]
For $n\ge2$, the orientation $r=c\,r_n$, with $c>0$,
therefore gives a function vanishing only at the origin,
whereas $r(t)=c\,r_n(-t)$ gives a function vanishing on
a full-dimensional ellipsoid.
 
\end{remark}
\input{functional_extremizers_2026-09-29.tex}

\FloatBarrier

\section{One-dimensional Legendre interpolation}
\label{sec:legendre-interpolation}

In this section we prove Theorem~\ref{thm:legendre-interpolation}, a
one-parameter interpolation between the functional volume $V$ in dimension one and
the exponential volume
\cite{BallThesis1986,ArtsteinKlartagMilman2004,FradeliziMeyer2007}.  The proof
follows the same idea
as the planar case in Section~\ref{sec:n1}: monotone
transport reduces the inequality to a path problem,
which we solve using a single global potential.
For finite $\beta$, a projective change of the level
variables gives a weighted version of the planar path
functional. The diagonal remains the unique maximizing
graph, and the equality functions remain $cx^2$.  We conclude with a weighted geometric interpretation
using the correspondence from Section~\ref{sec:functional}.

\subsection{Sublevel sets and monotone transport}

Fix $1\le\beta<\infty$ and set
\begin{equation*}
 F_\beta(u)=\left(1+\frac{u}{\beta}\right)^{-\beta-1},
 \qquad
 h_\beta(u)=-F_\beta'(u)
 =\frac{\beta+1}{\beta}
  \left(1+\frac{u}{\beta}\right)^{-\beta-2}.
\end{equation*}
For $u\ge0$, let $x(u)$ and $y(u)$ denote the radii of the sublevel sets of
$\varphi$ and $\Leg\varphi$:
\begin{equation}\label{eq:sublevel-radii}
 \{x\in\R:\varphi(x)\le u\}=[-x(u),x(u)],
 \qquad
 \{y\in\R:\Leg\varphi(y)\le u\}=[-y(u),y(u)].
\end{equation}

The finiteness assumptions imply that these intervals are bounded.  Moreover, $x(u),y(u)>0$ for every $u>0$, and both $x$ and
$y$ are nondecreasing concave functions on $[0,\infty)$.
In particular, they are continuous on $(0,\infty)$.  

Since $F_\beta(u)=\int_u^\infty h_\beta(v)\dd v$, we see that
\begin{equation}\label{eq:layer-cake-beta}
 V_\beta(\varphi)=2A,
 \qquad
 V_\beta(\Leg\varphi)=2B,
\end{equation}
where
\begin{equation*}
 A=\int_0^\infty h_\beta(u)x(u)\dd u,
 \qquad
 B=\int_0^\infty h_\beta(v)y(v)\dd v.
\end{equation*}
The definitions of $x,y$ and the Legendre transform imply
\begin{equation}\label{eq:fenchel-level-set}
 x(u)y(v)\le \varphi(x(u))+\Leg\varphi(y(v))\le u+v,
 \qquad\text{for}\,\,\, u,v\ge0.
\end{equation}

Let $T:[0,\infty)\to[0,\infty)$ be the increasing transport sending
\[
 \frac{h_\beta(u)x(u)}{A}\dd u
 \quad\text{to}\quad
 \frac{h_\beta(v)y(v)}{B}\dd v.
\]
Equivalently,
\[
 \frac1A\int_0^u h_\beta(r)x(r)\dd r
 =\frac1B\int_0^{T(u)}h_\beta(r)y(r)\dd r.
\]
Since the two densities are positive and continuous on $(0,\infty)$,
$T$ is increasing and locally absolutely continuous there,
and
\begin{equation}\label{eq:transport-identity-beta}
 \frac{h_\beta(T(u))y(T(u))}{B}T'(u)
 =\frac{h_\beta(u)x(u)}{A}
 \qquad\text{for almost every }u>0.
\end{equation}
Combining \eqref{eq:transport-identity-beta} with
\eqref{eq:fenchel-level-set}, evaluated at $v=T(u)$, gives
\begin{equation}\label{eq:transport-pointwise-beta}
 \sqrt{\frac BA}\,h_\beta(u)x(u) = \sqrt{h_\beta(u)x(u) h_\beta(T(u))y(T(u))T'(u) }
 \le
 \sqrt{(u+T(u))h_\beta(u)h_\beta(T(u))T'(u)}.
\end{equation}
Integrating over $u\in[0,\infty)$  we get
\begin{equation}\label{eq:transport-bound-beta}
 \sqrt{AB}
 \le
 \int_0^\infty
 \sqrt{(u+T(u))h_\beta(u)h_\beta(T(u))T'(u)}\dd u.
\end{equation}

\subsection{Compactification and the sharp path inequality}

We now compactify the level variables in order to return to the
square $[-1,1]^2$, where the calibration argument becomes transparent.
For finite $\beta$, the natural projective change of variables turns
the transport bound above into a weighted version of the planar path
functional.  The diagonal will again be the unique maximizing path.

Introduce the decreasing projective change of variables $t,s:[0, \infty)\to (-1,1]$
\begin{equation}\label{eq:beta-compactification}
 t=\frac{\beta-u}{\beta+u},
 \qquad
 s=\frac{\beta-v}{\beta+v},\qquad u=\beta\frac{1-t}{1+t},
 \qquad
 v=\beta\frac{1-s}{1+s}.
\end{equation}
If $v=T(u)$, define the ``compactified transport'' by composing $T$ with the variable changes
\[
 \sigma(t)=
 \frac{\beta-T\!\left(\beta\frac{1-t}{1+t}\right)}
      {\beta+T\!\left(\beta\frac{1-t}{1+t}\right)}.
\]
Both coordinate changes are decreasing, while $T$ is increasing;
therefore $\sigma$ is increasing and locally absolutely continuous on
$(-1,1)$, with endpoint limits $\sigma(-1)=-1$ and $\sigma(1)=1$.  The
identities
\begin{align*}
 u+v
 &=\frac{2\beta(1-ts)}{(1+t)(1+s)}, \quad
 h_\beta(u)
 =\frac{\beta+1}{\beta}
   \left(\frac{1+t}{2}\right)^{\beta+2},\quad
 T'(u)
 =\frac{(1+t)^2}{(1+s)^2}\sigma'(t),
\end{align*}
and $|d u|=2\beta(1+t)^{-2}\dd t$ show that the right-hand side of
\eqref{eq:transport-bound-beta} is
\begin{equation}\label{eq:c-beta-J-beta}
 c_\beta J_\beta(\sigma),
 \qquad
 c_\beta=\frac{(\beta+1)\sqrt\beta}{2^{\beta+1/2}},
\end{equation}
where 
\begin{equation*}
 J_\beta(\sigma)
 =\int_{-1}^1
 \bigl[(1+t)(1+\sigma(t))\bigr]^{(\beta-1)/2}
 \sqrt{(1-t\sigma(t))\sigma'(t)}\dd t.
\end{equation*}
Notice that $J_1=I_1$, the planar path functional from
Section~\ref{sec:n1}.

\begin{proposition}[Diagonal calibration]
\label{prop:diagonal-calibration-beta}
For every $1\le\beta<\infty$ and every increasing, locally absolutely
continuous $\sigma:(-1,1)\to(-1,1)$ with endpoint limits $\sigma(-1)=-1$ and
$\sigma(1)=1$, one has
\begin{equation}\label{eq:J-beta-sharp}
J_\beta(\sigma)\le D_\beta,
\end{equation}
where
\begin{equation*}
 D_\beta
 =\int_{-1}^1(1+t)^{\beta-1}\sqrt{1-t^2}\dd t
 =2^\beta\sqrt\pi\,
   \frac{\Gamma\!\left(\beta+\frac12\right)}{\Gamma(\beta+2)}.
\end{equation*}
Equality holds if and only if $\sigma(t)=t$.
\end{proposition}

\begin{proof}
The calibration argument will work if we can find a potential
$\Phi$ with positive partial derivatives such that
\[
4\Phi_t(t,s)\Phi_s(t,s)
\ge
(1-ts)(1+t)^{\beta-1}(1+s)^{\beta-1}.
\]
We want equality along the diagonal $s=t$.	
We first explain how to find the potential.  
The planar potential in Section~\ref{sec:n1} was affine
in $s$, so its $s$-derivative was independent of $s$.
For the weighted functional $J_\beta$, we seek a potential
whose $s$-derivative instead depends on $s$ through
the factor $(1+s)^{(\beta-1)/2}$ appearing in the
integrand. Integrating this factor suggests the form
\[
\Phi(t,s)=A(t)+B(t)(1+s)^\theta,
\qquad
\theta=\frac{\beta+1}{2}.
\]
We determine $A$ and $B$ by requiring the diagonal to
be calibrated, and then verify that the resulting
potential bounds the action throughout the square.

Along the diagonal, equality both in the Hamilton-Jacobi bound and
in the arithmetic-geometric mean step requires
\[
\Phi_t(t,t)=\Phi_s(t,t)
=\frac12(1+t)^{\beta-1}\sqrt{1-t^2}.
\]
Evaluating $\Phi_s$ and $\Phi_t$ on the diagonal gives, respectively,
\[
B(t)=\frac{(1+t)^{\beta/2}\sqrt{1-t}}{\beta+1},\qquad A'(t)=\frac1{\beta+1}
\frac{(1+t)^{\beta-1/2}}{\sqrt{1-t}}.
\]
Choosing $A(-1)=0$ leads to
\begin{equation}\label{eq:Phi-beta-definition}
	\Phi_\beta(t,s)
	=
	\frac1{\beta+1}
	\int_{-1}^{t}
	\frac{(1+r)^{\beta-1/2}}{\sqrt{1-r}}\dd r
	+
	\frac{(1+t)^{\beta/2}\sqrt{1-t}}{\beta+1}
	(1+s)^{(\beta+1)/2}.
\end{equation}

The function $\Phi_\beta$ is continuous on the closed square and $C^1$ in
its interior.  At $\beta=1$, it reduces, up to the additive constant
$\pi/4$, to the potential in Section~\ref{sec:n1}:
\begin{equation*}
 \Phi_1(t,s)
 =\frac\pi4+\frac12\left(\arcsin t+s\sqrt{1-t^2}\right).
\end{equation*}

Abbreviating 
$ \xi=1+t$ and $
 \eta=1+s, $ 
differentiating \eqref{eq:Phi-beta-definition} gives
\begin{equation}\label{eq:Phi-beta-s-derivative}
 (\Phi_\beta)_s(t,s)
 =\frac12\xi^{\theta-1/2}(2-\xi)^{1/2}\eta^{\theta-1}
 =\frac12(\xi\eta)^{(\beta-1)/2}\sqrt{1-t^2},
\end{equation}
and
\begin{equation*}
 (\Phi_\beta)_t(t,s)
 =\frac{\xi^{\theta-3/2}}{2\theta\sqrt{2-\xi}}
 \left[\xi^\theta+\eta^\theta(2\theta-1-\theta\xi)\right].
\end{equation*}
A direct simplification yields
\begin{align}
 &4(\Phi_\beta)_t(t,s)(\Phi_\beta)_s(t,s)
 -(1-ts)(1+t)^{\beta-1}(1+s)^{\beta-1}
 \nonumber\\
 &\quad=
 \frac{\xi^{\beta-1}\eta^{(\beta-1)/2}}{\theta}
 \left[\xi^\theta-\eta^\theta
 -\theta\eta^{\theta-1}(\xi-\eta)\right].
 \label{eq:Phi-beta-HJ-gap}
\end{align}
Since $\theta\ge1$, the function $r\mapsto r^\theta$ is convex on
$[0,\infty)$, so the term in square brackets is nonnegative.  Hence, throughout the open square,
\begin{equation}\label{eq:Phi-beta-HJ-super}
 4(\Phi_\beta)_t(\Phi_\beta)_s
 \ge(1-ts)(1+t)^{\beta-1}(1+s)^{\beta-1}.
\end{equation}
Thus, for $\beta>1$, the potential does not solve the
Hamilton-Jacobi equation away from the diagonal: the convexity
remainder makes it a strict supersolution there, while equality
holds on the diagonal.
Equations \eqref{eq:Phi-beta-s-derivative} and
\eqref{eq:Phi-beta-HJ-super} also imply that both partial derivatives are
positive in the open square.

For almost every $t\in(-1,1)$, the arithmetic-geometric mean inequality and
\eqref{eq:Phi-beta-HJ-super} give
\begin{align*}
 &\bigl[(1+t)(1+\sigma(t))\bigr]^{(\beta-1)/2}
  \sqrt{(1-t\sigma(t))\sigma'(t)}\le
 2\sqrt{(\Phi_\beta)_t(t,\sigma(t))(\Phi_\beta)_s(t,\sigma(t))\sigma'(t)}\\
 &\qquad\le
 (\Phi_\beta)_t(t,\sigma(t))
 +(\Phi_\beta)_s(t,\sigma(t))\sigma'(t)=\frac{d}{dt}\Phi_\beta(t,\sigma(t)).
\end{align*}
Apply this estimate on compact subintervals of $(-1,1)$ and then let the
endpoints tend to $-1$ and $1$.  By continuity of $\Phi_\beta$,
\begin{equation*}
 J_\beta(\sigma)
 \le\Phi_\beta(1,1)-\Phi_\beta(-1,-1).
\end{equation*}
The second term in \eqref{eq:Phi-beta-definition} vanishes at both corners,
and the beta-function identity gives
\begin{align*}
 \Phi_\beta(1,1)-\Phi_\beta(-1,-1)
 &=\frac1{\beta+1}
   \int_{-1}^1
   (1+t)^{\beta-1/2}(1-t)^{-1/2}\dd t\\
 &=2^{\beta+1}
   B\!\left(\beta+\frac12,\frac32\right)=D_\beta.
\end{align*}
Thus \eqref{eq:J-beta-sharp} holds.  The diagonal attains equality, since
\begin{equation*}
 (\Phi_\beta)_t(t,t)
 =(\Phi_\beta)_s(t,t)
 =\frac12(1+t)^{\beta-1}\sqrt{1-t^2}.
\end{equation*}

It remains to prove uniqueness.  Suppressing the arguments
$(t,\sigma(t))$, the pointwise gap in the calibration is
\begin{align}
 &(\Phi_\beta)_t+(\Phi_\beta)_s\sigma'
 -\sqrt{(1-t\sigma)(1+t)^{\beta-1}
              (1+\sigma)^{\beta-1}\sigma'}
 \nonumber\\
 &\quad=
 \left(\sqrt{(\Phi_\beta)_t}
       -\sqrt{(\Phi_\beta)_s\sigma'}\right)^2
+
 \left(
 2\sqrt{(\Phi_\beta)_t(\Phi_\beta)_s}
 -\sqrt{(1-t\sigma)(1+t)^{\beta-1}(1+\sigma)^{\beta-1}}
 \right)\sqrt{\sigma'}.
 \label{eq:calibration-gap-decomposition}
\end{align}
If equality holds in the integrated calibration bound, then the integral of
the nonnegative pointwise gap in
\eqref{eq:calibration-gap-decomposition} is zero.  Hence both nonnegative
terms in that decomposition vanish almost everywhere.

If $\beta>1$, then $\theta>1$ and
$r\mapsto r^\theta$ is strictly convex.  Equality in
\eqref{eq:Phi-beta-HJ-gap} therefore forces $\xi=\eta$, equivalently $\sigma(t)=t$, whenever $\sigma'>0$.  
Since $(\Phi_\beta)_t>0$ in the interior, vanishing of the first
term implies $\sigma'(t)>0$ for almost every $t\in(-1,1)$.
 Thus an
equality path satisfies $\sigma(t)=t$ almost everywhere, and hence everywhere by continuity.

When $\beta=1$, the Hamilton-Jacobi inequality
\eqref{eq:Phi-beta-HJ-super} is an identity throughout the square, and
\[
 (\Phi_1)_t(t,s)=\frac{1-ts}{2\sqrt{1-t^2}},
 \qquad
 (\Phi_1)_s(t,s)=\frac{\sqrt{1-t^2}}2.
\]
Equality in the arithmetic-geometric mean step is therefore equivalent to
\begin{equation}\label{eq:equality-ODE-beta-one}
 (1-t^2)\sigma'(t)=1-t\sigma(t)
 \qquad\text{for almost every }t\in(-1,1).
\end{equation}
Every absolutely continuous solution of \eqref{eq:equality-ODE-beta-one} is
of the form
\begin{equation*}
 \sigma(t)=t+C\sqrt{1-t^2}.
\end{equation*}
If $C>0$, then $\sigma'$ is negative near $t=1$, and if $C<0$, then $\sigma'$ is
negative near $t=-1$.  Since $\sigma$ is increasing, $C=0$.  The diagonal is
therefore the unique equality path also when $\beta=1$.
\end{proof}

We now return to the functional inequality. Combining
\eqref{eq:transport-bound-beta}, \eqref{eq:c-beta-J-beta}, and
Proposition~\ref{prop:diagonal-calibration-beta}, we obtain
\begin{equation*}
 \sqrt{AB}\le c_\beta D_\beta.
\end{equation*}
Using \eqref{eq:layer-cake-beta} and the explicit formulas for $c_\beta$ and
$D_\beta$, we find
\begin{align*}
 \sqrt{V_\beta(\varphi)V_\beta(\Leg\varphi)}
 &=2\sqrt{AB}\\
 &\le2c_\beta D_\beta\\
 &=\sqrt{2\beta\pi}\,
   \frac{\Gamma\!\left(\beta+\frac12\right)}{\Gamma(\beta+1)}.
\end{align*}
This proves \eqref{eq:legendre-interpolation} and
\eqref{eq:legendre-constant}.   The value of $V_\beta(q)$ also follows directly from the change of
variables $x=\sqrt{2\beta}\,r$ and
\[
\int_{\R}(1+r^2)^{-\beta-1}\dd r
=\sqrt\pi\,
\frac{\Gamma\!\left(\beta+\frac12\right)}{\Gamma(\beta+1)}.
\]

Suppose equality holds.  Then equality holds both in
\eqref{eq:transport-bound-beta} and in
Proposition~\ref{prop:diagonal-calibration-beta}.  Hence the compactified
transport path is the diagonal.  Since the map in
\eqref{eq:beta-compactification} is injective, the original transport map is
\begin{equation}\label{eq:identity-transport-equality}
 T(u)=u.
\end{equation}
The transport identity \eqref{eq:transport-identity-beta} gives
\begin{equation*}
 y(u)=\lambda x(u)
 \quad\text{for almost every }u>0,
 \qquad
 \lambda=\frac BA>0.
\end{equation*}
Moreover, equality in \eqref{eq:transport-pointwise-beta}, together with
$T(u)=u$, gives
\begin{equation}\label{eq:equality-radius-formula}
 \lambda x(u)^2=2u
 \qquad\text{for almost every }u>0.
\end{equation}
By continuity of the sublevel radii, \eqref{eq:equality-radius-formula} holds
for every $u>0$.  Consequently
\[
 x(u)=\sqrt{\frac{2u}{\lambda}},
\]
and the definition of $x(u)$ implies
\[
 \varphi(x)=\frac{\lambda x^2}{2}.
\]
Conversely, if $\varphi(x)=cx^2/2$, then
$\Leg\varphi(y)=y^2/(2c)$ and a change of variables gives
\[
 V_\beta(\varphi)=c^{-1/2}V_\beta(q),
 \qquad
 V_\beta(\Leg\varphi)=c^{1/2}V_\beta(q).
\]
This proves the theorem, including the equality statement, for finite
$\beta$.

\subsection{The exponential endpoint}

The case $\beta=\infty$ is the classical even functional
Santal\'o inequality.  Although it has several known proofs
\cite{BallThesis1986,ArtsteinKlartagMilman2004,
	FradeliziMeyer2007,Lehec}, the same transport argument gives
a particularly short proof in the present setting. 
Let $h_\infty(u)=e^{-u}$.  The same layer-cake and monotone transport
argument gives
\[
 \sqrt{AB}
 \le\int_0^\infty
 e^{-(u+T(u))/2}\sqrt{(u+T(u))T'(u)}\dd u,
\]
where
\[
 A=\int_0^\infty e^{-u}x(u)\dd u,
 \qquad
 B=\int_0^\infty e^{-v}y(v)\dd v,
\]
and $V_\infty(\varphi)=2A$,
$V_\infty(\Leg\varphi)=2B$.  Define
\[
 \Phi_\infty(u,v)
 =\frac12\int_0^{u+v}e^{-r/2}\sqrt r\dd r.
\]
Then
\[
 (\Phi_\infty)_u=(\Phi_\infty)_v
 =\frac12e^{-(u+v)/2}\sqrt{u+v}.
\]
Consequently,
\begin{align*}
 e^{-(u+T)/2}\sqrt{(u+T)T'}
 &=2\sqrt{(\Phi_\infty)_u(\Phi_\infty)_vT'}\\
 &\le(\Phi_\infty)_u+(\Phi_\infty)_vT'\\
 &=\frac{d}{du}\Phi_\infty(u,T(u)).
\end{align*}
Integration yields
\[
 \sqrt{AB}
 \le\frac12\int_0^\infty e^{-r/2}\sqrt r\dd r
 =\sqrt{\frac\pi2},
\]
and hence
\[
 V_\infty(\varphi)V_\infty(\Leg\varphi)\le2\pi.
\]
Since $T(0)=0$, equality in the arithmetic-geometric mean step,
which gives $T'=1$ almost everywhere, implies $T(u)=u$. The equality
argument following \eqref{eq:identity-transport-equality} then applies
verbatim and yields $\varphi(x)=cx^2/2$.  This completes the proof of
Theorem~\ref{thm:legendre-interpolation}.

\subsection{A weighted geometric form}
The projective correspondence of Section~\ref{sec:functional}
also gives a geometric form of the whole interpolation family.
In dimension one, \eqref{eq:projective-polar} reads
\[
K_\varphi^\circ=K_{\Leg\varphi}.
\]
We now identify the measure on the strip for which
$V_\beta(\varphi)$ becomes the weighted area of $K_\varphi$.

For $1\le\beta<\infty$, define
\begin{equation*}
 \omega_\beta(t)
 =4(\beta+1)\beta^{\beta+1}
 \frac{(1-t)^{\beta-1}}
 {\bigl[\beta+1-(\beta-1)t\bigr]^{\beta+2}},
 \qquad -1<t<1.
\end{equation*}
and at the endpoint set
\begin{equation*}
 \omega_\infty(t)
 =\frac{4}{(1-t)^3}
 \exp\!\left(-\frac{1+t}{1-t}\right).
\end{equation*}
Notice that $\omega_1\equiv1$.  Since
\[
 \left|\det D P(x,z)\right|
 =\frac{2\sqrt2}{(1+z)^3},
\]
one obtains
\begin{equation}\label{eq:weighted-projective-volume}
 {
 \int_{K_\varphi}\omega_\beta(t)\dd X\,\dd t
 =\sqrt2\,V_\beta(\varphi)},
 \qquad 1\le\beta\le\infty.
\end{equation}
Indeed, for finite $\beta$,
\begin{align*}
 &\frac{2\sqrt2}{(1+z)^3}
 \omega_\beta\!\left(\frac{z-1}{z+1}\right)=\sqrt2\,\frac{\beta+1}{\beta}
 \left(1+\frac z\beta\right)^{-\beta-2},
\end{align*}
and integration in $z$ from $\varphi(x)$ to $\infty$ gives
\eqref{eq:weighted-projective-volume}.
For $\beta=\infty$, similarly,
\[
\frac{2\sqrt2}{(1+z)^3}
\omega_\infty\!\left(\frac{z-1}{z+1}\right)
=\sqrt2\,e^{-z}.
\]
Define
\begin{equation*}
 { \mu_\beta(K)=\int_K\omega_\beta(t)d Xd t.}
\end{equation*}
The same vertical weight is used on $K_\varphi$ and on its polar.
The weight $\omega_\beta$ is precisely the pullback density for which
$\mu_\beta(K_\varphi)=\sqrt2\,V_\beta(\varphi)$. 

\begin{corollary}[The disk remains extremal]
\label{cor:weighted-disk}
For every $1\le\beta\le\infty$ and every even $\varphi$ as in
Theorem~\ref{thm:legendre-interpolation},
\begin{equation*}
 \mu_\beta(K_\varphi)\mu_\beta(K_\varphi^\circ)
 \le\mu_\beta(B_2^2)^2.
\end{equation*}
Equality holds if and only if $K_\varphi$ is a horizontal linear image of
$B_2^2$.
\end{corollary}

\begin{proof}
Equations \eqref{eq:projective-polar} and
\eqref{eq:weighted-projective-volume}, followed by
Theorem~\ref{thm:legendre-interpolation}, give
\begin{align*}
 \mu_\beta(K_\varphi)\mu_\beta(K_\varphi^\circ)
 &=2V_\beta(\varphi)V_\beta(\Leg\varphi)\\
 &\le2V_\beta(q)^2
 =\mu_\beta(K_q)^2.
\end{align*}
For $q(x)=x^2/2$, the condition $z\ge q(x)$ in the coordinates
\eqref{eq:projective-map} is equivalent to $X^2+t^2\le1$, and hence
$K_q=B_2^2$.  More generally, $\varphi(x)=cx^2/2$ gives
$cX^2+t^2\le1$.  The equality classification in
Theorem~\ref{thm:legendre-interpolation} completes the proof.
\end{proof}

\section{Comparison with earlier functional Santal\'o inequalities}
\label{sec:earlier-functional-santalo}

 We conclude by comparing the functional inequalities proved above with
 several earlier Santal\'o-type results.  The closest one for the present
 purposes is an inequality of Bobkov for powers of convex functions and the
 ordinary Legendre transform.  Our main theorem gives its sharp form at the
 exponent $n+1$ in dimension $n$, while
 Theorem~\ref{thm:legendre-interpolation} gives the sharp constant throughout
 the corresponding one-dimensional family.
 
 We also briefly discuss Rotem's sharp inequality for negatively concave
 functions, which uses a different duality, and the positive-concavity
 polarity appearing in the functional Santal\'o inequality of
 Artstein-Avidan, Klartag and Milman.

 \subsection{Bobkov's inequality}
 
 Let $d$ denote the dimension in this subsection.  Bobkov proved the
 following inequality for the ordinary Legendre transform.
 
 \begin{theorem}[Bobkov \cite{Bobkov2010}]
 	\label{thm:bobkov-literature}
 	Let $\gamma>d$, and let
 	$\psi:\R^d\to[0,\infty]$ be even, lower semicontinuous and convex, with
 	$\psi(0)=0$. Then
 	\begin{equation}\label{eq:bobkov-unscaled}
 		\left(\int_{\R^d}(1+\psi(x))^{-\gamma}\dd x\right)
 		\left(\int_{\R^d}(1+\Leg\psi(y))^{-\gamma}\dd y\right)
 		\le
 		\left(
 		\int_{\R^d}(1+|z|^2)^{-\gamma/2}\dd z
 		\right)^2.
 	\end{equation}
 	Equivalently,
 	\begin{equation}\label{eq:bobkov-scaled}
 		\left(\int_{\R^d}
 		\left(1+\frac{\psi(x)}{\gamma}\right)^{-\gamma}\dd x\right)
 		\left(\int_{\R^d}
 		\left(1+\frac{\Leg\psi(y)}{\gamma}\right)^{-\gamma}\dd y\right)
 		\le
 		\left(
 		\int_{\R^d}
 		\left(1+\frac{|z|^2}{\gamma}\right)^{-\gamma/2}\dd z
 		\right)^2.
 	\end{equation}
 \end{theorem}
 
 This is Theorem~7.1 of \cite{Bobkov2010}.  The scaled form follows from
 \eqref{eq:bobkov-unscaled} by rescaling the convex function and the
 variables.  Bobkov derived the result from the functional
 Blaschke-Santal\'o principle of Fradelizi and Meyer
 \cite[Proposition~3]{FradeliziMeyer2007}.  He also pointed out that, in
 contrast with the log-concave functional Santal\'o inequality, this estimate
 does not in general identify an extremal function.

 \begin{remark}[Relation with Rotem's notation]
 	Rotem \cite{Rotem2014Santalo} rewrites this framework in terms of
 	negatively $\alpha$-concave functions, with
 	\[
 	\alpha=-\frac1\gamma,
 	\qquad
 	f(x)=\left(1+\frac{\psi(x)}{\gamma}\right)^{-\gamma}.
 	\]
 	The ordinary Legendre transform of the convex base gives
 	\[
 	f^*(y)
 	=
 	\left(1+\frac{\Leg\psi(y)}{\gamma}\right)^{-\gamma},
 	\]
 	so Bobkov's inequality is an inequality for the product
 	\[
 	\left(\int_{\R^d}f\right)
 	\left(\int_{\R^d}f^*\right).
 	\]
 	
 	Rotem also introduced a different duality, the $\sharp$-transform, for
 	which he obtained a sharp Santal\'o inequality.  At finite $\gamma$ this is
 	a fractional, or projective, transform of the convex base rather than the
 	ordinary Legendre transform.  The present paper keeps the latter duality.
 	The two transforms have the same log-concave limit as
 	$\gamma\to\infty$.
 \end{remark}

\subsection{The sharp case \texorpdfstring{$\gamma=n+1$}{gamma=n+1}}

Our main theorem gives the sharp form of Bobkov's inequality at the
exponent $\gamma=n+1$ in dimension $n$.  Indeed, taking $d=n$,
$\gamma=n+1$, and $\psi=\varphi$ in \eqref{eq:bobkov-unscaled} gives
\begin{equation}\label{eq:bobkov-n-plus-one-bound}
	V(\varphi)V(\Leg\varphi)
	\le
	\left(
	\int_{\R^n}(1+|z|^2)^{-(n+1)/2}\dd z
	\right)^2
	=
	\frac{\pi^{n+1}}
	{\Gamma\!\left((n+1)/2\right)^2}.
\end{equation}
This is a valid comparison bound, but it is not in general sharp.

For precisely the same functional volume product and the same Legendre
duality, Theorem~\ref{thm:functional-main} gives the optimal constant
\begin{equation}\label{eq:sharp-bobkov-n-plus-one-bound}
	V(\varphi)V(\Leg\varphi)
	\le
	\frac{(n+1)^2}{2^{n+2}}\kappa_n^2M_n^2
	=
	\left(\frac{M_n}{\Sigma_n(0)}\right)^2 V(q)^2.
\end{equation}
Thus, at the exponent $\gamma=n+1$, we determine both the sharp
constant and all equality cases in Bobkov's inequality.

For $n=1$, one has $M_1=\Sigma_1(0)=\pi/2$, and
\[
V(\varphi)V(\Leg\varphi)\le V(q)^2=\frac{\pi^2}{2},
\]
with equality exactly for
\[
\varphi(x)=\frac{c x^2}{2},
\qquad c>0.
\]
For every $n\ge2$, however,
\[
M_n>\Sigma_n(0).
\]
Consequently the self-dual quadratic $q$ is no longer extremal.
The equality functions are the two nonquadratic Legendre-dual
families described in Theorem~\ref{thm:functional-main}.

\subsection{The sharp one-dimensional family}

Let $d=1$ and $\gamma=\beta+1$, with $\beta\ge1$.
Applying \eqref{eq:bobkov-unscaled} to $\psi=\varphi/\beta$ and using
\[
\Leg\!\left(\frac{\varphi}{\beta}\right)(y)
=
\frac1\beta\,\Leg\varphi(\beta y),
\]
we obtain, after a change of variables,
\begin{align*}
	V_\beta(\varphi)V_\beta(\Leg\varphi)
	&\le
	\beta
	\left(
	\int_\R (1+z^2)^{-(\beta+1)/2}\dd z
	\right)^2\\
	&=
	\beta\pi
	\left(
	\frac{\Gamma(\beta/2)}
	{\Gamma((\beta+1)/2)}
	\right)^2.
\end{align*}

The quadratic expression on the right-hand side should not be confused
with evaluating the volume product at the self-dual quadratic
$q(x)=x^2/2$.  Indeed, the comparison density in Bobkov's bound is
\[
(1+z^2)^{-(\beta+1)/2},
\]
whereas the density associated with $q$ in our normalization is
\[
\left(1+\frac{q(x)}{\beta}\right)^{-\beta-1}
=
\left(1+\frac{x^2}{2\beta}\right)^{-\beta-1}.
\]

For the same functional volume product and the same Legendre duality,
Theorem~\ref{thm:legendre-interpolation} gives the sharp inequality
\[
V_\beta(\varphi)V_\beta(\Leg\varphi)
\le
V_\beta(q)^2
=
2\beta\pi
\left(
\frac{\Gamma\!\left(\beta+\frac12\right)}
{\Gamma(\beta+1)}
\right)^2,
\]
with equality if and only if
\[
\varphi(x)=\frac{c x^2}{2},
\qquad c>0.
\]
Thus the one-dimensional interpolation determines the optimal constant
and all equality cases for Bobkov's product when
$\gamma=\beta+1\ge2$.  At $\beta=1$, for example, Bobkov's bound is
$\pi^2$, whereas the sharp value is $\pi^2/2$.  As
$\beta\to\infty$, the family tends to the classical even functional
Santal\'o inequality.

\subsection{Positive-concavity polarity}

There is also a connection with the positive-concavity version of the
functional Santal\'o inequality from
\cite[Corollary~4.1]{ArtsteinKlartagMilman2004}.
Let $m$ be a positive integer and let
$h:\R^d\to[0,\infty)$ be such that $h^{1/m}$ is concave on its support.
Define
\begin{equation}\label{eq:positive-s-polarity}
	\Leg_m h(y)
	=
	\inf_{\{x:\,h(x)>0\}}
	\frac{\left(1-\langle x,y\rangle/m\right)_+^m}{h(x)}.
\end{equation}
If the barycenter of $h$ is at the origin, then
\begin{equation}\label{eq:positive-s-santalo}
	\left(\int_{\R^d}h\right)
	\left(\int_{\R^d}\Leg_m h\right)
	\le
	\frac{m^d\kappa_{d+m}^2}{\kappa_m^2}.
\end{equation}

This contains an interesting special case of our profile problem.
Take $d=1$, $m=n$, and, for a concave profile $r$ on $[-1,1]$, define
\[
h(\sqrt n\,t)=r(t)^n,
\qquad -1\le t\le1,
\]
and set $h=0$ outside $[-\sqrt n,\sqrt n]$.  Then
$h^{1/n}$ is concave and a direct calculation gives
\[
\Leg_n h(\sqrt n\,u)=r^\circ(u)^n,
\qquad -1\le u\le1.
\]
Moreover, the barycenter condition for $h$ is exactly
\[
\int_{-1}^1 t\,r(t)^n\dd t=0.
\]
After the change of variables in \eqref{eq:positive-s-santalo}, we obtain
\[
\left(\int_{-1}^1r(t)^n\dd t\right)
\left(\int_{-1}^1r^\circ(t)^n\dd t\right)
\le
\left(\frac{\kappa_{n+1}}{\kappa_n}\right)^2
=
\Sigma_n(0)^2.
\]

Thus the positive-concavity inequality gives precisely the ball value
for our profile problem under the additional vertical barycenter
condition.  Our theorem removes this condition.  For $n\ge2$ this
changes the sharp value: the unrestricted maximum is
$M_n^2>\Sigma_n(0)^2$.

When $n=1$, the same observation gives
\[
\left(\int_{-1}^1r(t)\dd t\right)
\left(\int_{-1}^1r^\circ(t)\dd t\right)
\le\frac{\pi^2}{4}
\]
under the barycenter condition
$\int_{-1}^1t\,r(t)\dd t=0$; our planar profile theorem shows that
this barycenter assumption is in fact unnecessary.

\appendix
\section{Scalar analysis of the calibration bound}
\label{sec:scalar-analysis}

This appendix collects the scalar analytic facts about the function
$\Sigma_n$ and the closing mismatch $\Psi_n$, which were used in the
proof of the sharp bound. These include continuity, the derivative identity,
exclusion of the endpoint parameters, and uniqueness of the
maximizing parameter.

\subsection{Continuity of \texorpdfstring{$\Sigma_n$}{Sigma n}}\label{sec:proofofLemmaSigma-cont}

\begin{lemma}\label{lem:Sigma-continuity}
	The function $\Sigma_n$ is continuous on $[0,1]$.
\end{lemma}

\begin{proof}[Proof of Lemma~\ref{lem:Sigma-continuity}]
	Write the negative part of the integral with $w=-u$.  Rationalizing both
	parts gives
	\begin{align*}
		\Sigma_n(a)
		&=2p_a\log\frac1a
		+\int_0^1
		\frac{(1-w)^n}
		{\sqrt{4p_a^2+w(1-w)^n}+2p_a}\dd w\\
		&\quad+
		\int_0^a
		\frac{2p_a-\sqrt{4p_a^2-u(1+u)^n}}{u}\dd u.
	\end{align*}
	As $a\downarrow0$, the first term tends to zero.  The last integrand equals
	\[
	\frac{(1+u)^n}
	{2p_a+\sqrt{4p_a^2-u(1+u)^n}},
	\]
	so its integral is at most
	$a(1+a)^n/(2p_a)=2p_a\to0$.  The middle integrand converges pointwise to
	$(1-w)^{n/2}/\sqrt w$ and is bounded by that integrable function.  Dominated
	convergence therefore gives the value $\Sigma_n(0)$ stated in
	\eqref{eq:Sigma-def}.
	
	For $a$ in a compact subinterval of $(0,1]$, the substitution
	\[
	w=-a+(1+a)y^2,
	\qquad 0\le y\le1,
	\]
	fixes the moving endpoint.  The factor
	$4p_a^2+w(1-w)^n$ becomes $y^2$ times a positive jointly continuous
	function of $(a,y)$, and the factor $y$ from $dw=2(1+a)y\dd y$ removes the
	square-root singularity.  The transformed integrand is therefore jointly
	continuous on a fixed compact domain.  Dominated convergence gives
	continuity on $(0,1]$.  Together with the limit at zero, this proves the
	lemma.  The differentiability needed later is established separately in
	Lemma~\ref{lem:Sigma-derivative}.
\end{proof}

\subsection{The derivative identity}\label{sec:Sigma-derivative-proof}

Here we prove Lemma \ref{lem:Sigma-derivative}, namely that   on $(0,1)$  the function $\Sigma_n$ is continuously
differentiable  and  $
	\Sigma_n'(a)=2p_a'\Psi_n(a)$.

\begin{proof}[Proof of Lemma \ref{lem:Sigma-derivative}]
	Write
	\[
	G(p,w)
	=
	\frac{(1-w)^n}
	{\sqrt{4p^2+w(1-w)^n}+2p}.
	\]
	Then the definition of $\Sigma_n$ in \eqref{eq:Sigma-def} reads
	\[
	\Sigma_n(a)
	=
	2p_a\log\frac1a
	+
	\int_{-a}^{1}G(p_a,w)\dd w.
	\]
	For $-a<w\le1$, differentiation with respect to $p$ gives
	\begin{align*}
		\partial_pG(p_a,w)
		&=
		-\frac{2(1-w)^n}
		{\sqrt{4p_a^2+w(1-w)^n}
			\bigl(\sqrt{4p_a^2+w(1-w)^n}+2p_a\bigr)}\\
		&=
		\frac{\displaystyle
			\frac{4p_a}{\sqrt{4p_a^2+w(1-w)^n}}-2}{w},
		\qquad w\ne0.
	\end{align*}
	The first expression is regular at $w=0$. At $w=-a$,
	the derivative has an integrable singularity of order
	$(w+a)^{-1/2}$.
	
	To justify differentiation of the integral with its moving
	endpoint, set
	\[
	w=-a+(1+a)y^2,
	\qquad 0\le y\le1.
	\]
	Using $4p_a^2=a(1+a)^n$, we obtain
	\[
	4p_a^2+w(1-w)^n
	=
	(1+a)^n y^2
	\left[
	a\sum_{j=0}^{n-1}(1-y^2)^j
	+(1+a)(1-y^2)^n
	\right].
	\]
	The expression in brackets is smooth and positive for
	$a>0$ and $0\le y\le1$. Consequently, in the representation
	\[
	\int_{-a}^{1}G(p_a,w)\dd w
	=
	2(1+a)\int_0^1
	yG\bigl(p_a,-a+(1+a)y^2\bigr)\dd y,
	\]
	the transformed integrand and its $a$-derivative extend
	continuously to $y=0$, locally uniformly in $a$. This proves
	that the integral is $C^1$ on $(0,1)$.
	
	Applying the ordinary Leibniz rule after truncating to
	$y\ge\varepsilon$, and then letting $\varepsilon\downarrow0$,
	gives
	\[
	\frac{\dd}{\dd a}
	\int_{-a}^{1}G(p_a,w)\dd w
	=
	G(p_a,-a)
	+
	p_a'\int_{-a}^{1}\partial_pG(p_a,w)\dd w.
	\]
	The passage to the limit is justified by continuity of $G$
	at the lower endpoint and the locally uniform bound
	$|\partial_pG(p_a,w)|\le C(w+a)^{-1/2}$.
	
	At the lower endpoint,
	\[
	G(p_a,-a)
	=
	\frac{(1+a)^n}{2p_a}
	=
	\frac{2p_a}{a}.
	\]
	Thus the moving-endpoint contribution cancels the term
	$-2p_a/a$ from differentiating $2p_a\log(1/a)$. We obtain
	\begin{align*}
		\Sigma_n'(a)
		&=
		2p_a'\log\frac1a
		-
		\frac{2p_a}{a}
		+
		G(p_a,-a)
		+
		p_a'\int_{-a}^{1}\partial_pG(p_a,w)\dd w\\
		&=
		2p_a'
		\left[
		\log\frac1a
		+
		\int_{-a}^{1}
		\frac{\displaystyle
			\frac{2p_a}{\sqrt{4p_a^2+w(1-w)^n}}-1}
		{w}\dd w
		\right]\\
		&=
		2p_a'\Psi_n(a),
	\end{align*}
	where the last equality is \eqref{eq:Psi-def}.
	Finally,
	\[
	p_a'
	=
	\frac{p_a}{2}
	\left(\frac1a+\frac{n}{1+a}\right)>0,
	\]
	which proves the last assertion.
\end{proof}

\subsection{Endpoint behavior}
\label{sec:endpoint-exclusion}

\begin{lemma}[Endpoint exclusion]\label{lem:endpoint-exclusion}
	For every $n\ge2$, neither $a=0$ nor $a=1$ is a
	maximizer of $\Sigma_n$.
\end{lemma}

\begin{proof}
	We exclude the two endpoints separately.

		To exclude $a=1$, observe that the fixed-domain substitution
	in the proof of Lemma~\ref{lem:Sigma-derivative} remains
	regular up to $a=1$. Thus $\Sigma_n'$ and $\Psi_n$ extend
	continuously to this endpoint from the left, and
	\[
	\Sigma_n'(1-)=2p_1'\Psi_n(1).
	\]
	Since $4p_1^2=2^n$, the rationalized formula in
	\eqref{eq:Psi-def} gives
	\[
	\Psi_n(1)
	=
	-\int_{-1}^{1}
	\frac{(1-w)^n}
	{\sqrt{2^n+w(1-w)^n}
		\bigl(\sqrt{2^n+w(1-w)^n}+2^{n/2}\bigr)}
	\dd w<0.
	\]
	This integral converges: its only singularity is of order
	$(w+1)^{-1/2}$ at $w=-1$. Since $p_1'>0$, we have
	$\Sigma_n'(1-)<0$. Hence values just to the left of $1$
	are larger than $\Sigma_n(1)$, and $1$ is not a maximizer.
	
	To exclude $a=0$, we use the already proved bound
	\[
	\sqrt{A(r)B(r)}\le M_n,
	\]
	which follows from Lemma~\ref{lem:transport} and
	Proposition~\ref{prop:two-layer-calibration}.
	It suffices to find an admissible profile whose product is
	strictly larger than $\Sigma_n(0)^2$.
	
	Begin with the cone profile $r(t)=1+t$. Its polar profile is
	$r^\circ(s)=(1-s)/2$, so
	\[
	A(r)=\frac{2^{n+1}}{n+1},
	\qquad
	B(r)=\frac{2}{n+1},
	\qquad
	\sqrt{A(r)B(r)}=\frac{2^{(n+2)/2}}{n+1}.
	\]
	When $n=2$, this last quantity equals
	$\Sigma_2(0)=4/3$. The quantity
	\[
	\Sigma_n(0)=\int_{-1}^{1}(1-t^2)^{n/2}\dd t
	\]
	strictly decreases as $n$ increases. In contrast, the
	successive ratio of the cone values is
	\[
	\frac{2^{(n+3)/2}/(n+2)}{2^{(n+2)/2}/(n+1)}
	=
	\sqrt2\,\frac{n+1}{n+2}
	\ge\frac{3\sqrt2}{4}>1,
	\qquad n\ge2.
	\]
	Therefore
	\[
	M_n\ge\frac{2^{(n+2)/2}}{n+1}>\Sigma_n(0)
	\qquad(n\ge3).
	\]
	
	For $n=2$, the cone has exactly the ball value, so we perturb
	it by setting
	\[
	r_\varepsilon(t)=
	\begin{cases}
		1+t,&-1\le t\le0,\\
		1+(1-\varepsilon)t,&0\le t\le1,
	\end{cases}
	\qquad 0<\varepsilon<1.
	\]
	This is a nonnegative concave profile, since its slope
	decreases from $1$ to $1-\varepsilon$ at $t=0$.
	On each linear piece, the quotient
	$(1-st)/r_\varepsilon(t)$ is monotone in $t$. Its infimum is therefore the smaller of its value at
	$t=0$ and its limit as $t\uparrow1$. Thus
	\[
	r_\varepsilon^\circ(s)
	=
	\min\left\{1,\frac{1-s}{2-\varepsilon}\right\}
	=
	\begin{cases}
		1,&-1\le s\le-1+\varepsilon,\\
		\dfrac{1-s}{2-\varepsilon},
		&-1+\varepsilon\le s\le1.
	\end{cases}
	\]
	Computing the two integrals with $n=2$ gives
	\[
	A(r_\varepsilon)=\frac{8-5\varepsilon+\varepsilon^2}{3},
	\qquad
	B(r_\varepsilon)=\frac{2+2\varepsilon}{3}.
	\]
	Consequently,
	\[
	A(r_\varepsilon)B(r_\varepsilon)
	=
	\frac{16}{9}
	+
	\frac{2\varepsilon(1-\varepsilon)(3-\varepsilon)}{9}
	>
	\frac{16}{9}
	=
	\Sigma_2(0)^2.
	\]
	It follows that $M_2>\Sigma_2(0)$. This excludes $a=0$
	for every $n\ge2$.
	
\end{proof}  
  
\subsection{Uniqueness of the maximizing parameter}
\label{sec:unique-maximizer}

In this section we prove that the scalar function $\Sigma_n$ has a unique
maximizer.  The key point is that the normalized closing mismatch
$\Psi_n(a)/\sqrt a$ is strictly decreasing.  The comparison is with
the planar case, where the closing mismatch can be computed explicitly.

\begin{lemma}[The planar closing mismatch]
	\label{lem:planar-closing-mismatch}
	For $0<a<1$,
	\[
	\Psi_1(a)=-\log(1+2a).
	\]
	Consequently, $\Psi_1(a)<0$ and $\Sigma_1$ is strictly decreasing
	on $(0,1)$.
\end{lemma}

\begin{proof}
	Pairing the negative and positive parts in the definition
	\eqref{eq:Psi-def}, and then using
	\[
	w=-a+(1+2a)\sin^2\theta,
	\qquad
	z=\sqrt{\frac{1+a}{a}}\tan\theta,
	\]
	gives
	\[
	\Psi_1(a)
	=
	\left.
	\log\left|\frac{z-1}{z+1}\right|
	\right|_{0}^{(1+a)/a}
	=
	-\log(1+2a).
	\]
	Since $p_a'>0$, the derivative identity
	\eqref{eq:Sigma-derivative} implies
	\[
	\Sigma_1'(a)=2p_a'\Psi_1(a)<0.
	\]
	Thus $\Sigma_1$ is strictly decreasing on $(0,1)$.
\end{proof}

\begin{lemma}[Monotonicity of the normalized closing mismatch]
	\label{lem:scaled-Psi-monotonicity}
	For every integer $n\ge1$, the function $\Psi_n$ defined in
	\eqref{eq:Psi-def} belongs to $C^1((0,1))$, and
	\begin{equation}\label{eq:app-scaled-Psi-derivative}
	\frac{d}{da}\left(\frac{\Psi_n(a)}{\sqrt a}\right)
	\le
	\frac{d}{da}\left(\frac{\Psi_1(a)}{\sqrt a}\right)
	=\frac{\displaystyle
		\log(1+2a)-\frac{4a}{1+2a}}{2a^{3/2}}
	<0,
	\qquad 0<a<1.
	\end{equation}
	The first inequality is strict when $n>1$.
\end{lemma}

\begin{proof}
	We first put the integral defining $\Psi_n$ on a fixed interval.
	In \eqref{eq:Psi-def}, substitute
	\[
	z=\frac{1-w}{1+a},
	\qquad d=1-(1+a)z,
	\qquad
	Q_n(a,z)=a+z^n d.
	\]
	Since $4p_a^2=a(1+a)^n$, we have
	\[
	4p_a^2+w(1-w)^n=(1+a)^nQ_n(a,z),
	\qquad
	\frac{2p_a}{\sqrt{4p_a^2+w(1-w)^n}}
	=\frac{\sqrt a}{\sqrt{Q_n(a,z)}}.
	\]
	Using $Q_n(a,z)-a=z^n d$ to rationalize the integrand gives
	\begin{equation}\label{eq:app-Psi-fixed-domain}
	\Psi_n(a)
	=\log\frac1a
	-(1+a)\int_0^1
	\frac{z^n}
	{\sqrt{Q_n(a,z)}\bigl(\sqrt a+\sqrt{Q_n(a,z)}\bigr)}\dd z.
	\end{equation}
	In particular, the apparent singularity at $d=0$ has disappeared.
	
	For the remainder of the proof, write
	\[
	F_n(a)=\frac{\Psi_n(a)}{\sqrt a},
	\qquad
	G_n(a,z)=
	\frac{(1+a)z^n}
	{\sqrt a\,\sqrt{Q_n(a,z)}
		\bigl(\sqrt a+\sqrt{Q_n(a,z)}\bigr)}.
	\]
	Then
	\begin{equation}\label{eq:app-scaled-Psi-representation}
	F_n(a)=\frac1{\sqrt a}\log\frac1a-\int_0^1G_n(a,z)\dd z.
	\end{equation}
	Differentiation under this integral is legitimate.  Indeed,
	\begin{equation}\label{eq:app-Q-factorization}
	Q_n(a,z)
	=(1-z)\left(a\sum_{j=0}^{n}z^j+z^n\right),
	\qquad
	\partial_a Q_n(a,z)=1-z^{n+1}.
	\end{equation}
	For $a$ in a compact subinterval of $(0,1)$, the factor in parentheses
	is bounded above and bounded away from zero, uniformly for $0\le z\le1$,
	and $\partial_a Q_n(a,z)=O(1-z)$.  Consequently
	\[
	|G_n(a,z)|+|\partial_aG_n(a,z)|
	\le C(1-z)^{-1/2},
	\qquad 0<z<1,
	\]
	with $C$ independent of $a$ in that compact subinterval.  This integrable
	majorant proves that $F_n\in C^1((0,1))$ and permits differentiation of
	\eqref{eq:app-scaled-Psi-representation}.
	
	To compare the derivatives for different $n$, introduce, for
	$0<z<1$ and $0\le s\le z$, the auxiliary function
	\begin{equation}\label{eq:app-G-auxiliary}
	G(a,z,s)
	=\frac{1+a}{2}\int_0^s
	\frac{\dd u}{(a+u d)^{3/2}}
	=\frac{1+a}{d}
	\left(\frac1{\sqrt a}-\frac1{\sqrt{a+s d}}\right),
	\end{equation}
	where the integral defines the expression also when $d=0$.
	The denominators are positive: $a>0$ and
	\[
	a+z d=(1-z)\bigl(a+(1+a)z\bigr)>0,
	\]
	so $a+u d>0$ for $0\le u\le z$.
	Because $z^n\le z$, rationalization shows that
	\[
	G_n(a,z)=G(a,z,z^n).
	\]
	Differentiating first with respect to $s$ gives
	\[
	G_s(a,z,s)=\frac{1+a}{2(a+s d)^{3/2}},
	\]
	and differentiating with respect to $a$, with $z,s$ fixed, yields
	\begin{equation}\label{eq:app-G-mixed-derivative}
	G_{as}(a,z,s)
	=\frac{s\bigl((1+a)z+2\bigr)-(a+3)}
	{4\bigl(a+s(1-(1+a)z)\bigr)^{5/2}}.
	\end{equation}
	The numerator is strictly negative, since
	\[
	\begin{aligned}
	s\bigl((1+a)z+2\bigr)-(a+3)
	&\le z\bigl((1+a)z+2\bigr)-(a+3)\\
	&=(z-1)\bigl((1+a)z+a+3\bigr)<0.
	\end{aligned}
	\]
	Thus $G_a(a,z,s)$ is strictly decreasing in $s$.  It follows that
	\[
	\partial_aG_n(a,z)
	=G_a(a,z,z^n)
	\ge G_a(a,z,z)
	=\partial_aG_1(a,z),
	\]
	with strict inequality for $n>1$.  Differentiating
	\eqref{eq:app-scaled-Psi-representation} and integrating this comparison
	therefore gives
	\[
	F_n'(a)\le F_1'(a),
	\]
	again strictly when $n>1$.
	
	Finally,  Lemma~\ref{lem:planar-closing-mismatch} gives
	$\Psi_1(a)=-\log(1+2a)$, and hence
	\[
	F_1'(a)
	=\frac{1}{2a^{3/2}}
	\left(\log(1+2a)-\frac{4a}{1+2a}\right).
	\]
	Strict convexity of $u\mapsto1/u$ gives
	\[
	\log(1+2a)
	=\int_1^{1+2a}\frac{\dd u}{u}
	<a\left(1+\frac1{1+2a}\right)
	=\frac{2a(1+a)}{1+2a}
	\le\frac{4a}{1+2a},
	\qquad 0<a\le1.
	\]
	Thus $F_1'(a)<0$, completing the proof.
\end{proof}

 \begin{proof}[Proof of Proposition~\ref{prop:scalar-maximizer}]
 	For $n=1$, Lemma~\ref{lem:planar-closing-mismatch} and
 	\eqref{eq:Sigma-derivative} show that
 	$\Sigma_1'(a)<0$ for every $0<a<1$.  Together with continuity
 	at $a=0$, this proves the assertion for $n=1$.
 	
Let $n\ge2$. By Lemma~\ref{lem:Sigma-continuity} and
 	Lemma~\ref{lem:endpoint-exclusion}, $\Sigma_n$ has a maximizer
 	$a_n\in(0,1)$. Lemma~\ref{lem:Sigma-derivative} gives
 	\[
 	0=\Sigma_n'(a_n)=2p_{a_n}'\Psi_n(a_n),
 	\]
 	and hence $\Psi_n(a_n)=0$.
 	
 	By Lemma~\ref{lem:scaled-Psi-monotonicity},
 	$\Psi_n(a)/\sqrt a$ is strictly decreasing.  Thus this zero is
 	unique and
 	\[
 	\Psi_n(a)>0\quad(0<a<a_n),
 	\qquad
 	\Psi_n(a)<0\quad(a_n<a<1).
 	\]
 	Since $p_a'>0$, the identity
 	\[
 	\Sigma_n'(a)=2p_a'\Psi_n(a)
 	\]
 	shows that $\Sigma_n$ is strictly increasing on $[0,a_n]$
 	and strictly decreasing on $[a_n,1]$.  This proves
 	Proposition~\ref{prop:scalar-maximizer}.
 \end{proof}

\begin{remark}
	The maximizing critical point is nondegenerate. Indeed,
	Lemma~\ref{lem:scaled-Psi-monotonicity} gives
	\[
	\Psi_n'(a_n)<0,
	\]
	and therefore
	\[
	\Sigma_n''(a_n)
	=2p_{a_n}'\Psi_n'(a_n)<0.
	\]
\end{remark}

\section*{Acknowledgments}

Supported by the Israel Science Foundation (grant 1626/25)

\paragraph{AI-assisted tools.}
The author used ChatGPT (OpenAI) as an AI-assisted tool during the
development and preparation of this work, including for exploring
possible approaches, checking computations and arguments, and improving
the exposition. All mathematical statements and proofs were independently
verified by the author, who takes full responsibility for the content of
the manuscript.

\end{document}

%% file: sharp_santalo_optimal_profiles_tikz_fragment_2026-09-01.tex
\begingroup
\definecolor{curveone}{RGB}{31,119,180}
\definecolor{curvetwo}{RGB}{230,126,34}
\definecolor{curvethree}{RGB}{44,160,44}
\definecolor{curvefour}{RGB}{214,39,40}
\definecolor{curvefive}{RGB}{148,103,189}

\def\ProfileDataTwo{%
    (0.000000000000,-1.000000000000) (0.042712745298,-0.953912809426) (0.064048241728,-0.930869214139) (0.085358798641,-0.907825618852)
    (0.106636283437,-0.884782023565) (0.117260121298,-0.873260225921) (0.127872708756,-0.861738428278) (0.138473074538,-0.850216630634)
    (0.149060259614,-0.838694832991) (0.159633317826,-0.827173035347) (0.170191316440,-0.815651237704) (0.180733336648,-0.804129440060)
    (0.191258474001,-0.792607642416) (0.201765838782,-0.781085844773) (0.212254556322,-0.769564047129) (0.222723767259,-0.758042249486)
    (0.233172627740,-0.746520451842) (0.243600309571,-0.734998654199) (0.254006000322,-0.723476856555) (0.264388903370,-0.711955058912)
    (0.274748237917,-0.700433261268) (0.285083238941,-0.688911463625) (0.295393157127,-0.677389665981) (0.305677258748,-0.665867868338)
    (0.315934825514,-0.654346070694) (0.326165154387,-0.642824273051) (0.336367557365,-0.631302475407) (0.346541361241,-0.619780677764)
    (0.356685907327,-0.608258880120) (0.366800551167,-0.596737082476) (0.376884662215,-0.585215284833) (0.386937623503,-0.573693487189)
    (0.396958831284,-0.562171689546) (0.406947694669,-0.550649891902) (0.416903635232,-0.539128094259) (0.426826086626,-0.527606296615)
    (0.436714494164,-0.516084498972) (0.446568314411,-0.504562701328) (0.456387014751,-0.493040903685) (0.466170072958,-0.481519106041)
    (0.475916976756,-0.469997308398) (0.485627223376,-0.458475510754) (0.495300319106,-0.446953713111) (0.504935778840,-0.435431915467)
    (0.514533125626,-0.423910117824) (0.524091890211,-0.412388320180) (0.533611610579,-0.400866522536) (0.543091831501,-0.389344724893)
    (0.547816989151,-0.383583826071) (0.552532104068,-0.377822927249) (0.557237121053,-0.372062028428) (0.561931985244,-0.366301129606)
    (0.566616642101,-0.360540230784) (0.571291037397,-0.354779331962) (0.575955117198,-0.349018433141) (0.580608827851,-0.343257534319)
    (0.585252115971,-0.337496635497) (0.589884928425,-0.331735736675) (0.594507212318,-0.325974837854) (0.599118914978,-0.320213939032)
    (0.603719983945,-0.314453040210) (0.608310366954,-0.308692141388) (0.612890011920,-0.302931242566) (0.617458866929,-0.297170343745)
    (0.622016880215,-0.291409444923) (0.626564000157,-0.285648546101) (0.631100175253,-0.279887647279) (0.635625354116,-0.274126748458)
    (0.640139485454,-0.268365849636) (0.644642518058,-0.262604950814) (0.649134400788,-0.256844051992) (0.653615082556,-0.251083153171)
    (0.658084512316,-0.245322254349) (0.662542639048,-0.239561355527) (0.666989411744,-0.233800456705) (0.671424779391,-0.228039557884)
    (0.675848690963,-0.222278659062) (0.680261095401,-0.216517760240) (0.684661941602,-0.210756861418) (0.689051178403,-0.204995962596)
    (0.693428754567,-0.199235063775) (0.697794618771,-0.193474164953) (0.702148719589,-0.187713266131) (0.706491005476,-0.181952367309)
    (0.710821424758,-0.176191468488) (0.715139925615,-0.170430569666) (0.719446456067,-0.164669670844) (0.723740963958,-0.158908772022)
    (0.728023396943,-0.153147873201) (0.732293702472,-0.147386974379) (0.736551827777,-0.141626075557) (0.740797719854,-0.135865176735)
    (0.745031325451,-0.130104277914) (0.749252591051,-0.124343379092) (0.753461462857,-0.118582480270) (0.757657886777,-0.112821581448)
    (0.761841808410,-0.107060682626) (0.766013173026,-0.101299783805) (0.770171925556,-0.095538884983) (0.774318010571,-0.089777986161)
    (0.778451372271,-0.084017087339) (0.782571954463,-0.078256188518) (0.786679700550,-0.072495289696) (0.790774553512,-0.066734390874)
    (0.794856455888,-0.060973492052) (0.798925349764,-0.055212593231) (0.802981176749,-0.049451694409) (0.807023877965,-0.043690795587)
    (0.811053394024,-0.037929896765) (0.815069665014,-0.032168997944) (0.819072630477,-0.026408099122) (0.823062229395,-0.020647200300)
    (0.827038400170,-0.014886301478) (0.831001080605,-0.009125402656) (0.834950207886,-0.003364503835) (0.838885718560,0.002396394987)
    (0.842807548520,0.008157293809) (0.846715632980,0.013918192631) (0.850609906460,0.019679091452) (0.854490302761,0.025439990274)
    (0.858356754948,0.031200889096) (0.862209195324,0.036961787918) (0.866047555415,0.042722686739) (0.869871765940,0.048483585561)
    (0.873681756796,0.054244484383) (0.877477457031,0.060005383205) (0.881258794819,0.065766282026) (0.885025697444,0.071527180848)
    (0.888778091265,0.077288079670) (0.892515901700,0.083048978492) (0.896239053196,0.088809877314) (0.899947469205,0.094570776135)
    (0.903641072159,0.100331674957) (0.907319783437,0.106092573779) (0.910983523345,0.111853472601) (0.914632211084,0.117614371422)
    (0.918265764721,0.123375270244) (0.921884101158,0.129136169066) (0.925487136107,0.134897067888) (0.929074784053,0.140657966709)
    (0.932646958226,0.146418865531) (0.936203570567,0.152179764353) (0.939744531695,0.157940663175) (0.943269750871,0.163701561996)
    (0.946779135968,0.169462460818) (0.950272593426,0.175223359640) (0.953750028225,0.180984258462) (0.957211343840,0.186745157284)
    (0.960656442203,0.192506056105) (0.964085223665,0.198266954927) (0.967497586954,0.204027853749) (0.970893429130,0.209788752571)
    (0.974272645546,0.215549651392) (0.977635129797,0.221310550214) (0.980980773679,0.227071449036) (0.984309467137,0.232832347858)
    (0.987621098220,0.238593246679) (0.990915553027,0.244354145501) (0.994192715655,0.250115044323) (0.997452468149,0.255875943145)
    (1.000694690441,0.261636841966) (1.003919260295,0.267397740788) (1.007126053249,0.273158639610) (1.010314942553,0.278919538432)
    (1.013485799104,0.284680437254) (1.016638491385,0.290441336075) (1.019772885393,0.296202234897) (1.022888844572,0.301963133719)
    (1.025986229740,0.307724032541) (1.029064899015,0.313484931362) (1.032124707737,0.319245830184) (1.035165508389,0.325006729006)
    (1.038187150515,0.330767627828) (1.041189480629,0.336528526649) (1.044172342133,0.342289425471) (1.047135575223,0.348050324293)
    (1.050079016788,0.353811223115) (1.053002500319,0.359572121936) (1.055905855801,0.365333020758) (1.058788909605,0.371093919580)
    (1.061651484382,0.376854818402) (1.064493398942,0.382615717224) (1.067314468139,0.388376616045) (1.070114502743,0.394137514867)
    (1.072893309310,0.399898413689) (1.075650690050,0.405659312511) (1.078386442681,0.411420211332) (1.081100360288,0.417181110154)
    (1.083792231169,0.422942008976) (1.086461838671,0.428702907798) (1.089108961031,0.434463806619) (1.091733371197,0.440224705441)
    (1.094334836651,0.445985604263) (1.096913119217,0.451746503085) (1.099467974865,0.457507401906) (1.101999153504,0.463268300728)
    (1.104506398768,0.469029199550) (1.106989447787,0.474790098372) (1.108743097751,0.478893300728) (1.110484236752,0.482996503085)
    (1.112212763377,0.487099705441) (1.113928574321,0.491202907798) (1.115631564334,0.495306110154) (1.117321626168,0.499409312511)
    (1.118998650523,0.503512514867) (1.120662525995,0.507615717224) (1.122313139010,0.511718919580) (1.123950373771,0.515822121936)
    (1.125574112192,0.519925324293) (1.127184233834,0.524028526649) (1.128780615838,0.528131729006) (1.130363132856,0.532234931362)
    (1.131931656977,0.536338133719) (1.133486057653,0.540441336075) (1.135026201621,0.544544538432) (1.136551952822,0.548647740788)
    (1.138063172320,0.552750943145) (1.139559718207,0.556854145501) (1.141041445521,0.560957347858) (1.142508206144,0.565060550214)
    (1.143959848709,0.569163752571) (1.145396218495,0.573266954927) (1.146817157322,0.577370157284) (1.148222503436,0.581473359640)
    (1.149612091399,0.585576561996) (1.150985751962,0.589679764353) (1.152343311946,0.593782966709) (1.153684594105,0.597886169066)
    (1.155009416989,0.601989371422) (1.156317594805,0.606092573779) (1.157608937264,0.610195776135) (1.158883249425,0.614298978492)
    (1.160140331530,0.618402180848) (1.161379978837,0.622505383205) (1.162601981434,0.626608585561) (1.163806124058,0.630711787918)
    (1.164992185895,0.634814990274) (1.166159940371,0.638918192631) (1.167309154941,0.643021394987) (1.168439590858,0.647124597344)
    (1.169551002936,0.651227799700) (1.170643139297,0.655331002056) (1.171715741111,0.659434204413) (1.172768542316,0.663537406769)
    (1.173801269327,0.667640609126) (1.174813640726,0.671743811482) (1.175805366944,0.675847013839) (1.176776149913,0.679950216195)
    (1.177725682708,0.684053418552) (1.178653649165,0.688156620908) (1.179559723481,0.692259823265) (1.180443569784,0.696363025621)
    (1.181304841688,0.700466227978) (1.182143181807,0.704569430334) (1.182958221260,0.708672632691) (1.183749579124,0.712775835047)
    (1.184516861872,0.716879037404) (1.185259662762,0.720982239760) (1.185977561196,0.725085442116) (1.186670122032,0.729188644473)
    (1.187336894857,0.733291846829) (1.187977413202,0.737395049186) (1.188591193718,0.741498251542) (1.189177735280,0.745601453899)
    (1.189736518039,0.749704656255) (1.190267002405,0.753807858612) (1.190768627953,0.757911060968) (1.191240812252,0.762014263325)
    (1.191682949605,0.766117465681) (1.192094409696,0.770220668038) (1.192474536129,0.774323870394) (1.192822644853,0.778427072751)
    (1.193138022464,0.782530275107) (1.193419924363,0.786633477464) (1.193667572763,0.790736679820) (1.193880154529,0.794839882176)
    (1.194056818828,0.798943084533) (1.194196674572,0.803046286889) (1.194298787629,0.807149489246) (1.194362177783,0.811252691602)
    (1.194379032037,0.813304292781) (1.194385815395,0.815355893959) (1.194382390833,0.817407495137) (1.194368617759,0.819459096315)
    (1.194344351886,0.821510697494) (1.194309445084,0.823562298672) (1.194263745235,0.825613899850) (1.194207096079,0.827665501028)
    (1.194139337044,0.829717102206) (1.194060303081,0.831768703385) (1.193969824478,0.833820304563) (1.193867726669,0.835871905741)
    (1.193753830034,0.837923506919) (1.193627949685,0.839975108098) (1.193489895245,0.842026709276) (1.193339470606,0.844078310454)
    (1.193176473682,0.846129911632) (1.193000696142,0.848181512811) (1.192811923133,0.850233113989) (1.192609932979,0.852284715167)
    (1.192394496868,0.854336316345) (1.192165378518,0.856387917524) (1.191922333822,0.858439518702) (1.191665110470,0.860491119880)
    (1.191393447549,0.862542721058) (1.191107075116,0.864594322236) (1.190805713744,0.866645923415) (1.190489074033,0.868697524593)
    (1.190156856096,0.870749125771) (1.189808749002,0.872800726949) (1.189444430186,0.874852328128) (1.189063564813,0.876903929306)
    (1.188665805096,0.878955530484) (1.188250789570,0.881007131662) (1.187818142301,0.883058732841) (1.187367472048,0.885110334019)
    (1.186898371345,0.887161935197) (1.186410415527,0.889213536375) (1.185903161664,0.891265137554) (1.185376147414,0.893316738732)
    (1.184828889783,0.895368339910) (1.184260883771,0.897419941088) (1.183671600910,0.899471542266) (1.183060487667,0.901523143445)
    (1.182426963704,0.903574744623) (1.181770419975,0.905626345801) (1.181090216651,0.907677946979) (1.180385680832,0.909729548158)
    (1.179656104050,0.911781149336) (1.178900739500,0.913832750514) (1.178118799006,0.915884351692) (1.177309449649,0.917935952871)
    (1.176471810034,0.919987554049) (1.175604946147,0.922039155227) (1.174707866729,0.924090756405) (1.173779518111,0.926142357584)
    (1.172818778428,0.928193958762) (1.171824451110,0.930245559940) (1.170795257543,0.932297161118) (1.169729828758,0.934348762296)
    (1.168626695998,0.936400363475) (1.168060501738,0.937426164064) (1.167484279957,0.938451964653) (1.166897813947,0.939477765242)
    (1.166300878453,0.940503565831) (1.165693239177,0.941529366420) (1.165074652257,0.942555167009) (1.164444863689,0.943580967599)
    (1.163803608717,0.944606768188) (1.163150611169,0.945632568777) (1.162485582735,0.946658369366) (1.161808222193,0.947684169955)
    (1.161118214559,0.948709970544) (1.160415230173,0.949735771133) (1.159698923699,0.950761571722) (1.158968933032,0.951787372311)
    (1.158224878110,0.952813172901) (1.157466359609,0.953838973490) (1.156692957515,0.954864774079) (1.155904229555,0.955890574668)
    (1.155099709464,0.956916375257) (1.154278905083,0.957942175846) (1.153441296242,0.958967976435) (1.152586332428,0.959993777024)
    (1.151713430179,0.961019577614) (1.150821970190,0.962045378203) (1.149911294070,0.963071178792) (1.148980700719,0.964096979381)
    (1.148029442241,0.965122779970) (1.147056719343,0.966148580559) (1.146061676123,0.967174381148) (1.145043394153,0.968200181737)
    (1.144000885727,0.969225982326) (1.142933086126,0.970251782916) (1.141838844732,0.971277583505) (1.140716914744,0.972303384094)
    (1.139565941243,0.973329184683) (1.138384447257,0.974354985272) (1.137170817382,0.975380785861) (1.135923278442,0.976406586450)
    (1.134639876457,0.977432387039) (1.133318449060,0.978458187629) (1.132642735046,0.978971087923) (1.131956592151,0.979483988218)
    (1.131259675433,0.979996888512) (1.130551619275,0.980509788807) (1.129832035560,0.981022689101) (1.129100511641,0.981535589396)
    (1.128356608057,0.982048489690) (1.127599855972,0.982561389985) (1.126829754283,0.983074290280) (1.126045766344,0.983587190574)
    (1.125247316250,0.984100090869) (1.124433784589,0.984612991163) (1.123604503591,0.985125891458) (1.122758751533,0.985638791752)
    (1.121895746285,0.986151692047) (1.121014637803,0.986664592341) (1.120114499371,0.987177492636) (1.119194317310,0.987690392931)
    (1.118252978814,0.988203293225) (1.117289257476,0.988716193520) (1.116301795943,0.989229093814) (1.115289084953,0.989741994109)
    (1.114249437793,0.990254894403) (1.113180958867,0.990767794698) (1.112081504619,0.991280694992) (1.110948634372,0.991793595287)
    (1.109779547690,0.992306495582) (1.109180430229,0.992562945729) (1.108571003392,0.992819395876) (1.107950775737,0.993075846023)
    (1.107319213120,0.993332296171) (1.106675733178,0.993588746318) (1.106019698856,0.993845196465) (1.105350410755,0.994101646613)
    (1.104667098026,0.994358096760) (1.103968907457,0.994614546907) (1.103254890281,0.994870997054) (1.102523986105,0.995127447202)
    (1.101775003134,0.995383897349) (1.101006593573,0.995640347496) (1.100217222698,0.995896797644) (1.099405129432,0.996153247791)
    (1.098568275375,0.996409697938) (1.097704277818,0.996666148085) (1.096810320062,0.996922598233) (1.095883028743,0.997179048380)
    (1.094918301851,0.997435498527) (1.094420360076,0.997563723601) (1.093911060476,0.997691948674) (1.093389554981,0.997820173748)
    (1.092854877853,0.997948398822) (1.092305920870,0.998076623895) (1.091741401233,0.998204848969) (1.091159819394,0.998333074043)
    (1.090559402590,0.998461299116) (1.089938027601,0.998589524190) (1.089293112456,0.998717749264) (1.088621460125,0.998845974337)
    (1.087919024920,0.998974199411) (1.087180548445,0.999102424485) (1.086795680455,0.999166537021) (1.086398961795,0.999230649558)
    (1.085989055642,0.999294762095) (1.085564336283,0.999358874632) (1.085122790672,0.999422987169) (1.084661871316,0.999487099705)
    (1.084178266989,0.999551212242) (1.083667526229,0.999615324779) (1.083123396470,0.999679437316) (1.082836075456,0.999711493584)
    (1.082536556715,0.999743549853) (1.082222754834,0.999775606121) (1.081891866214,0.999807662390) (1.081539952809,0.999839718658)
    (1.081161143847,0.999871774926) (1.080958903730,0.999887803061) (1.080745902779,0.999903831195) (1.080519660501,0.999919859329)
    (1.080276478967,0.999935887463) (1.080010351125,0.999951915597) (1.079865558187,0.999959929664) (1.079710095693,0.999967943732)
    (1.079540179671,0.999975957799) (1.079348766718,0.999983971866) (1.079240595363,0.999987978899) (1.079118882610,0.999991985933)
    (1.078972935632,0.999995992966) (1.078880430850,0.999997996483) (1.078821874745,0.999998998242) (1.078784845930,0.999999499121)
    (1.078761448884,0.999999749560) (1.078746674505,0.999999874780) (1.078721427414,1.000000000000)
}

\def\ProfileDataThree{%
    (0.000000000000,-1.000000000000) (0.079945220714,-0.898428291952) (0.119911920382,-0.847642437928) (0.159867225672,-0.796856583904)
    (0.199802471871,-0.746070729880) (0.239706578940,-0.695284875856) (0.259642846260,-0.669891948844) (0.279566109599,-0.644499021832)
    (0.299474333515,-0.619106094820) (0.319365345681,-0.593713167808) (0.339236840093,-0.568320240795) (0.359086380410,-0.542927313783)
    (0.378911403402,-0.517534386771) (0.398709222435,-0.492141459759) (0.418477030944,-0.466748532747) (0.438211905833,-0.441355605735)
    (0.457910810735,-0.415962678723) (0.467745795360,-0.403266215217) (0.477570599063,-0.390569751711) (0.487384811758,-0.377873288205)
    (0.497188016797,-0.365176824699) (0.506979791042,-0.352480361193) (0.516759704923,-0.339783897687) (0.526527322491,-0.327087434181)
    (0.536282201463,-0.314390970675) (0.546023893261,-0.301694507169) (0.555751943032,-0.288998043663) (0.565465889672,-0.276301580157)
    (0.575165265829,-0.263605116651) (0.584849597903,-0.250908653145) (0.594518406026,-0.238212189639) (0.604171204036,-0.225515726133)
    (0.613807499440,-0.212819262627) (0.623426793359,-0.200122799121) (0.633028580460,-0.187426335615) (0.642612348874,-0.174729872109)
    (0.652177580105,-0.162033408603) (0.661723748912,-0.149336945097) (0.671250323182,-0.136640481591) (0.680756763784,-0.123944018085)
    (0.690242524407,-0.111247554579) (0.699707051376,-0.098551091073) (0.709149783449,-0.085854627567) (0.718570151596,-0.073158164061)
    (0.727967578754,-0.060461700555) (0.737341479560,-0.047765237049) (0.746691260062,-0.035068773543) (0.756016317404,-0.022372310037)
    (0.765316039487,-0.009675846531) (0.774589804597,0.003020616975) (0.783836981014,0.015717080481) (0.793056926581,0.028413543987)
    (0.802248988250,0.041110007493) (0.811412501587,0.053806470999) (0.820546790244,0.066502934505) (0.829651165398,0.079199398011)
    (0.838724925145,0.091895861517) (0.847767353847,0.104592325023) (0.856777721448,0.117288788529) (0.865755282727,0.129985252035)
    (0.874699276509,0.142681715541) (0.883608924815,0.155378179047) (0.892483431959,0.168074642553) (0.901321983577,0.180771106059)
    (0.910123745588,0.193467569565) (0.918887863088,0.206164033071) (0.927613459155,0.218860496577) (0.936299633576,0.231556960083)
    (0.940627649650,0.237905191836) (0.944945461481,0.244253423589) (0.949252949367,0.250601655342) (0.953549991875,0.256949887095)
    (0.957836465791,0.263298118848) (0.962112246061,0.269646350601) (0.966377205744,0.275994582354) (0.970631215946,0.282342814107)
    (0.974874145763,0.288691045861) (0.979105862218,0.295039277614) (0.983326230198,0.301387509367) (0.987535112383,0.307735741120)
    (0.991732369179,0.314083972873) (0.995917858644,0.320432204626) (1.000091436413,0.326780436379) (1.004252955618,0.333128668132)
    (1.008402266809,0.339476899885) (1.012539217864,0.345825131638) (1.016663653906,0.352173363391) (1.020775417205,0.358521595144)
    (1.024874347086,0.364869826897) (1.028960279827,0.371218058650) (1.033033048555,0.377566290403) (1.037092483137,0.383914522156)
    (1.041138410063,0.390262753909) (1.045170652333,0.396610985662) (1.049189029327,0.402959217415) (1.053193356681,0.409307449168)
    (1.057183446144,0.415655680921) (1.061159105445,0.422003912674) (1.065120138138,0.428352144427) (1.069066343448,0.434700376180)
    (1.072997516111,0.441048607933) (1.076913446200,0.447396839686) (1.080813918949,0.453745071439) (1.084698714564,0.460093303192)
    (1.088567608027,0.466441534945) (1.092420368887,0.472789766698) (1.096256761047,0.479137998451) (1.100076542531,0.485486230204)
    (1.103879465248,0.491834461957) (1.107665274736,0.498182693710) (1.111433709895,0.504530925463) (1.115184502710,0.510879157216)
    (1.118917377951,0.517227388969) (1.122632052863,0.523575620722) (1.126328236832,0.529923852475) (1.130005631041,0.536272084228)
    (1.133663928096,0.542620315981) (1.137302811637,0.548968547734) (1.140921955923,0.555316779487) (1.144521025394,0.561665011240)
    (1.148099674205,0.568013242993) (1.151657545728,0.574361474746) (1.155194272029,0.580709706499) (1.158709473305,0.587057938252)
    (1.162202757290,0.593406170005) (1.165673718619,0.599754401758) (1.169121938146,0.606102633511) (1.172546982223,0.612450865264)
    (1.175948401925,0.618799097017) (1.179325732217,0.625147328770) (1.182419978505,0.631004401758) (1.185492916505,0.636861474746)
    (1.188544140420,0.642718547734) (1.191573230034,0.648575620722) (1.194579749928,0.654432693710) (1.197563248641,0.660289766698)
    (1.200523257772,0.666146839686) (1.203459291012,0.672003912674) (1.206370843107,0.677860985662) (1.209257388732,0.683718058650)
    (1.212118381293,0.689575131638) (1.214953251612,0.695432204626) (1.217761406526,0.701289277614) (1.220542227352,0.707146350601)
    (1.223295068234,0.713003423589) (1.226019254339,0.718860496577) (1.228714079898,0.724717569565) (1.231378806069,0.730574642553)
    (1.234012658603,0.736431715541) (1.236614825287,0.742288788529) (1.239184453147,0.748145861517) (1.241720645365,0.754002934505)
    (1.244222457894,0.759860007493) (1.246688895720,0.765717080481) (1.249118908718,0.771574153469) (1.251511387070,0.777431226457)
    (1.253865156163,0.783288299445) (1.255027138566,0.786216835939) (1.256178970900,0.789145372433) (1.257320485470,0.792073908927)
    (1.258451509344,0.795002445421) (1.259571864102,0.797930981915) (1.260681365586,0.800859518409) (1.261779823620,0.803788054903)
    (1.262867041723,0.806716591397) (1.263942816798,0.809645127891) (1.265006938801,0.812573664385) (1.266059190390,0.815502200879)
    (1.267099346551,0.818430737373) (1.268127174195,0.821359273867) (1.269142431726,0.824287810361) (1.270144868588,0.827216346855)
    (1.271134224767,0.830144883349) (1.272110230264,0.833073419843) (1.273072604527,0.836001956337) (1.274021055839,0.838930492831)
    (1.274955280662,0.841859029325) (1.275874962921,0.844787565819) (1.276779773245,0.847716102313) (1.277669368129,0.850644638807)
    (1.278543389042,0.853573175301) (1.279401461445,0.856501711795) (1.280243193737,0.859430248289) (1.281068176097,0.862358784783)
    (1.281875979227,0.865287321277) (1.282666152978,0.868215857771) (1.283438224845,0.871144394265) (1.284191698316,0.874072930759)
    (1.284926051063,0.877001467253) (1.285640732938,0.879930003747) (1.286335163778,0.882858540241) (1.287008730960,0.885787076735)
    (1.287660786700,0.888715613229) (1.288290645045,0.891644149723) (1.288897578520,0.894572686217) (1.289480814378,0.897501222711)
    (1.290039530396,0.900429759205) (1.290572850140,0.903358295698) (1.291079837623,0.906286832192) (1.291559491251,0.909215368686)
    (1.292010736935,0.912143905180) (1.292432420226,0.915072441674) (1.292823297291,0.918000978168) (1.293182024522,0.920929514662)
    (1.293507146494,0.923858051156) (1.293797081954,0.926786587650) (1.294050107412,0.929715124144) (1.294264337820,0.932643660638)
    (1.294356266074,0.934107928885) (1.294437703657,0.935572197132) (1.294508358569,0.937036465379) (1.294567923567,0.938500733626)
    (1.294616074983,0.939965001873) (1.294652471423,0.941429270120) (1.294676752323,0.942893538367) (1.294688536352,0.944357806614)
    (1.294687419633,0.945822074861) (1.294672973756,0.947286343108) (1.294644743568,0.948750611355) (1.294602244676,0.950214879602)
    (1.294544960654,0.951679147849) (1.294472339877,0.953143416096) (1.294383791941,0.954607684343) (1.294278683585,0.956071952590)
    (1.294156334037,0.957536220837) (1.294016009673,0.959000489084) (1.293856917862,0.960464757331) (1.293678199841,0.961929025578)
    (1.293478922415,0.963393293825) (1.293258068245,0.964857562072) (1.293014524399,0.966321830319) (1.292747068783,0.967786098566)
    (1.292454353929,0.969250366813) (1.292134887487,0.970714635060) (1.291787008538,0.972178903307) (1.291408858564,0.973643171554)
    (1.290998345522,0.975107439801) (1.290553098837,0.976571708048) (1.290070412345,0.978035976295) (1.289814074094,0.978768110419)
    (1.289547170883,0.979500244542) (1.289269233700,0.980232378666) (1.288979754341,0.980964512789) (1.288678180485,0.981696646913)
    (1.288363909939,0.982428781036) (1.288036283841,0.983160915160) (1.287694578616,0.983893049283) (1.287337996356,0.984625183407)
    (1.286965653246,0.985357317530) (1.286576565510,0.986089451654) (1.286169632187,0.986821585777) (1.285743613808,0.987553719901)
    (1.285297105694,0.988285854024) (1.284828504085,0.989017988148) (1.284335962569,0.989750122271) (1.283817335130,0.990482256395)
    (1.283270100327,0.991214390518) (1.282691258215,0.991946524642) (1.282077186797,0.992678658765) (1.281755589304,0.993044725827)
    (1.281423436294,0.993410792889) (1.281079992398,0.993776859950) (1.280724424158,0.994142927012) (1.280355779906,0.994508994074)
    (1.279972963891,0.994875061136) (1.279574702445,0.995241128197) (1.279159498943,0.995607195259) (1.278725572570,0.995973262321)
    (1.278270773043,0.996339329383) (1.277792458379,0.996705396444) (1.277287313691,0.997071463506) (1.276751071219,0.997437530568)
    (1.276469587780,0.997620564099) (1.276178055095,0.997803597630) (1.275875415351,0.997986631160) (1.275560391191,0.998169664691)
    (1.275231413072,0.998352698222) (1.274886511570,0.998535731753) (1.274523150980,0.998718765284) (1.274137958735,0.998901798815)
    (1.273726255748,0.999084832346) (1.273281167429,0.999267865877) (1.273042930177,0.999359382642) (1.272791724003,0.999450899407)
    (1.272524767737,0.999542416173) (1.272237975061,0.999633932938) (1.271924825130,0.999725449704) (1.271755036380,0.999771208086)
    (1.271573458069,0.999816966469) (1.271376213527,0.999862724852) (1.271156251794,0.999908483235) (1.271033453712,0.999931362426)
    (1.270897150352,0.999954241617) (1.270737577419,0.999977120809) (1.270639941156,0.999988560404) (1.270580510010,0.999994280202)
    (1.270544484094,0.999997140101) (1.270522719050,0.999998570051) (1.270489947667,1.000000000000)
}

\def\ProfileDataFour{%
    (0.000000000000,-1.000000000000) (0.159783570902,-0.786660098391) (0.239665374879,-0.679990147586) (0.279597238360,-0.626655172184)
    (0.319518038449,-0.573320196782) (0.359422386917,-0.519985221379) (0.399303394174,-0.466650245977) (0.439152498873,-0.413315270575)
    (0.459061914918,-0.386647782874) (0.478959304739,-0.359980295173) (0.498843056488,-0.333312807471) (0.518711424446,-0.306645319770)
    (0.538562524432,-0.279977832069) (0.558394329206,-0.253310344368) (0.578204663868,-0.226642856667) (0.597991201165,-0.199975368966)
    (0.617751456667,-0.173307881265) (0.637482783734,-0.146640393564) (0.657182368190,-0.119972905862) (0.676847222600,-0.093305418161)
    (0.696474180047,-0.066637930460) (0.706272407000,-0.053304186610) (0.716059887277,-0.039970442759) (0.725836171354,-0.026636698908)
    (0.735600797054,-0.013302955058) (0.745353289254,0.000030788793) (0.755093159575,0.013364532643) (0.764819906041,0.026698276494)
    (0.774533012729,0.040032020344) (0.784231949380,0.053365764195) (0.793916170994,0.066699508046) (0.803585117383,0.080033251896)
    (0.813238212703,0.093366995747) (0.822874864946,0.106700739597) (0.832494465390,0.120034483448) (0.842096388011,0.133368227298)
    (0.851679988853,0.146701971149) (0.861244605342,0.160035715000) (0.870789555549,0.173369458850) (0.880314137400,0.186703202701)
    (0.889817627813,0.200036946551) (0.899299281776,0.213370690402) (0.908758331344,0.226704434252) (0.918193984555,0.240038178103)
    (0.927605424261,0.253371921954) (0.936991806852,0.266705665804) (0.946352260889,0.280039409655) (0.955685885602,0.293373153505)
    (0.964991749272,0.306706897356) (0.974268887470,0.320040641206) (0.983516301133,0.333374385057) (0.992732954482,0.346708128908)
    (1.001917772738,0.360041872758) (1.011069639635,0.373375616609) (1.020187394700,0.386709360459) (1.029269830272,0.400043104310)
    (1.038315688235,0.413376848160) (1.047323656429,0.426710592011) (1.056292364699,0.440044335862) (1.060761551267,0.446711207787)
    (1.065220380537,0.453378079712) (1.069668662913,0.460044951637) (1.074106204276,0.466711823563) (1.078532805798,0.473378695488)
    (1.082948263760,0.480045567413) (1.087352369347,0.486712439339) (1.091744908438,0.493379311264) (1.096125661388,0.500046183189)
    (1.100494402789,0.506713055114) (1.104850901222,0.513379927040) (1.109194918998,0.520046798965) (1.113526211874,0.526713670890)
    (1.117844528766,0.533380542816) (1.122149611427,0.540047414741) (1.126441194126,0.546714286666) (1.130719003287,0.553381158591)
    (1.134982757120,0.560048030517) (1.139232165222,0.566714902442) (1.143466928150,0.573381774367) (1.147686736973,0.580048646293)
    (1.151891272790,0.586715518218) (1.156080206212,0.593382390143) (1.160253196812,0.600049262068) (1.164409892541,0.606716133994)
    (1.168549929092,0.613383005919) (1.172672929226,0.620049877844) (1.176778502043,0.626716749770) (1.180866242207,0.633383621695)
    (1.184935729098,0.640050493620) (1.188986525911,0.646717365545) (1.193018178678,0.653384237471) (1.197030215208,0.660051109396)
    (1.201022143947,0.666717981321) (1.204993452737,0.673384853247) (1.208943607468,0.680051725172) (1.212872050617,0.686718597097)
    (1.216778199653,0.693385469022) (1.220661445295,0.700052340948) (1.224521149611,0.706719212873) (1.227160141490,0.711301725172)
    (1.229787469763,0.715884237471) (1.232402900592,0.720466749770) (1.235006191989,0.725049262068) (1.237597093377,0.729631774367)
    (1.240175345126,0.734214286666) (1.242740678058,0.738796798965) (1.245292812904,0.743379311264) (1.247831459735,0.747961823563)
    (1.250356317342,0.752544335862) (1.252867072560,0.757126848160) (1.255363399560,0.761709360459) (1.257844959058,0.766291872758)
    (1.260311397481,0.770874385057) (1.262762346048,0.775456897356) (1.265197419780,0.780039409655) (1.267616216423,0.784621921954)
    (1.270018315267,0.789204434252) (1.272403275871,0.793786946551) (1.274770636657,0.798369458850) (1.277119913372,0.802951971149)
    (1.279450597410,0.807534483448) (1.281762153948,0.812116995747) (1.284054019908,0.816699508046) (1.286325601697,0.821282020344)
    (1.288576272702,0.825864532643) (1.290805370513,0.830447044942) (1.293012193829,0.835029557241) (1.295195999000,0.839612069540)
    (1.297355996154,0.844194581839) (1.299491344852,0.848777094138) (1.301601149181,0.853359606436) (1.303684452202,0.857942118735)
    (1.305740229654,0.862524631034) (1.307767382764,0.867107143333) (1.309764730021,0.871689655632) (1.311730997707,0.876272167931)
    (1.313664808962,0.880854680230) (1.315564671061,0.885437192529) (1.317428960567,0.890019704827) (1.319255905859,0.894602217126)
    (1.320154775267,0.896893473276) (1.321043566472,0.899184729425) (1.321922006495,0.901475985575) (1.322789808467,0.903767241724)
    (1.323646670567,0.906058497873) (1.324492274857,0.908349754023) (1.325326285994,0.910641010172) (1.326148349797,0.912932266322)
    (1.326958091663,0.915223522471) (1.327755114802,0.917514778621) (1.328538998259,0.919806034770) (1.329309294706,0.922097290919)
    (1.330065527957,0.924388547069) (1.330807190167,0.926679803218) (1.331533738661,0.928971059368) (1.332244592333,0.931262315517)
    (1.332939127546,0.933553571667) (1.333616673419,0.935844827816) (1.334276506430,0.938136083965) (1.334917844151,0.940427340115)
    (1.335539837990,0.942718596264) (1.336141564691,0.945009852414) (1.336722016349,0.947301108563) (1.337280088599,0.949592364713)
    (1.337814566523,0.951883620862) (1.338324107743,0.954174877011) (1.338807221929,0.956466133161) (1.339262245755,0.958757389310)
    (1.339687311977,0.961048645460) (1.340080310804,0.963339901609) (1.340438841066,0.965631157759) (1.340760147566,0.967922413908)
    (1.340905864524,0.969068041983) (1.341041039480,0.970213670057) (1.341165187789,0.971359298132) (1.341277782081,0.972504926207)
    (1.341378246536,0.973650554282) (1.341465950122,0.974796182356) (1.341540198522,0.975941810431) (1.341600224447,0.977087438506)
    (1.341645175899,0.978233066580) (1.341674101817,0.979378694655) (1.341685934362,0.980524322730) (1.341679466779,0.981669950805)
    (1.341653325410,0.982815578879) (1.341605933788,0.983961206954) (1.341535465860,0.985106835029) (1.341439783918,0.986252463103)
    (1.341316354537,0.987398091178) (1.341162131959,0.988543719253) (1.340973391688,0.989689347328) (1.340864669858,0.990262161365)
    (1.340745484964,0.990834975402) (1.340615049941,0.991407789440) (1.340472461424,0.991980603477) (1.340316672914,0.992553417514)
    (1.340146459168,0.993126231552) (1.339960367975,0.993699045589) (1.339756653236,0.994271859626) (1.339533179488,0.994844673664)
    (1.339287281010,0.995417487701) (1.339015545271,0.995990301738) (1.338713462883,0.996563115776) (1.338549192823,0.996849522795)
    (1.338374824069,0.997135929813) (1.338189131399,0.997422336832) (1.337990585236,0.997708743851) (1.337777227341,0.997995150869)
    (1.337546469844,0.998281557888) (1.337294748752,0.998567964907) (1.337016873197,0.998854371925) (1.336704647921,0.999140778944)
    (1.336531515172,0.999283982453) (1.336343373407,0.999427185963) (1.336135517206,0.999570389472) (1.335899506123,0.999713592981)
    (1.335765874650,0.999785194736) (1.335616078194,0.999856796491) (1.335439054918,0.999928398245) (1.335330266804,0.999964199123)
    (1.335264228131,0.999982099561) (1.335224521113,0.999991049781) (1.335200826673,0.999995524890) (1.335166767080,1.000000000000)
}

\def\ProfileDataFive{%
    (0.000000000000,-1.000000000000) (0.322239124185,-0.560336990017) (0.402778672919,-0.450421237521) (0.443036343966,-0.395463361273)
    (0.483280718177,-0.340505485025) (0.523506308857,-0.285547608777) (0.563706094631,-0.230589732529) (0.603871263122,-0.175631856282)
    (0.623937554970,-0.148152918158) (0.643990935690,-0.120673980034) (0.664029673978,-0.093195041910) (0.684051870981,-0.065716103786)
    (0.704055449513,-0.038237165662) (0.724038142681,-0.010758227538) (0.743997481780,0.016720710586) (0.763930783352,0.044199648710)
    (0.783835135205,0.071678586834) (0.803707381213,0.099157524958) (0.823544104626,0.126636463082) (0.833448001852,0.140375932144)
    (0.843341609578,0.154115401206) (0.853224418821,0.167854870268) (0.863095900433,0.181594339330) (0.872955504300,0.195333808392)
    (0.882802658493,0.209073277454) (0.892636768354,0.222812746516) (0.902457215518,0.236552215578) (0.912263356865,0.250291684640)
    (0.922054523380,0.264031153702) (0.931830018941,0.277770622764) (0.941589118995,0.291510091826) (0.951331069126,0.305249560888)
    (0.961055083513,0.318989029950) (0.970760343239,0.332728499012) (0.980445994463,0.346467968074) (0.990111146419,0.360207437136)
    (0.999754869231,0.373946906197) (1.009376191525,0.387686375259) (1.018974097808,0.401425844321) (1.028547525586,0.415165313383)
    (1.038095362195,0.428904782445) (1.047616441301,0.442644251507) (1.057109539035,0.456383720569) (1.066573369704,0.470123189631)
    (1.076006581033,0.483862658693) (1.085407748862,0.497602127755) (1.094775371230,0.511341596817) (1.104107861753,0.525081065879)
    (1.113403542180,0.538820534941) (1.122660634026,0.552560004003) (1.127274123741,0.559429738534) (1.131877249104,0.566299473065)
    (1.136469755192,0.573169207596) (1.141051378805,0.580038942127) (1.145621848049,0.586908676658) (1.150180881884,0.593778411189)
    (1.154728189644,0.600648145720) (1.159263470522,0.607517880251) (1.163786413013,0.614387614782) (1.168296694325,0.621257349313)
    (1.172793979734,0.628127083844) (1.177277921895,0.634996818375) (1.181748160102,0.641866552906) (1.186204319477,0.648736287437)
    (1.190646010111,0.655606021968) (1.195072826113,0.662475756499) (1.199484344592,0.669345491030) (1.203880124548,0.676215225561)
    (1.208259705661,0.683084960092) (1.212622606974,0.689954694623) (1.216968325450,0.696824429154) (1.221296334401,0.703694163685)
    (1.225606081753,0.710563898216) (1.229896988152,0.717433632747) (1.234168444865,0.724303367278) (1.238419811474,0.731173101809)
    (1.242650413317,0.738042836340) (1.246859538654,0.744912570871) (1.251046435509,0.751782305402) (1.255210308153,0.758652039933)
    (1.259754209014,0.766194163685) (1.264268151564,0.773736287437) (1.268750863314,0.781278411189) (1.273200978913,0.788820534941)
    (1.277617029654,0.796362658693) (1.281997431348,0.803904782445) (1.286340470252,0.811446906197) (1.288497403704,0.815217968074)
    (1.290644286651,0.818989029950) (1.292780860660,0.822760091826) (1.294906855579,0.826531153702) (1.297021988732,0.830302215578)
    (1.299125964036,0.834073277454) (1.301218471043,0.837844339330) (1.303299183876,0.841615401206) (1.305367760074,0.845386463082)
    (1.307423839304,0.849157524958) (1.309467041954,0.852928586834) (1.311496967564,0.856699648710) (1.313513193092,0.860470710586)
    (1.315515270979,0.864241772462) (1.317502726996,0.868012834338) (1.319475057827,0.871783896214) (1.321431728366,0.875554958090)
    (1.323372168657,0.879326019966) (1.325295770451,0.883097081842) (1.327201883289,0.886868143718) (1.329089810042,0.890639205595)
    (1.330958801814,0.894410267471) (1.332808052077,0.898181329347) (1.334636689903,0.901952391223) (1.336443772110,0.905723453099)
    (1.338228274093,0.909494514975) (1.339989079056,0.913265576851) (1.341724965290,0.917036638727) (1.343434591032,0.920807700603)
    (1.345116476307,0.924578762479) (1.346768980966,0.928349824355) (1.348390277892,0.932120886231) (1.349978319961,0.935891948107)
    (1.351530798861,0.939663009983) (1.353045093111,0.943434071859) (1.353787000157,0.945319602797) (1.354518201541,0.947205133735)
    (1.355238248627,0.949090664673) (1.355946656794,0.950976195611) (1.356642901040,0.952861726549) (1.357326410858,0.954747257487)
    (1.357996564212,0.956632788425) (1.358652680420,0.958518319363) (1.359294011680,0.960403850301) (1.359919732917,0.962289381239)
    (1.360528929490,0.964174912178) (1.361120582196,0.966060443116) (1.361693548783,0.967945974054) (1.362246540900,0.969831504992)
    (1.362778094996,0.971717035930) (1.363286535081,0.973602566868) (1.363769924300,0.975488097806) (1.364226000830,0.977373628744)
    (1.364652091277,0.979259159682) (1.365044990849,0.981144690620) (1.365400792873,0.983030221558) (1.365563310227,0.983972987027)
    (1.365714638038,0.984915752496) (1.365853957391,0.985858517965) (1.365980330371,0.986801283434) (1.366092672927,0.987744048903)
    (1.366189718941,0.988686814372) (1.366269971693,0.989629579841) (1.366331636687,0.990572345310) (1.366372526089,0.991515110779)
    (1.366389918139,0.992457876248) (1.366380341848,0.993400641717) (1.366339230305,0.994343407186) (1.366304984100,0.994814789920)
    (1.366260325614,0.995286172655) (1.366204028550,0.995757555389) (1.366134567200,0.996228938124) (1.366049995684,0.996700320858)
    (1.365947753479,0.997171703593) (1.365824331613,0.997643086327) (1.365674649026,0.998114469062) (1.365490740721,0.998585851796)
    (1.365381958654,0.998821543164) (1.365258451566,0.999057234531) (1.365115751032,0.999292925898) (1.364945915321,0.999528617265)
    (1.364846101271,0.999646462949) (1.364731064376,0.999764308633) (1.364590771188,0.999882154316) (1.364502181916,0.999941077158)
    (1.364447537931,0.999970538579) (1.364414407981,0.999985269290) (1.364366421013,1.000000000000)
}

\begin{figure}[t]
\centering
\begin{tikzpicture}[
  x=3.25cm,y=3.25cm,
  >=Latex,
  line cap=round,
  line join=round,
  curve1/.style={curveone,very thick},
  curve2/.style={curvetwo,very thick},
  curve3/.style={curvethree,very thick},
  curve4/.style={curvefour,very thick},
  curve5/.style={curvefive,very thick},
  conecurve/.style={black!65,densely dashed,thick}
]
  \draw[->,black!65] (-1.52,0)--(1.53,0) node[right] {$x$};
  \draw[->,black!65] (0,-1.10)--(0,1.12) node[above] {$t$};
  \foreach \x in {-1,1} {
    \draw[black!65] (\x,-.018)--(\x,.018);
    \node[below=2pt] at (\x,0) {$\x$};
  }
  \foreach \y in {-1,1} {
    \draw[black!65] (-.018,\y)--(.018,\y);
    \node[left=2pt] at (0,\y) {$\y$};
  }
  \node[below left=2pt] at (0,0) {$0$};

  \draw[conecurve]
    (-1.414213562373,1)--(0,-1)--(1.414213562373,1)--cycle;

  \draw[curve5] plot coordinates {\ProfileDataFive};
  \begin{scope}[xscale=-1]\draw[curve5] plot coordinates {\ProfileDataFive};\end{scope}
  \draw[curve5] (-1.364366421013,1)--(1.364366421013,1);

  \draw[curve4] plot coordinates {\ProfileDataFour};
  \begin{scope}[xscale=-1]\draw[curve4] plot coordinates {\ProfileDataFour};\end{scope}
  \draw[curve4] (-1.335166767080,1)--(1.335166767080,1);

  \draw[curve3] plot coordinates {\ProfileDataThree};
  \begin{scope}[xscale=-1]\draw[curve3] plot coordinates {\ProfileDataThree};\end{scope}
  \draw[curve3] (-1.270489947667,1)--(1.270489947667,1);

  \draw[curve2] plot coordinates {\ProfileDataTwo};
  \begin{scope}[xscale=-1]\draw[curve2] plot coordinates {\ProfileDataTwo};\end{scope}
  \draw[curve2] (-1.078721427414,1)--(1.078721427414,1);

  \draw[curve1] (0,0) circle[radius=1];

  \draw[curve1] (1.60,.74)--(1.82,.74);
  \node[anchor=west] at (1.88,.74) {$n=1$};
  \draw[curve2] (1.60,.49)--(1.82,.49);
  \node[anchor=west] at (1.88,.49) {$n=2$};
  \draw[curve3] (1.60,.24)--(1.82,.24);
  \node[anchor=west] at (1.88,.24) {$n=3$};
  \draw[curve4] (1.60,-.01)--(1.82,-.01);
  \node[anchor=west] at (1.88,-.01) {$n=4$};
  \draw[curve5] (1.60,-.26)--(1.82,-.26);
  \node[anchor=west] at (1.88,-.26) {$n=5$};
  \draw[conecurve] (1.60,-.51)--(1.82,-.51);
  \node[anchor=west] at (1.88,-.51) {balanced cone};
\end{tikzpicture}
\caption{Approximate extremal profiles for $n=1,\ldots,5$.
The dashed triangle is the balanced cone
$r_{\rm cone}(t)=(1+t)/\sqrt2$. }
\label{fig:numerical-optimal-profiles}
\end{figure}
\endgroup

%% file: sharp_santalo_optimal_paths_tikz_fragment_2026-09-01.tex
\begingroup
\definecolor{curveone}{RGB}{31,119,180}
\definecolor{curvetwo}{RGB}{230,126,34}
\definecolor{curvethree}{RGB}{44,160,44}
\definecolor{curvefour}{RGB}{214,39,40}
\definecolor{curvefive}{RGB}{148,103,189}

\def\PathDataTwo{%
    (-1.000000000000,-1.000000000000) (-0.990445273387,-0.999999785280) (-0.980929192993,-0.999998292911) (-0.976186015552,-0.999996676515)
    (-0.971452934073,-0.999994275739) (-0.966730082831,-0.999990940239) (-0.962017591067,-0.999986522015) (-0.957315582996,-0.999980875470)
    (-0.952624177836,-0.999973857457) (-0.947943489835,-0.999965327333) (-0.943273628305,-0.999955146996) (-0.938614697667,-0.999943180925)
    (-0.933966797490,-0.999929296216) (-0.929330022545,-0.999913362605) (-0.924704462861,-0.999895252500) (-0.920090203777,-0.999874840994)
    (-0.915487326013,-0.999852005888) (-0.910895905730,-0.999826627700) (-0.906316014604,-0.999798589674) (-0.901747719896,-0.999767777787)
    (-0.897191084527,-0.999734080748) (-0.892646167159,-0.999697389999) (-0.888113022272,-0.999657599707) (-0.883591700249,-0.999614606759)
    (-0.879082247458,-0.999568310751) (-0.874584706340,-0.999518613970) (-0.870099115494,-0.999465421386) (-0.865625509764,-0.999408640626)
    (-0.861163920334,-0.999348181954) (-0.856714374809,-0.999283958253) (-0.852276897312,-0.999215884994) (-0.847851508572,-0.999143880214)
    (-0.843438226015,-0.999067864483) (-0.839037063855,-0.998987760875) (-0.834648033184,-0.998903494938) (-0.830271142063,-0.998814994660)
    (-0.825906395611,-0.998722190432) (-0.821553796095,-0.998625015017) (-0.817213343020,-0.998523403508) (-0.812885033212,-0.998417293298)
    (-0.808568860913,-0.998306624032) (-0.804264817857,-0.998191337579) (-0.799972893365,-0.998071377982) (-0.795693074421,-0.997946691426)
    (-0.791425345759,-0.997817226194) (-0.787169689943,-0.997682932626) (-0.782926087446,-0.997543763079) (-0.778694516732,-0.997399671884)
    (-0.774474954328,-0.997250615309) (-0.770267374904,-0.997096551512) (-0.766071751346,-0.996937440504) (-0.761888054828,-0.996773244104)
    (-0.757716254883,-0.996603925901) (-0.753556319475,-0.996429451213) (-0.749408215063,-0.996249787043) (-0.745271906672,-0.996064902041)
    (-0.741147357954,-0.995874766463) (-0.737034531252,-0.995679352132) (-0.732933387664,-0.995478632399) (-0.728843887101,-0.995272582099)
    (-0.724765988344,-0.995061177521) (-0.720699649103,-0.994844396362) (-0.716644826074,-0.994622217694) (-0.712601474987,-0.994394621924)
    (-0.708569550663,-0.994161590759) (-0.704549007063,-0.993923107172) (-0.700539797337,-0.993679155362) (-0.696541873873,-0.993429720721)
    (-0.692555188338,-0.993174789801) (-0.688579691731,-0.992914350280) (-0.684615334418,-0.992648390926) (-0.680662066180,-0.992376901570)
    (-0.676719836250,-0.992099873068) (-0.672788593353,-0.991817297275) (-0.668868285743,-0.991529167010) (-0.664958861241,-0.991235476030)
    (-0.661060267271,-0.990936218999) (-0.657172450890,-0.990631391457) (-0.653295358825,-0.990320989796) (-0.649428937501,-0.990005011231)
    (-0.645573133076,-0.989683453771) (-0.641727891465,-0.989356316195) (-0.637893158372,-0.989023598028) (-0.634068879315,-0.988685299512)
    (-0.630254999655,-0.988341421584) (-0.626451464616,-0.987991965852) (-0.622658219315,-0.987636934572) (-0.618875208781,-0.987276330624)
    (-0.615102377978,-0.986910157492) (-0.611339671828,-0.986538419241) (-0.607587035229,-0.986161120496) (-0.600111750281,-0.985389862709)
    (-0.592676082589,-0.984596431582) (-0.585279592412,-0.983780882285) (-0.577921841323,-0.982943277128) (-0.570602392683,-0.982083685013)
    (-0.563320812064,-0.981202180921) (-0.556076667637,-0.980298845422) (-0.548869530517,-0.979373764219) (-0.541698975088,-0.978427027710)
    (-0.534564579274,-0.977458730582) (-0.527465924802,-0.976468971432) (-0.520402597423,-0.975457852398) (-0.513374187114,-0.974425478825)
    (-0.506380288253,-0.973371958947) (-0.499420499770,-0.972297403587) (-0.492494425289,-0.971201925875) (-0.485601673234,-0.970085640990)
    (-0.478741856931,-0.968948665910) (-0.471914594690,-0.967791119186) (-0.465119509869,-0.966613120724) (-0.458356230931,-0.965414791584)
    (-0.451624391482,-0.964196253797) (-0.444923630301,-0.962957630185) (-0.438253591360,-0.961699044198) (-0.431613923835,-0.960420619767)
    (-0.425004282104,-0.959122481157) (-0.418424325745,-0.957804752838) (-0.411873719517,-0.956467559364) (-0.405352133341,-0.955111025255)
    (-0.398859242279,-0.953735274897) (-0.392394726495,-0.952340432443) (-0.385958271223,-0.950926621720) (-0.379549566728,-0.949493966148)
    (-0.373168308258,-0.948042588663) (-0.366814196001,-0.946572611645) (-0.360486935028,-0.945084156854) (-0.354186235249,-0.943577345368)
    (-0.347911811348,-0.942052297530) (-0.341663382734,-0.940509132896) (-0.335440673476,-0.938947970189) (-0.329243412246,-0.937368927258)
    (-0.323071332255,-0.935772121037) (-0.316924171192,-0.934157667516) (-0.310801671160,-0.932525681703) (-0.304703578610,-0.930876277598)
    (-0.298629644281,-0.929209568170) (-0.292579623128,-0.927525665333) (-0.286553274265,-0.925824679924) (-0.280550360893,-0.924106721686)
    (-0.274570650239,-0.922371899257) (-0.268613913488,-0.920620320148) (-0.262679925722,-0.918852090740) (-0.256768465852,-0.917067316267)
    (-0.250879316557,-0.915266100813) (-0.245012264217,-0.913448547303) (-0.239167098855,-0.911614757499) (-0.233343614070,-0.909764831994)
    (-0.227541606976,-0.907898870213) (-0.221760878145,-0.906016970408) (-0.216001231541,-0.904119229659) (-0.210262474464,-0.902205743874)
    (-0.204544417492,-0.900276607793) (-0.198846874418,-0.898331914986) (-0.193169662198,-0.896371757859) (-0.187512600893,-0.894396227657)
    (-0.181875513612,-0.892405414469) (-0.176258226460,-0.890399407232) (-0.170660568481,-0.888378293738) (-0.165082371607,-0.886342160637)
    (-0.159523470607,-0.884291093447) (-0.153983703035,-0.882225176558) (-0.148462909177,-0.880144493244) (-0.142960932005,-0.878049125663)
    (-0.137477617130,-0.875939154873) (-0.132012812747,-0.873814660834) (-0.126566369597,-0.871675722419) (-0.121138140916,-0.869522417425)
    (-0.115727982390,-0.867354822576) (-0.110335752114,-0.865173013539) (-0.104961310544,-0.862977064925) (-0.099604520460,-0.860767050308)
    (-0.094265246919,-0.858543042226) (-0.088943357217,-0.856305112194) (-0.083638720850,-0.854053330714) (-0.078351209470,-0.851787767283)
    (-0.073080696853,-0.849508490405) (-0.067827058855,-0.847215567597) (-0.062590173380,-0.844909065404) (-0.057369920342,-0.842589049401)
    (-0.052166181626,-0.840255584212) (-0.046978841059,-0.837908733511) (-0.041807784374,-0.835548560038) (-0.036652899174,-0.833175125604)
    (-0.031514074903,-0.830788491105) (-0.026391202810,-0.828388716526) (-0.021284175922,-0.825975860957) (-0.016192889012,-0.823549982596)
    (-0.011117238565,-0.821111138763) (-0.006057122755,-0.818659385906) (-0.001012441411,-0.816194779613) (0.004016904008,-0.813717374619)
    (0.009031010443,-0.811227224816) (0.014029973259,-0.808724383259) (0.019013886277,-0.806208902182) (0.023982841792,-0.803680832998)
    (0.028936930605,-0.801140226313) (0.033876242045,-0.798587131933) (0.038800863995,-0.796021598874) (0.043710882913,-0.793443675366)
    (0.048606383857,-0.790853408865) (0.053487450510,-0.788250846061) (0.058354165196,-0.785636032883) (0.063206608910,-0.783009014512)
    (0.068044861333,-0.780369835382) (0.072869000856,-0.777718539193) (0.077679104599,-0.775055168917) (0.082475248433,-0.772379766804)
    (0.087257506999,-0.769692374390) (0.092025953724,-0.766993032506) (0.096780660844,-0.764281781282) (0.101521699423,-0.761558660158)
    (0.106249139364,-0.758823707884) (0.110963049437,-0.756076962535) (0.115663497287,-0.753318461512) (0.120350549456,-0.750548241550)
    (0.125024271400,-0.747766338724) (0.129684727500,-0.744972788457) (0.134331981084,-0.742167625523) (0.138966094438,-0.739350884057)
    (0.143587128826,-0.736522597556) (0.148195144499,-0.733682798890) (0.152790200712,-0.730831520304) (0.157372355742,-0.727968793424)
    (0.161941666895,-0.725094649266) (0.166498190526,-0.722209118236) (0.171041982047,-0.719312230139) (0.175573095944,-0.716404014181)
    (0.180091585789,-0.713484498980) (0.184597504251,-0.710553712563) (0.189090903108,-0.707611682377) (0.193571833263,-0.704658435291)
    (0.198040344749,-0.701693997601) (0.202496486748,-0.698718395034) (0.206940307596,-0.695731652754) (0.211371854798,-0.692733795366)
    (0.215791175035,-0.689724846918) (0.220198314179,-0.686704830906) (0.224593317300,-0.683673770281) (0.228976228677,-0.680631687450)
    (0.233347091807,-0.677578604279) (0.237705949418,-0.674514542099) (0.242052843473,-0.671439521709) (0.246387815183,-0.668353563379)
    (0.250710905016,-0.665256686854) (0.255022152705,-0.662148911356) (0.259321597255,-0.659030255589) (0.263609276953,-0.655900737742)
    (0.267885229378,-0.652760375490) (0.272149491406,-0.649609186000) (0.276402099218,-0.646447185932) (0.280643088311,-0.643274391441)
    (0.284872493502,-0.640090818181) (0.289090348936,-0.636896481309) (0.293296688095,-0.633691395485) (0.297491543804,-0.630475574874)
    (0.301674948235,-0.627249033153) (0.305846932920,-0.624011783506) (0.310007528751,-0.620763838633) (0.314156765992,-0.617505210749)
    (0.318294674279,-0.614235911585) (0.322421282633,-0.610955952392) (0.326536619460,-0.607665343941) (0.330640712561,-0.604364096527)
    (0.334733589135,-0.601052219970) (0.338815275785,-0.597729723614) (0.342885798526,-0.594396616332) (0.346945182786,-0.591052906524)
    (0.350993453414,-0.587698602124) (0.355030634683,-0.584333710593) (0.359056750299,-0.580958238928) (0.363071823399,-0.577572193658)
    (0.367075876564,-0.574175580848) (0.371068931813,-0.570768406098) (0.375051010619,-0.567350674545) (0.379022133903,-0.563922390862)
    (0.382982322046,-0.560483559261) (0.386931594888,-0.557034183494) (0.390869971732,-0.553574266850) (0.394797471353,-0.550103812159)
    (0.398714111995,-0.546622821789) (0.402619911379,-0.543131297651) (0.406514886703,-0.539629241195) (0.410399054650,-0.536116653410)
    (0.414272431388,-0.532593534829) (0.418135032572,-0.529059885523) (0.421986873351,-0.525515705105) (0.425827968367,-0.521960992727)
    (0.429658331761,-0.518395747081) (0.433477977172,-0.514819966400) (0.437286917746,-0.511233648454) (0.441085166130,-0.507636790555)
    (0.444872734481,-0.504029389549) (0.448649634468,-0.500411441823) (0.452415877269,-0.496782943298) (0.456171473580,-0.493143889431)
    (0.459916433612,-0.489494275216) (0.463650767095,-0.485834095178) (0.467374483279,-0.482163343376) (0.471087590939,-0.478482013400)
    (0.474790098372,-0.474790098372) (0.478482013400,-0.471087590939) (0.482163343376,-0.467374483279) (0.485834095178,-0.463650767095)
    (0.489494275216,-0.459916433612) (0.493143889431,-0.456171473580) (0.496782943298,-0.452415877269) (0.500411441823,-0.448649634468)
    (0.504029389549,-0.444872734481) (0.507636790555,-0.441085166130) (0.511233648454,-0.437286917746) (0.514819966400,-0.433477977172)
    (0.518395747081,-0.429658331761) (0.521960992727,-0.425827968367) (0.525515705105,-0.421986873351) (0.529059885523,-0.418135032572)
    (0.532593534829,-0.414272431388) (0.536116653410,-0.410399054650) (0.539629241195,-0.406514886703) (0.543131297651,-0.402619911379)
    (0.546622821789,-0.398714111995) (0.550103812159,-0.394797471353) (0.553574266850,-0.390869971732) (0.557034183494,-0.386931594888)
    (0.560483559261,-0.382982322046) (0.563922390862,-0.379022133903) (0.567350674545,-0.375051010619) (0.570768406098,-0.371068931813)
    (0.574175580848,-0.367075876564) (0.577572193658,-0.363071823399) (0.580958238928,-0.359056750299) (0.584333710593,-0.355030634683)
    (0.587698602124,-0.350993453414) (0.591052906524,-0.346945182786) (0.594396616332,-0.342885798526) (0.597729723614,-0.338815275785)
    (0.601052219970,-0.334733589135) (0.604364096527,-0.330640712561) (0.607665343941,-0.326536619460) (0.610955952392,-0.322421282633)
    (0.614235911585,-0.318294674279) (0.617505210749,-0.314156765992) (0.620763838633,-0.310007528751) (0.624011783506,-0.305846932920)
    (0.627249033153,-0.301674948235) (0.630475574874,-0.297491543804) (0.633691395485,-0.293296688095) (0.636896481309,-0.289090348936)
    (0.640090818181,-0.284872493502) (0.643274391441,-0.280643088311) (0.646447185932,-0.276402099218) (0.649609186000,-0.272149491406)
    (0.652760375490,-0.267885229378) (0.655900737742,-0.263609276953) (0.659030255589,-0.259321597255) (0.662148911356,-0.255022152705)
    (0.665256686854,-0.250710905016) (0.668353563379,-0.246387815183) (0.671439521709,-0.242052843473) (0.674514542099,-0.237705949418)
    (0.677578604279,-0.233347091807) (0.680631687450,-0.228976228677) (0.683673770281,-0.224593317300) (0.686704830906,-0.220198314179)
    (0.689724846918,-0.215791175035) (0.692733795366,-0.211371854798) (0.695731652754,-0.206940307596) (0.698718395034,-0.202496486748)
    (0.701693997601,-0.198040344749) (0.704658435291,-0.193571833263) (0.707611682377,-0.189090903108) (0.710553712563,-0.184597504251)
    (0.713484498980,-0.180091585789) (0.716404014181,-0.175573095944) (0.719312230139,-0.171041982047) (0.722209118236,-0.166498190526)
    (0.725094649266,-0.161941666895) (0.727968793424,-0.157372355742) (0.730831520304,-0.152790200712) (0.733682798890,-0.148195144499)
    (0.736522597556,-0.143587128826) (0.739350884057,-0.138966094438) (0.742167625523,-0.134331981084) (0.744972788457,-0.129684727500)
    (0.747766338724,-0.125024271400) (0.750548241550,-0.120350549456) (0.753318461512,-0.115663497287) (0.756076962535,-0.110963049437)
    (0.758823707884,-0.106249139364) (0.761558660158,-0.101521699423) (0.764281781282,-0.096780660844) (0.766993032506,-0.092025953724)
    (0.769692374390,-0.087257506999) (0.772379766804,-0.082475248433) (0.775055168917,-0.077679104599) (0.777718539193,-0.072869000856)
    (0.780369835382,-0.068044861333) (0.783009014512,-0.063206608910) (0.785636032883,-0.058354165196) (0.788250846061,-0.053487450510)
    (0.790853408865,-0.048606383857) (0.793443675366,-0.043710882913) (0.796021598874,-0.038800863995) (0.798587131933,-0.033876242045)
    (0.801140226313,-0.028936930605) (0.803680832998,-0.023982841792) (0.806208902182,-0.019013886277) (0.808724383259,-0.014029973259)
    (0.811227224816,-0.009031010443) (0.813717374619,-0.004016904008) (0.816194779613,0.001012441411) (0.818659385906,0.006057122755)
    (0.821111138763,0.011117238565) (0.823549982596,0.016192889012) (0.825975860957,0.021284175922) (0.828388716526,0.026391202810)
    (0.830788491105,0.031514074903) (0.833175125604,0.036652899174) (0.835548560038,0.041807784374) (0.837908733511,0.046978841059)
    (0.840255584212,0.052166181626) (0.842589049401,0.057369920342) (0.844909065404,0.062590173380) (0.847215567597,0.067827058855)
    (0.849508490405,0.073080696853) (0.851787767283,0.078351209470) (0.854053330714,0.083638720850) (0.856305112194,0.088943357217)
    (0.858543042226,0.094265246919) (0.860767050308,0.099604520460) (0.862977064925,0.104961310544) (0.865173013539,0.110335752114)
    (0.867354822576,0.115727982390) (0.869522417425,0.121138140916) (0.871675722419,0.126566369597) (0.873814660834,0.132012812747)
    (0.875939154873,0.137477617130) (0.878049125663,0.142960932005) (0.880144493244,0.148462909177) (0.882225176558,0.153983703035)
    (0.884291093447,0.159523470607) (0.886342160637,0.165082371607) (0.888378293738,0.170660568481) (0.890399407232,0.176258226460)
    (0.892405414469,0.181875513612) (0.894396227657,0.187512600893) (0.896371757859,0.193169662198) (0.898331914986,0.198846874418)
    (0.900276607793,0.204544417492) (0.902205743874,0.210262474464) (0.904119229659,0.216001231541) (0.906016970408,0.221760878145)
    (0.907898870213,0.227541606976) (0.909764831994,0.233343614070) (0.911614757499,0.239167098855) (0.913448547303,0.245012264217)
    (0.915266100813,0.250879316557) (0.917067316267,0.256768465852) (0.918852090740,0.262679925722) (0.920620320148,0.268613913488)
    (0.922371899257,0.274570650239) (0.924106721686,0.280550360893) (0.925824679924,0.286553274265) (0.927525665333,0.292579623128)
    (0.929209568170,0.298629644281) (0.930876277598,0.304703578610) (0.932525681703,0.310801671160) (0.934157667516,0.316924171192)
    (0.935772121037,0.323071332255) (0.937368927258,0.329243412246) (0.938947970189,0.335440673476) (0.940509132896,0.341663382734)
    (0.942052297530,0.347911811348) (0.943577345368,0.354186235249) (0.945084156854,0.360486935028) (0.946572611645,0.366814196001)
    (0.948042588663,0.373168308258) (0.949493966148,0.379549566728) (0.950926621720,0.385958271223) (0.952340432443,0.392394726495)
    (0.953735274897,0.398859242279) (0.955111025255,0.405352133341) (0.956467559364,0.411873719517) (0.957804752838,0.418424325745)
    (0.959122481157,0.425004282104) (0.960420619767,0.431613923835) (0.961699044198,0.438253591360) (0.962957630185,0.444923630301)
    (0.964196253797,0.451624391482) (0.965414791584,0.458356230931) (0.966613120724,0.465119509869) (0.967791119186,0.471914594690)
    (0.968948665910,0.478741856931) (0.970085640990,0.485601673234) (0.971201925875,0.492494425289) (0.972297403587,0.499420499770)
    (0.973371958947,0.506380288253) (0.974425478825,0.513374187114) (0.975457852398,0.520402597423) (0.976468971432,0.527465924802)
    (0.977458730582,0.534564579274) (0.978427027710,0.541698975088) (0.979373764219,0.548869530517) (0.980298845422,0.556076667637)
    (0.981202180921,0.563320812064) (0.982083685013,0.570602392683) (0.982943277128,0.577921841323) (0.983780882285,0.585279592412)
    (0.984596431582,0.592676082589) (0.985389862709,0.600111750281) (0.986161120496,0.607587035229) (0.986538419241,0.611339671828)
    (0.986910157492,0.615102377978) (0.987276330624,0.618875208781) (0.987636934572,0.622658219315) (0.987991965852,0.626451464616)
    (0.988341421584,0.630254999655) (0.988685299512,0.634068879315) (0.989023598028,0.637893158372) (0.989356316195,0.641727891465)
    (0.989683453771,0.645573133076) (0.990005011231,0.649428937501) (0.990320989796,0.653295358825) (0.990631391457,0.657172450890)
    (0.990936218999,0.661060267271) (0.991235476030,0.664958861241) (0.991529167010,0.668868285743) (0.991817297275,0.672788593353)
    (0.992099873068,0.676719836250) (0.992376901570,0.680662066180) (0.992648390926,0.684615334418) (0.992914350280,0.688579691731)
    (0.993174789801,0.692555188338) (0.993429720721,0.696541873873) (0.993679155362,0.700539797337) (0.993923107172,0.704549007063)
    (0.994161590759,0.708569550663) (0.994394621924,0.712601474987) (0.994622217694,0.716644826074) (0.994844396362,0.720699649103)
    (0.995061177521,0.724765988344) (0.995272582099,0.728843887101) (0.995478632399,0.732933387664) (0.995679352132,0.737034531252)
    (0.995874766463,0.741147357954) (0.996064902041,0.745271906672) (0.996249787043,0.749408215063) (0.996429451213,0.753556319475)
    (0.996603925901,0.757716254883) (0.996773244104,0.761888054828) (0.996937440504,0.766071751346) (0.997096551512,0.770267374904)
    (0.997250615309,0.774474954328) (0.997399671884,0.778694516732) (0.997543763079,0.782926087446) (0.997682932626,0.787169689943)
    (0.997817226194,0.791425345759) (0.997946691426,0.795693074421) (0.998071377982,0.799972893365) (0.998191337579,0.804264817857)
    (0.998306624032,0.808568860913) (0.998417293298,0.812885033212) (0.998523403508,0.817213343020) (0.998625015017,0.821553796095)
    (0.998722190432,0.825906395611) (0.998814994660,0.830271142063) (0.998903494938,0.834648033184) (0.998987760875,0.839037063855)
    (0.999067864483,0.843438226015) (0.999143880214,0.847851508572) (0.999215884994,0.852276897312) (0.999283958253,0.856714374809)
    (0.999348181954,0.861163920334) (0.999408640626,0.865625509764) (0.999465421386,0.870099115494) (0.999518613970,0.874584706340)
    (0.999568310751,0.879082247458) (0.999614606759,0.883591700249) (0.999657599707,0.888113022272) (0.999697389999,0.892646167159)
    (0.999734080748,0.897191084527) (0.999767777787,0.901747719896) (0.999798589674,0.906316014604) (0.999826627700,0.910895905730)
    (0.999852005888,0.915487326013) (0.999874840994,0.920090203777) (0.999895252500,0.924704462861) (0.999913362605,0.929330022545)
    (0.999929296216,0.933966797490) (0.999943180925,0.938614697667) (0.999955146996,0.943273628305) (0.999965327333,0.947943489835)
    (0.999973857457,0.952624177836) (0.999980875470,0.957315582996) (0.999986522015,0.962017591067) (0.999990940239,0.966730082831)
    (0.999994275739,0.971452934073) (0.999996676515,0.976186015552) (0.999998292911,0.980929192993) (0.999999785280,0.990445273387)
    (1.000000000000,1.000000000000)
}

\def\PathDataThree{%
    (-1.000000000000,-1.000000000000) (-0.994572457197,-0.999999999948) (-0.989155526138,-0.999999999178) (-0.983749208613,-0.999999995854)
    (-0.978353507546,-0.999999986948) (-0.972968426932,-0.999999968261) (-0.967593971788,-0.999999934444) (-0.962230148099,-0.999999879027)
    (-0.956876962763,-0.999999794437) (-0.951534423544,-0.999999672027) (-0.946202539022,-0.999999502093) (-0.940881318541,-0.999999273899)
    (-0.935570772165,-0.999998975705) (-0.930270910630,-0.999998594779) (-0.924981745297,-0.999998117433) (-0.919703288107,-0.999997529032)
    (-0.914435551541,-0.999996814028) (-0.909178548573,-0.999995955976) (-0.903932292630,-0.999994937556) (-0.898696797548,-0.999993740600)
    (-0.893472077536,-0.999992346109) (-0.888258147136,-0.999990734279) (-0.883055021182,-0.999988884520) (-0.877862714764,-0.999986775480)
    (-0.872681243193,-0.999984385066) (-0.867510621962,-0.999981690466) (-0.862350866714,-0.999978668170) (-0.857201993209,-0.999975293992)
    (-0.852064017287,-0.999971543091) (-0.846936954838,-0.999967389995) (-0.841820821773,-0.999962808618) (-0.836715633989,-0.999957772284)
    (-0.831621407343,-0.999952253746) (-0.826538157621,-0.999946225210) (-0.821465900513,-0.999939658352) (-0.816404651582,-0.999932524341)
    (-0.811354426240,-0.999924793860) (-0.806315239723,-0.999916437124) (-0.801287107066,-0.999907423901) (-0.796270043080,-0.999897723533)
    (-0.791264062325,-0.999887304955) (-0.786269179094,-0.999876136716) (-0.781285407387,-0.999864186994) (-0.776312760894,-0.999851423623)
    (-0.771351252972,-0.999837814104) (-0.766400896627,-0.999823325629) (-0.761461704498,-0.999807925099) (-0.756533688836,-0.999791579139)
    (-0.751616861490,-0.999774254121) (-0.746711233892,-0.999755916181) (-0.741816817038,-0.999736531232) (-0.736933621476,-0.999716064989)
    (-0.732061657293,-0.999694482980) (-0.727200934101,-0.999671750565) (-0.722351461025,-0.999647832957) (-0.717513246690,-0.999622695230)
    (-0.712686299212,-0.999596302344) (-0.707870626190,-0.999568619155) (-0.703066234689,-0.999539610432) (-0.698273131240,-0.999509240876)
    (-0.693491321825,-0.999477475130) (-0.688720811873,-0.999444277798) (-0.683961606253,-0.999409613455) (-0.679213709266,-0.999373446666)
    (-0.674477124639,-0.999335741999) (-0.669751855522,-0.999296464036) (-0.665037904483,-0.999255577388) (-0.660335273501,-0.999213046710)
    (-0.655643963968,-0.999168836710) (-0.650963976679,-0.999122912165) (-0.646295311837,-0.999075237930) (-0.641637969044,-0.999025778953)
    (-0.636991947307,-0.998974500285) (-0.632357245031,-0.998921367089) (-0.627733860019,-0.998866344657) (-0.623121789478,-0.998809398414)
    (-0.618521030012,-0.998750493931) (-0.613931577626,-0.998689596938) (-0.609353427730,-0.998626673328) (-0.604786575136,-0.998561689170)
    (-0.600231014062,-0.998494610719) (-0.595686738136,-0.998425404421) (-0.591153740396,-0.998354036924) (-0.586632013297,-0.998280475087)
    (-0.582121548709,-0.998204685986) (-0.577622337927,-0.998126636923) (-0.573134371671,-0.998046295431) (-0.568657640092,-0.997963629284)
    (-0.564192132776,-0.997878606503) (-0.559737838750,-0.997791195360) (-0.555294746486,-0.997701364387) (-0.550862843909,-0.997609082381)
    (-0.546442118399,-0.997514318409) (-0.542032556802,-0.997417041815) (-0.537634145431,-0.997317222222) (-0.533246870074,-0.997214829540)
    (-0.528870716005,-0.997109833970) (-0.524505667984,-0.997002206006) (-0.520151710269,-0.996891916442) (-0.515808826620,-0.996778936373)
    (-0.511477000308,-0.996663237202) (-0.507156214124,-0.996544790639) (-0.502846450380,-0.996423568709) (-0.498547690927,-0.996299543750)
    (-0.494259917151,-0.996172688421) (-0.489983109992,-0.996042975698) (-0.485717249945,-0.995910378882) (-0.481462317067,-0.995774871598)
    (-0.477218290993,-0.995636427797) (-0.472985150936,-0.995495021760) (-0.468762875699,-0.995350628094) (-0.464551443683,-0.995203221741)
    (-0.460350832894,-0.995052777972) (-0.456161020954,-0.994899272391) (-0.451981985106,-0.994742680936) (-0.447813702226,-0.994582979878)
    (-0.443656148828,-0.994420145823) (-0.439509301075,-0.994254155712) (-0.435373134786,-0.994084986818) (-0.431247625445,-0.993912616751)
    (-0.427132748209,-0.993737023454) (-0.423028477916,-0.993558185203) (-0.418934789096,-0.993376080608) (-0.414851655975,-0.993190688612)
    (-0.410779052486,-0.993001988487) (-0.406716952280,-0.992809959839) (-0.402665328726,-0.992614582602) (-0.398624154927,-0.992415837038)
    (-0.394593403727,-0.992213703738) (-0.390573047714,-0.992008163617) (-0.386563059234,-0.991799197915) (-0.382563410395,-0.991586788195)
    (-0.378574073077,-0.991370916341) (-0.374595018938,-0.991151564557) (-0.370626219425,-0.990928715363) (-0.366667645776,-0.990702351596)
    (-0.362719269035,-0.990472456406) (-0.358781060053,-0.990239013254) (-0.354852989499,-0.990002005911) (-0.350935027866,-0.989761418455)
    (-0.347027145478,-0.989517235268) (-0.343129312499,-0.989269441035) (-0.339241498938,-0.989018020742) (-0.331495809377,-0.988504243398)
    (-0.323789834048,-0.987975789020) (-0.316123328174,-0.987432548029) (-0.308496045159,-0.986874415397) (-0.300907736789,-0.986301290558)
    (-0.293358153434,-0.985713077319) (-0.285847044233,-0.985109683767) (-0.278374157282,-0.984491022175) (-0.270939239812,-0.983857008909)
    (-0.263542038358,-0.983207564333) (-0.256182298926,-0.982542612709) (-0.248859767152,-0.981862082104) (-0.241574188452,-0.981165904292)
    (-0.234325308173,-0.980454014657) (-0.227112871724,-0.979726352096) (-0.219936624721,-0.978982858923) (-0.212796313105,-0.978223480775)
    (-0.205691683272,-0.977448166511) (-0.198622482182,-0.976656868126) (-0.191588457476,-0.975849540650) (-0.184589357577,-0.975026142059)
    (-0.177624931794,-0.974186633185) (-0.170694930411,-0.973330977622) (-0.163799104782,-0.972459141640) (-0.156937207414,-0.971571094095)
    (-0.150108992046,-0.970666806347) (-0.143314213726,-0.969746252168) (-0.136552628880,-0.968809407668) (-0.129823995381,-0.967856251206)
    (-0.123128072610,-0.966886763313) (-0.116464621516,-0.965900926614) (-0.109833404668,-0.964898725748) (-0.103234186309,-0.963880147297)
    (-0.096666732401,-0.962845179710) (-0.090130810670,-0.961793813230) (-0.083626190645,-0.960726039829) (-0.077152643698,-0.959641853131)
    (-0.070709943078,-0.958541248354) (-0.064297863939,-0.957424222238) (-0.057916183372,-0.956290772984) (-0.051564680431,-0.955140900195)
    (-0.045243136155,-0.953974604809) (-0.038951333587,-0.952791889046) (-0.032689057798,-0.951592756348) (-0.026456095901,-0.950377211325)
    (-0.020252237060,-0.949145259699) (-0.014077272513,-0.947896908254) (-0.007930995571,-0.946632164781) (-0.001813201637,-0.945351038034)
    (0.004276311796,-0.944053537676) (0.010337745134,-0.942739674237) (0.016371296682,-0.941409459063) (0.022377162638,-0.940062904280)
    (0.028355537095,-0.938700022744) (0.034306612043,-0.937320828003) (0.040230577365,-0.935925334256) (0.046127620846,-0.934513556314)
    (0.051997928174,-0.933085509565) (0.057841682947,-0.931641209931) (0.063659066674,-0.930180673840) (0.069450258790,-0.928703918186)
    (0.075215436656,-0.927210960302) (0.080954775574,-0.925701817920) (0.086668448791,-0.924176509147) (0.092356627512,-0.922635052430)
    (0.098019480913,-0.921077466531) (0.103657176145,-0.919503770495) (0.109269878355,-0.917913983626) (0.114857750692,-0.916308125458)
    (0.120420954320,-0.914686215734) (0.125959648436,-0.913048274373) (0.131473990280,-0.911394321457) (0.136964135150,-0.909724377198)
    (0.142430236417,-0.908038461924) (0.147872445539,-0.906336596052) (0.153290912076,-0.904618800069) (0.158685783706,-0.902885094513)
    (0.164057206242,-0.901135499952) (0.169405323643,-0.899370036966) (0.174730278037,-0.897588726131) (0.180032209730,-0.895791587995)
    (0.185311257226,-0.893978643071) (0.190567557244,-0.892149911812) (0.195801244732,-0.890305414601) (0.201012452885,-0.888445171730)
    (0.206201313161,-0.886569203392) (0.211367955300,-0.884677529664) (0.216512507336,-0.882770170490) (0.221635095617,-0.880847145674)
    (0.226735844822,-0.878908474863) (0.231814877975,-0.876954177534) (0.236872316464,-0.874984272987) (0.241908280058,-0.872998780328)
    (0.246922886921,-0.870997718461) (0.251916253629,-0.868981106077) (0.256888495191,-0.866948961643) (0.261839725057,-0.864901303392)
    (0.266770055143,-0.862838149313) (0.271679595842,-0.860759517144) (0.276568456039,-0.858665424359) (0.281436743133,-0.856555888162)
    (0.286284563047,-0.854430925479) (0.291112020246,-0.852290552948) (0.295919217753,-0.850134786913) (0.300706257166,-0.847963643412)
    (0.305473238668,-0.845777138176) (0.310220261049,-0.843575286619) (0.314947421717,-0.841358103828) (0.319654816713,-0.839125604560)
    (0.324342540727,-0.836877803237) (0.329010687115,-0.834614713933) (0.333659347909,-0.832336350376) (0.338288613834,-0.830042725936)
    (0.342898574321,-0.827733853620) (0.347489317524,-0.825409746071) (0.352060930330,-0.823070415558) (0.356613498375,-0.820715873971)
    (0.361147106059,-0.818346132817) (0.365661836555,-0.815961203217) (0.370157771827,-0.813561095897) (0.374634992642,-0.811145821185)
    (0.379093578581,-0.808715389008) (0.383533608052,-0.806269808885) (0.387955158307,-0.803809089926) (0.392358305448,-0.801333240821)
    (0.396743124446,-0.798842269844) (0.401109689146,-0.796336184844) (0.405458072287,-0.793814993242) (0.409788345507,-0.791278702026)
    (0.414100579358,-0.788727317750) (0.418394843319,-0.786160846527) (0.422671205802,-0.783579294027) (0.426929734170,-0.780982665474)
    (0.431170494744,-0.778370965639) (0.435393552812,-0.775744198841) (0.439598972648,-0.773102368940) (0.443786817511,-0.770445479335)
    (0.447957149666,-0.767773532961) (0.452110030389,-0.765086532285) (0.456245519977,-0.762384479300) (0.460363677759,-0.759667375528)
    (0.464464562107,-0.756935222011) (0.468548230445,-0.754188019309) (0.472614739257,-0.751425767498) (0.476664144099,-0.748648466167)
    (0.480696499607,-0.745856114413) (0.484711859504,-0.743048710838) (0.488710276613,-0.740226253548) (0.492691802865,-0.737388740148)
    (0.496656489305,-0.734536167737) (0.500604386103,-0.731668532910) (0.504535542562,-0.728785831748) (0.508450007126,-0.725888059822)
    (0.512347827388,-0.722975212183) (0.516229050099,-0.720047283365) (0.520093721176,-0.717104267376) (0.523941885708,-0.714146157698)
    (0.527773587966,-0.711172947284) (0.531588871408,-0.708184628554) (0.535387778690,-0.705181193389) (0.539170351670,-0.702162633133)
    (0.542936631416,-0.699128938586) (0.546686658216,-0.696080099999) (0.550420471579,-0.693016107076) (0.554138110248,-0.689936948966)
    (0.557839612204,-0.686842614260) (0.561525014673,-0.683733090989) (0.565194354131,-0.680608366620) (0.568847666313,-0.677468428052)
    (0.572484986218,-0.674313261613) (0.576106348114,-0.671142853052) (0.579711785549,-0.667957187544) (0.583301331350,-0.664756249678)
    (0.586875017633,-0.661540023456) (0.590432875811,-0.658308492290) (0.593974936593,-0.655061638999) (0.597501229997,-0.651799445798)
    (0.601011785350,-0.648521894305) (0.604506631297,-0.645228965527) (0.607985795804,-0.641920639862) (0.611449306166,-0.638596897090)
    (0.614897189007,-0.635257716373) (0.618329470293,-0.631903076247) (0.621746175329,-0.628532954621) (0.625147328770,-0.625147328770)
    (0.628532954621,-0.621746175329) (0.631903076247,-0.618329470293) (0.635257716373,-0.614897189007) (0.638596897090,-0.611449306166)
    (0.641920639862,-0.607985795804) (0.645228965527,-0.604506631297) (0.648521894305,-0.601011785350) (0.651799445798,-0.597501229997)
    (0.655061638999,-0.593974936593) (0.658308492290,-0.590432875811) (0.661540023456,-0.586875017633) (0.664756249678,-0.583301331350)
    (0.667957187544,-0.579711785549) (0.671142853052,-0.576106348114) (0.674313261613,-0.572484986218) (0.677468428052,-0.568847666313)
    (0.680608366620,-0.565194354131) (0.683733090989,-0.561525014673) (0.686842614260,-0.557839612204) (0.689936948966,-0.554138110248)
    (0.693016107076,-0.550420471579) (0.696080099999,-0.546686658216) (0.699128938586,-0.542936631416) (0.702162633133,-0.539170351670)
    (0.705181193389,-0.535387778690) (0.708184628554,-0.531588871408) (0.711172947284,-0.527773587966) (0.714146157698,-0.523941885708)
    (0.717104267376,-0.520093721176) (0.720047283365,-0.516229050099) (0.722975212183,-0.512347827388) (0.725888059822,-0.508450007126)
    (0.728785831748,-0.504535542562) (0.731668532910,-0.500604386103) (0.734536167737,-0.496656489305) (0.737388740148,-0.492691802865)
    (0.740226253548,-0.488710276613) (0.743048710838,-0.484711859504) (0.745856114413,-0.480696499607) (0.748648466167,-0.476664144099)
    (0.751425767498,-0.472614739257) (0.754188019309,-0.468548230445) (0.756935222011,-0.464464562107) (0.759667375528,-0.460363677759)
    (0.762384479300,-0.456245519977) (0.765086532285,-0.452110030389) (0.767773532961,-0.447957149666) (0.770445479335,-0.443786817511)
    (0.773102368940,-0.439598972648) (0.775744198841,-0.435393552812) (0.778370965639,-0.431170494744) (0.780982665474,-0.426929734170)
    (0.783579294027,-0.422671205802) (0.786160846527,-0.418394843319) (0.788727317750,-0.414100579358) (0.791278702026,-0.409788345507)
    (0.793814993242,-0.405458072287) (0.796336184844,-0.401109689146) (0.798842269844,-0.396743124446) (0.801333240821,-0.392358305448)
    (0.803809089926,-0.387955158307) (0.806269808885,-0.383533608052) (0.808715389008,-0.379093578581) (0.811145821185,-0.374634992642)
    (0.813561095897,-0.370157771827) (0.815961203217,-0.365661836555) (0.818346132817,-0.361147106059) (0.820715873971,-0.356613498375)
    (0.823070415558,-0.352060930330) (0.825409746071,-0.347489317524) (0.827733853620,-0.342898574321) (0.830042725936,-0.338288613834)
    (0.832336350376,-0.333659347909) (0.834614713933,-0.329010687115) (0.836877803237,-0.324342540727) (0.839125604560,-0.319654816713)
    (0.841358103828,-0.314947421717) (0.843575286619,-0.310220261049) (0.845777138176,-0.305473238668) (0.847963643412,-0.300706257166)
    (0.850134786913,-0.295919217753) (0.852290552948,-0.291112020246) (0.854430925479,-0.286284563047) (0.856555888162,-0.281436743133)
    (0.858665424359,-0.276568456039) (0.860759517144,-0.271679595842) (0.862838149313,-0.266770055143) (0.864901303392,-0.261839725057)
    (0.866948961643,-0.256888495191) (0.868981106077,-0.251916253629) (0.870997718461,-0.246922886921) (0.872998780328,-0.241908280058)
    (0.874984272987,-0.236872316464) (0.876954177534,-0.231814877975) (0.878908474863,-0.226735844822) (0.880847145674,-0.221635095617)
    (0.882770170490,-0.216512507336) (0.884677529664,-0.211367955300) (0.886569203392,-0.206201313161) (0.888445171730,-0.201012452885)
    (0.890305414601,-0.195801244732) (0.892149911812,-0.190567557244) (0.893978643071,-0.185311257226) (0.895791587995,-0.180032209730)
    (0.897588726131,-0.174730278037) (0.899370036966,-0.169405323643) (0.901135499952,-0.164057206242) (0.902885094513,-0.158685783706)
    (0.904618800069,-0.153290912076) (0.906336596052,-0.147872445539) (0.908038461924,-0.142430236417) (0.909724377198,-0.136964135150)
    (0.911394321457,-0.131473990280) (0.913048274373,-0.125959648436) (0.914686215734,-0.120420954320) (0.916308125458,-0.114857750692)
    (0.917913983626,-0.109269878355) (0.919503770495,-0.103657176145) (0.921077466531,-0.098019480913) (0.922635052430,-0.092356627512)
    (0.924176509147,-0.086668448791) (0.925701817920,-0.080954775574) (0.927210960302,-0.075215436656) (0.928703918186,-0.069450258790)
    (0.930180673840,-0.063659066674) (0.931641209931,-0.057841682947) (0.933085509565,-0.051997928174) (0.934513556314,-0.046127620846)
    (0.935925334256,-0.040230577365) (0.937320828003,-0.034306612043) (0.938700022744,-0.028355537095) (0.940062904280,-0.022377162638)
    (0.941409459063,-0.016371296682) (0.942739674237,-0.010337745134) (0.944053537676,-0.004276311796) (0.945351038034,0.001813201637)
    (0.946632164781,0.007930995571) (0.947896908254,0.014077272513) (0.949145259699,0.020252237060) (0.950377211325,0.026456095901)
    (0.951592756348,0.032689057798) (0.952791889046,0.038951333587) (0.953974604809,0.045243136155) (0.955140900195,0.051564680431)
    (0.956290772984,0.057916183372) (0.957424222238,0.064297863939) (0.958541248354,0.070709943078) (0.959641853131,0.077152643698)
    (0.960726039829,0.083626190645) (0.961793813230,0.090130810670) (0.962845179710,0.096666732401) (0.963880147297,0.103234186309)
    (0.964898725748,0.109833404668) (0.965900926614,0.116464621516) (0.966886763313,0.123128072610) (0.967856251206,0.129823995381)
    (0.968809407668,0.136552628880) (0.969746252168,0.143314213726) (0.970666806347,0.150108992046) (0.971571094095,0.156937207414)
    (0.972459141640,0.163799104782) (0.973330977622,0.170694930411) (0.974186633185,0.177624931794) (0.975026142059,0.184589357577)
    (0.975849540650,0.191588457476) (0.976656868126,0.198622482182) (0.977448166511,0.205691683272) (0.978223480775,0.212796313105)
    (0.978982858923,0.219936624721) (0.979726352096,0.227112871724) (0.980454014657,0.234325308173) (0.981165904292,0.241574188452)
    (0.981862082104,0.248859767152) (0.982542612709,0.256182298926) (0.983207564333,0.263542038358) (0.983857008909,0.270939239812)
    (0.984491022175,0.278374157282) (0.985109683767,0.285847044233) (0.985713077319,0.293358153434) (0.986301290558,0.300907736789)
    (0.986874415397,0.308496045159) (0.987432548029,0.316123328174) (0.987975789020,0.323789834048) (0.988504243398,0.331495809377)
    (0.989018020742,0.339241498938) (0.989269441035,0.343129312499) (0.989517235268,0.347027145478) (0.989761418455,0.350935027866)
    (0.990002005911,0.354852989499) (0.990239013254,0.358781060053) (0.990472456406,0.362719269035) (0.990702351596,0.366667645776)
    (0.990928715363,0.370626219425) (0.991151564557,0.374595018938) (0.991370916341,0.378574073077) (0.991586788195,0.382563410395)
    (0.991799197915,0.386563059234) (0.992008163617,0.390573047714) (0.992213703738,0.394593403727) (0.992415837038,0.398624154927)
    (0.992614582602,0.402665328726) (0.992809959839,0.406716952280) (0.993001988487,0.410779052486) (0.993190688612,0.414851655975)
    (0.993376080608,0.418934789096) (0.993558185203,0.423028477916) (0.993737023454,0.427132748209) (0.993912616751,0.431247625445)
    (0.994084986818,0.435373134786) (0.994254155712,0.439509301075) (0.994420145823,0.443656148828) (0.994582979878,0.447813702226)
    (0.994742680936,0.451981985106) (0.994899272391,0.456161020954) (0.995052777972,0.460350832894) (0.995203221741,0.464551443683)
    (0.995350628094,0.468762875699) (0.995495021760,0.472985150936) (0.995636427797,0.477218290993) (0.995774871598,0.481462317067)
    (0.995910378882,0.485717249945) (0.996042975698,0.489983109992) (0.996172688421,0.494259917151) (0.996299543750,0.498547690927)
    (0.996423568709,0.502846450380) (0.996544790639,0.507156214124) (0.996663237202,0.511477000308) (0.996778936373,0.515808826620)
    (0.996891916442,0.520151710269) (0.997002206006,0.524505667984) (0.997109833970,0.528870716005) (0.997214829540,0.533246870074)
    (0.997317222222,0.537634145431) (0.997417041815,0.542032556802) (0.997514318409,0.546442118399) (0.997609082381,0.550862843909)
    (0.997701364387,0.555294746486) (0.997791195360,0.559737838750) (0.997878606503,0.564192132776) (0.997963629284,0.568657640092)
    (0.998046295431,0.573134371671) (0.998126636923,0.577622337927) (0.998204685986,0.582121548709) (0.998280475087,0.586632013297)
    (0.998354036924,0.591153740396) (0.998425404421,0.595686738136) (0.998494610719,0.600231014062) (0.998561689170,0.604786575136)
    (0.998626673328,0.609353427730) (0.998689596938,0.613931577626) (0.998750493931,0.618521030012) (0.998809398414,0.623121789478)
    (0.998866344657,0.627733860019) (0.998921367089,0.632357245031) (0.998974500285,0.636991947307) (0.999025778953,0.641637969044)
    (0.999075237930,0.646295311837) (0.999122912165,0.650963976679) (0.999168836710,0.655643963968) (0.999213046710,0.660335273501)
    (0.999255577388,0.665037904483) (0.999296464036,0.669751855522) (0.999335741999,0.674477124639) (0.999373446666,0.679213709266)
    (0.999409613455,0.683961606253) (0.999444277798,0.688720811873) (0.999477475130,0.693491321825) (0.999509240876,0.698273131240)
    (0.999539610432,0.703066234689) (0.999568619155,0.707870626190) (0.999596302344,0.712686299212) (0.999622695230,0.717513246690)
    (0.999647832957,0.722351461025) (0.999671750565,0.727200934101) (0.999694482980,0.732061657293) (0.999716064989,0.736933621476)
    (0.999736531232,0.741816817038) (0.999755916181,0.746711233892) (0.999774254121,0.751616861490) (0.999791579139,0.756533688836)
    (0.999807925099,0.761461704498) (0.999823325629,0.766400896627) (0.999837814104,0.771351252972) (0.999851423623,0.776312760894)
    (0.999864186994,0.781285407387) (0.999876136716,0.786269179094) (0.999887304955,0.791264062325) (0.999897723533,0.796270043080)
    (0.999907423901,0.801287107066) (0.999916437124,0.806315239723) (0.999924793860,0.811354426240) (0.999932524341,0.816404651582)
    (0.999939658352,0.821465900513) (0.999946225210,0.826538157621) (0.999952253746,0.831621407343) (0.999957772284,0.836715633989)
    (0.999962808618,0.841820821773) (0.999967389995,0.846936954838) (0.999971543091,0.852064017287) (0.999975293992,0.857201993209)
    (0.999978668170,0.862350866714) (0.999981690466,0.867510621962) (0.999984385066,0.872681243193) (0.999986775480,0.877862714764)
    (0.999988884520,0.883055021182) (0.999990734279,0.888258147136) (0.999992346109,0.893472077536) (0.999993740600,0.898696797548)
    (0.999994937556,0.903932292630) (0.999995955976,0.909178548573) (0.999996814028,0.914435551541) (0.999997529032,0.919703288107)
    (0.999998117433,0.924981745297) (0.999998594779,0.930270910630) (0.999998975705,0.935570772165) (0.999999273899,0.940881318541)
    (0.999999502093,0.946202539022) (0.999999672027,0.951534423544) (0.999999794437,0.956876962763) (0.999999879027,0.962230148099)
    (0.999999934444,0.967593971788) (0.999999968261,0.972968426932) (0.999999986948,0.978353507546) (0.999999995854,0.983749208613)
    (0.999999999178,0.989155526138) (0.999999999948,0.994572457197) (1.000000000000,1.000000000000)
}

\def\PathDataFour{%
    (-1.000000000000,-1.000000000000) (-0.994148485402,-1.000000000000) (-0.988308410719,-0.999999999996) (-0.982479775970,-0.999999999967)
    (-0.976662581207,-0.999999999863) (-0.970856826524,-0.999999999584) (-0.965062512073,-0.999999998969) (-0.959279638074,-0.999999997783)
    (-0.953508204829,-0.999999995698) (-0.947748212729,-0.999999992287) (-0.941999662266,-0.999999987001) (-0.936262554039,-0.999999979169)
    (-0.930536888763,-0.999999967973) (-0.924822667275,-0.999999952447) (-0.919119890542,-0.999999931459) (-0.913428559665,-0.999999903702)
    (-0.907748675881,-0.999999867686) (-0.902080240575,-0.999999821723) (-0.896423255272,-0.999999763923) (-0.890777721651,-0.999999692179)
    (-0.885143641539,-0.999999604163) (-0.879521016915,-0.999999497314) (-0.873909849914,-0.999999368832) (-0.868310142822,-0.999999215669)
    (-0.862721898079,-0.999999034521) (-0.857145118278,-0.999998821823) (-0.851579806165,-0.999998573742) (-0.846025964634,-0.999998286166)
    (-0.840483596728,-0.999997954706) (-0.834952705634,-0.999997574682) (-0.829433294682,-0.999997141123) (-0.823925367343,-0.999996648763)
    (-0.818428927221,-0.999996092028) (-0.812943978051,-0.999995465043) (-0.807470523697,-0.999994761618) (-0.802008568142,-0.999993975250)
    (-0.796558115490,-0.999993099118) (-0.791119169952,-0.999992126079) (-0.785691735849,-0.999991048666) (-0.780275817599,-0.999989859085)
    (-0.774871419716,-0.999988549212) (-0.769478546800,-0.999987110595) (-0.764097203532,-0.999985534444) (-0.758727394668,-0.999983811639)
    (-0.753369125029,-0.999981932721) (-0.748022399499,-0.999979887897) (-0.742687223011,-0.999977667035) (-0.737363600544,-0.999975259665)
    (-0.732051537115,-0.999972654980) (-0.726751037769,-0.999969841836) (-0.721462107574,-0.999966808749) (-0.716184751611,-0.999963543900)
    (-0.710918974965,-0.999960035133) (-0.705664782720,-0.999956269956) (-0.700422179949,-0.999952235546) (-0.695191171704,-0.999947918744)
    (-0.689971763010,-0.999943306062) (-0.684763958858,-0.999938383682) (-0.679567764192,-0.999933137462) (-0.674383183904,-0.999927552933)
    (-0.669210222825,-0.999921615303) (-0.664048885717,-0.999915309463) (-0.658899177263,-0.999908619987) (-0.653761102060,-0.999901531135)
    (-0.648634664609,-0.999894026857) (-0.643519869311,-0.999886090797) (-0.638416720454,-0.999877706295) (-0.633325222206,-0.999868856392)
    (-0.628245378609,-0.999859523833) (-0.623177193568,-0.999849691071) (-0.618120670847,-0.999839340274) (-0.613075814055,-0.999828453324)
    (-0.608042626645,-0.999817011828) (-0.603021111904,-0.999804997115) (-0.598011272941,-0.999792390248) (-0.593013112687,-0.999779172025)
    (-0.588026633882,-0.999765322984) (-0.583051839072,-0.999750823409) (-0.578088730597,-0.999735653336) (-0.573137310589,-0.999719792555)
    (-0.568197580963,-0.999703220620) (-0.563269543410,-0.999685916850) (-0.558353199392,-0.999667860339) (-0.553448550134,-0.999649029957)
    (-0.548555596619,-0.999629404360) (-0.543674339583,-0.999608961991) (-0.538804779507,-0.999587681093) (-0.533946916612,-0.999565539706)
    (-0.529100750854,-0.999542515680) (-0.524266281919,-0.999518586679) (-0.519443509217,-0.999493730186) (-0.514632431880,-0.999467923509)
    (-0.509833048752,-0.999441143789) (-0.505045358388,-0.999413368004) (-0.500269359050,-0.999384572978) (-0.495505048701,-0.999354735383)
    (-0.490752425003,-0.999323831750) (-0.486011485311,-0.999291838473) (-0.481282226671,-0.999258731813) (-0.476564645815,-0.999224487909)
    (-0.471858739161,-0.999189082781) (-0.467164502806,-0.999152492338) (-0.462481932525,-0.999114692380) (-0.457811023771,-0.999075658612)
    (-0.453151771668,-0.999035366644) (-0.448504171009,-0.998993791998) (-0.443868216260,-0.998950910118) (-0.439243901550,-0.998906696370)
    (-0.434631220675,-0.998861126055) (-0.430030167094,-0.998814174410) (-0.425440733929,-0.998765816615) (-0.420862913963,-0.998716027801)
    (-0.416296699638,-0.998664783056) (-0.411742083058,-0.998612057426) (-0.407199055983,-0.998557825928) (-0.402667609835,-0.998502063552)
    (-0.398147735690,-0.998444745265) (-0.393639424287,-0.998385846021) (-0.389142666020,-0.998325340765) (-0.384657450943,-0.998263204436)
    (-0.380183768768,-0.998199411976) (-0.375721608869,-0.998133938335) (-0.371270960277,-0.998066758474) (-0.366831811687,-0.997997847373)
    (-0.362404151456,-0.997927180034) (-0.357987967605,-0.997854731488) (-0.353583247818,-0.997780476801) (-0.349189979448,-0.997704391074)
    (-0.344808149515,-0.997626449454) (-0.340437744709,-0.997546627135) (-0.336078751391,-0.997464899364) (-0.331731155599,-0.997381241446)
    (-0.327394943043,-0.997295628748) (-0.323070099115,-0.997208036704) (-0.318756608885,-0.997118440819) (-0.314454457109,-0.997026816672)
    (-0.310163628227,-0.996933139924) (-0.305884106368,-0.996837386318) (-0.301615875354,-0.996739531686) (-0.297358918700,-0.996639551951)
    (-0.293113219620,-0.996537423132) (-0.288878761026,-0.996433121348) (-0.284655525539,-0.996326622819) (-0.280443495482,-0.996217903875)
    (-0.276242652892,-0.996106940952) (-0.272052979520,-0.995993710604) (-0.267874456833,-0.995878189498) (-0.263707066022,-0.995760354423)
    (-0.259550788002,-0.995640182292) (-0.255405603415,-0.995517650141) (-0.251271492639,-0.995392735139) (-0.247148435788,-0.995265414583)
    (-0.243036412715,-0.995135665908) (-0.238935403018,-0.995003466684) (-0.234845386045,-0.994868794621) (-0.230766340897,-0.994731627574)
    (-0.226698246430,-0.994591943538) (-0.222641081262,-0.994449720658) (-0.218594823777,-0.994304937229) (-0.214559452128,-0.994157571693)
    (-0.210534944242,-0.994007602650) (-0.206521277824,-0.993855008851) (-0.202518430361,-0.993699769207) (-0.198526379129,-0.993541862786)
    (-0.194545101194,-0.993381268816) (-0.190574573416,-0.993217966688) (-0.186614772457,-0.993051935957) (-0.182665674783,-0.992883156340)
    (-0.178727256669,-0.992711607723) (-0.174799494203,-0.992537270157) (-0.170882363291,-0.992360123865) (-0.166975839662,-0.992180149236)
    (-0.163079898869,-0.991997326830) (-0.159194516298,-0.991811637382) (-0.155319667172,-0.991623061796) (-0.151455326550,-0.991431581151)
    (-0.147601469338,-0.991237176699) (-0.143758070291,-0.991039829867) (-0.139925104014,-0.990839522258) (-0.136102544974,-0.990636235649)
    (-0.132290367496,-0.990429951996) (-0.128488545773,-0.990220653429) (-0.124697053867,-0.990008322256) (-0.120915865716,-0.989792940962)
    (-0.117144955136,-0.989574492210) (-0.113384295828,-0.989352958841) (-0.109633861378,-0.989128323871) (-0.105893625264,-0.988900570496)
    (-0.102163560862,-0.988669682090) (-0.098443641446,-0.988435642202) (-0.094733840193,-0.988198434562) (-0.091034130190,-0.987958043074)
    (-0.087344484436,-0.987714451821) (-0.083664875843,-0.987467645062) (-0.079995277247,-0.987217607232) (-0.072686001004,-0.986707776982)
    (-0.065416436927,-0.986184840068) (-0.058186365200,-0.985648678278) (-0.050995565094,-0.985099176178) (-0.043843815081,-0.984536221086)
    (-0.036730892947,-0.983959703054) (-0.029656575907,-0.983369514836) (-0.022620640710,-0.982765551867) (-0.015622863745,-0.982147712231)
    (-0.008663021143,-0.981515896629) (-0.001740888879,-0.980870008348) (0.005143757136,-0.980209953225) (0.011991140957,-0.979535639614)
    (0.018801486519,-0.978846978345) (0.025575017547,-0.978143882690) (0.032311957472,-0.977426268323) (0.039012529351,-0.976694053279)
    (0.045676955786,-0.975947157912) (0.052305458852,-0.975185504857) (0.058898260022,-0.974409018984) (0.065455580100,-0.973617627358)
    (0.071977639150,-0.972811259196) (0.078464656436,-0.971989845820) (0.084916850359,-0.971153320618) (0.091334438400,-0.970301618994)
    (0.097717637060,-0.969434678331) (0.104066661813,-0.968552437941) (0.110381727049,-0.967654839021) (0.116663046031,-0.966741824612)
    (0.122910830846,-0.965813339551) (0.129125292362,-0.964869330431) (0.135306640187,-0.963909745553) (0.141455082632,-0.962934534882)
    (0.147570826671,-0.961943650009) (0.153654077909,-0.960937044101) (0.159705040547,-0.959914671863) (0.165723917355,-0.958876489491)
    (0.171710909641,-0.957822454636) (0.177666217224,-0.956752526355) (0.183590038410,-0.955666665077) (0.189482569970,-0.954564832555)
    (0.195344007117,-0.953446991832) (0.201174543485,-0.952313107196) (0.206974371115,-0.951163144144) (0.212743680435,-0.949997069343)
    (0.218482660245,-0.948814850589) (0.224191497705,-0.947616456773) (0.229870378322,-0.946401857841) (0.235519485937,-0.945171024758)
    (0.241139002719,-0.943923929474) (0.246729109153,-0.942660544886) (0.252289984035,-0.941380844804) (0.257821804463,-0.940084803917)
    (0.263324745834,-0.938772397760) (0.268798981841,-0.937443602680) (0.274244684467,-0.936098395802) (0.279662023983,-0.934736755001)
    (0.285051168947,-0.933358658867) (0.290412286207,-0.931964086678) (0.295745540895,-0.930553018366) (0.301051096432,-0.929125434489)
    (0.306329114531,-0.927681316202) (0.311579755195,-0.926220645230) (0.316803176724,-0.924743403837) (0.321999535718,-0.923249574803)
    (0.327168987080,-0.921739141392) (0.332311684023,-0.920212087331) (0.337427778075,-0.918668396780) (0.342517419082,-0.917108054310)
    (0.347580755219,-0.915531044876) (0.352617932996,-0.913937353794) (0.357629097263,-0.912326966719) (0.362614391219,-0.910699869619)
    (0.367573956421,-0.909056048753) (0.372507932793,-0.907395490649) (0.377416458632,-0.905718182084) (0.382299670620,-0.904024110061)
    (0.387157703834,-0.902313261789) (0.391990691751,-0.900585624660) (0.396798766265,-0.898841186234) (0.401582057690,-0.897079934215)
    (0.406340694777,-0.895301856434) (0.411074804719,-0.893506940831) (0.415784513167,-0.891695175435) (0.420469944237,-0.889866548345)
    (0.425131220522,-0.888021047717) (0.429768463105,-0.886158661741) (0.434381791570,-0.884279378631) (0.438971324010,-0.882383186602)
    (0.443537177046,-0.880470073856) (0.448079465830,-0.878540028568) (0.452598304063,-0.876593038869) (0.457093804004,-0.874629092831)
    (0.461566076484,-0.872648178453) (0.466015230915,-0.870650283644) (0.470441375306,-0.868635396211) (0.474844616272,-0.866603503844)
    (0.479225059045,-0.864554594104) (0.483582807491,-0.862488654407) (0.487917964116,-0.860405672011) (0.492230630084,-0.858305634006)
    (0.496520905225,-0.856188527296) (0.500788888047,-0.854054338594) (0.505034675752,-0.851903054400) (0.509258364243,-0.849734660999)
    (0.513460048140,-0.847549144443) (0.517639820788,-0.845346490539) (0.521797774275,-0.843126684841) (0.525933999436,-0.840889712637)
    (0.530048585872,-0.838635558937) (0.534141621958,-0.836364208465) (0.538213194855,-0.834075645644) (0.542263390521,-0.831769854589)
    (0.546292293728,-0.829446819095) (0.550299988064,-0.827106522626) (0.554286555953,-0.824748948309) (0.558252078663,-0.822374078917)
    (0.562196636318,-0.819981896867) (0.566120307908,-0.817572384201) (0.570023171300,-0.815145522588) (0.573905303253,-0.812701293303)
    (0.577766779425,-0.810239677226) (0.581607674385,-0.807760654828) (0.585428061625,-0.805264206163) (0.589228013569,-0.802750310861)
    (0.593007601587,-0.800218948116) (0.596766896001,-0.797670096678) (0.600505966100,-0.795103734846) (0.604224880148,-0.792519840456)
    (0.607923705396,-0.789918390875) (0.611602508088,-0.787299362991) (0.615261353481,-0.784662733206) (0.618900305842,-0.782008477425)
    (0.622519428471,-0.779336571049) (0.626118783701,-0.776646988967) (0.629698432915,-0.773939705547) (0.633258436551,-0.771214694628)
    (0.636798854114,-0.768471929511) (0.640319744187,-0.765711382952) (0.643821164438,-0.762933027153) (0.647303171631,-0.760136833752)
    (0.650765821635,-0.757322773820) (0.654209169434,-0.754490817847) (0.657633269138,-0.751640935737) (0.661038173986,-0.748773096799)
    (0.664423936364,-0.745887269741) (0.667790607808,-0.742983422657) (0.671138239014,-0.740061523023) (0.674466879851,-0.737121537690)
    (0.677776579362,-0.734163432870) (0.681067385783,-0.731187174134) (0.684339346543,-0.728192726400) (0.687592508279,-0.725180053928)
    (0.690826916841,-0.722149120308) (0.694042617301,-0.719099888455) (0.697239653964,-0.716032320599) (0.700418070376,-0.712946378277)
    (0.703577909328,-0.709842022326) (0.706719212873,-0.706719212873) (0.709842022326,-0.703577909328) (0.712946378277,-0.700418070376)
    (0.716032320599,-0.697239653964) (0.719099888455,-0.694042617301) (0.722149120308,-0.690826916841) (0.725180053928,-0.687592508279)
    (0.728192726400,-0.684339346543) (0.731187174134,-0.681067385783) (0.734163432870,-0.677776579362) (0.737121537690,-0.674466879851)
    (0.740061523023,-0.671138239014) (0.742983422657,-0.667790607808) (0.745887269741,-0.664423936364) (0.748773096799,-0.661038173986)
    (0.751640935737,-0.657633269138) (0.754490817847,-0.654209169434) (0.757322773820,-0.650765821635) (0.760136833752,-0.647303171631)
    (0.762933027153,-0.643821164438) (0.765711382952,-0.640319744187) (0.768471929511,-0.636798854114) (0.771214694628,-0.633258436551)
    (0.773939705547,-0.629698432915) (0.776646988967,-0.626118783701) (0.779336571049,-0.622519428471) (0.782008477425,-0.618900305842)
    (0.784662733206,-0.615261353481) (0.787299362991,-0.611602508088) (0.789918390875,-0.607923705396) (0.792519840456,-0.604224880148)
    (0.795103734846,-0.600505966100) (0.797670096678,-0.596766896001) (0.800218948116,-0.593007601587) (0.802750310861,-0.589228013569)
    (0.805264206163,-0.585428061625) (0.807760654828,-0.581607674385) (0.810239677226,-0.577766779425) (0.812701293303,-0.573905303253)
    (0.815145522588,-0.570023171300) (0.817572384201,-0.566120307908) (0.819981896867,-0.562196636318) (0.822374078917,-0.558252078663)
    (0.824748948309,-0.554286555953) (0.827106522626,-0.550299988064) (0.829446819095,-0.546292293728) (0.831769854589,-0.542263390521)
    (0.834075645644,-0.538213194855) (0.836364208465,-0.534141621958) (0.838635558937,-0.530048585872) (0.840889712637,-0.525933999436)
    (0.843126684841,-0.521797774275) (0.845346490539,-0.517639820788) (0.847549144443,-0.513460048140) (0.849734660999,-0.509258364243)
    (0.851903054400,-0.505034675752) (0.854054338594,-0.500788888047) (0.856188527296,-0.496520905225) (0.858305634006,-0.492230630084)
    (0.860405672011,-0.487917964116) (0.862488654407,-0.483582807491) (0.864554594104,-0.479225059045) (0.866603503844,-0.474844616272)
    (0.868635396211,-0.470441375306) (0.870650283644,-0.466015230915) (0.872648178453,-0.461566076484) (0.874629092831,-0.457093804004)
    (0.876593038869,-0.452598304063) (0.878540028568,-0.448079465830) (0.880470073856,-0.443537177046) (0.882383186602,-0.438971324010)
    (0.884279378631,-0.434381791570) (0.886158661741,-0.429768463105) (0.888021047717,-0.425131220522) (0.889866548345,-0.420469944237)
    (0.891695175435,-0.415784513167) (0.893506940831,-0.411074804719) (0.895301856434,-0.406340694777) (0.897079934215,-0.401582057690)
    (0.898841186234,-0.396798766265) (0.900585624660,-0.391990691751) (0.902313261789,-0.387157703834) (0.904024110061,-0.382299670620)
    (0.905718182084,-0.377416458632) (0.907395490649,-0.372507932793) (0.909056048753,-0.367573956421) (0.910699869619,-0.362614391219)
    (0.912326966719,-0.357629097263) (0.913937353794,-0.352617932996) (0.915531044876,-0.347580755219) (0.917108054310,-0.342517419082)
    (0.918668396780,-0.337427778075) (0.920212087331,-0.332311684023) (0.921739141392,-0.327168987080) (0.923249574803,-0.321999535718)
    (0.924743403837,-0.316803176724) (0.926220645230,-0.311579755195) (0.927681316202,-0.306329114531) (0.929125434489,-0.301051096432)
    (0.930553018366,-0.295745540895) (0.931964086678,-0.290412286207) (0.933358658867,-0.285051168947) (0.934736755001,-0.279662023983)
    (0.936098395802,-0.274244684467) (0.937443602680,-0.268798981841) (0.938772397760,-0.263324745834) (0.940084803917,-0.257821804463)
    (0.941380844804,-0.252289984035) (0.942660544886,-0.246729109153) (0.943923929474,-0.241139002719) (0.945171024758,-0.235519485937)
    (0.946401857841,-0.229870378322) (0.947616456773,-0.224191497705) (0.948814850589,-0.218482660245) (0.949997069343,-0.212743680435)
    (0.951163144144,-0.206974371115) (0.952313107196,-0.201174543485) (0.953446991832,-0.195344007117) (0.954564832555,-0.189482569970)
    (0.955666665077,-0.183590038410) (0.956752526355,-0.177666217224) (0.957822454636,-0.171710909641) (0.958876489491,-0.165723917355)
    (0.959914671863,-0.159705040547) (0.960937044101,-0.153654077909) (0.961943650009,-0.147570826671) (0.962934534882,-0.141455082632)
    (0.963909745553,-0.135306640187) (0.964869330431,-0.129125292362) (0.965813339551,-0.122910830846) (0.966741824612,-0.116663046031)
    (0.967654839021,-0.110381727049) (0.968552437941,-0.104066661813) (0.969434678331,-0.097717637060) (0.970301618994,-0.091334438400)
    (0.971153320618,-0.084916850359) (0.971989845820,-0.078464656436) (0.972811259196,-0.071977639150) (0.973617627358,-0.065455580100)
    (0.974409018984,-0.058898260022) (0.975185504857,-0.052305458852) (0.975947157912,-0.045676955786) (0.976694053279,-0.039012529351)
    (0.977426268323,-0.032311957472) (0.978143882690,-0.025575017547) (0.978846978345,-0.018801486519) (0.979535639614,-0.011991140957)
    (0.980209953225,-0.005143757136) (0.980870008348,0.001740888879) (0.981515896629,0.008663021143) (0.982147712231,0.015622863745)
    (0.982765551867,0.022620640710) (0.983369514836,0.029656575907) (0.983959703054,0.036730892947) (0.984536221086,0.043843815081)
    (0.985099176178,0.050995565094) (0.985648678278,0.058186365200) (0.986184840068,0.065416436927) (0.986707776982,0.072686001004)
    (0.987217607232,0.079995277247) (0.987467645062,0.083664875843) (0.987714451821,0.087344484436) (0.987958043074,0.091034130190)
    (0.988198434562,0.094733840193) (0.988435642202,0.098443641446) (0.988669682090,0.102163560862) (0.988900570496,0.105893625264)
    (0.989128323871,0.109633861378) (0.989352958841,0.113384295828) (0.989574492210,0.117144955136) (0.989792940962,0.120915865716)
    (0.990008322256,0.124697053867) (0.990220653429,0.128488545773) (0.990429951996,0.132290367496) (0.990636235649,0.136102544974)
    (0.990839522258,0.139925104014) (0.991039829867,0.143758070291) (0.991237176699,0.147601469338) (0.991431581151,0.151455326550)
    (0.991623061796,0.155319667172) (0.991811637382,0.159194516298) (0.991997326830,0.163079898869) (0.992180149236,0.166975839662)
    (0.992360123865,0.170882363291) (0.992537270157,0.174799494203) (0.992711607723,0.178727256669) (0.992883156340,0.182665674783)
    (0.993051935957,0.186614772457) (0.993217966688,0.190574573416) (0.993381268816,0.194545101194) (0.993541862786,0.198526379129)
    (0.993699769207,0.202518430361) (0.993855008851,0.206521277824) (0.994007602650,0.210534944242) (0.994157571693,0.214559452128)
    (0.994304937229,0.218594823777) (0.994449720658,0.222641081262) (0.994591943538,0.226698246430) (0.994731627574,0.230766340897)
    (0.994868794621,0.234845386045) (0.995003466684,0.238935403018) (0.995135665908,0.243036412715) (0.995265414583,0.247148435788)
    (0.995392735139,0.251271492639) (0.995517650141,0.255405603415) (0.995640182292,0.259550788002) (0.995760354423,0.263707066022)
    (0.995878189498,0.267874456833) (0.995993710604,0.272052979520) (0.996106940952,0.276242652892) (0.996217903875,0.280443495482)
    (0.996326622819,0.284655525539) (0.996433121348,0.288878761026) (0.996537423132,0.293113219620) (0.996639551951,0.297358918700)
    (0.996739531686,0.301615875354) (0.996837386318,0.305884106368) (0.996933139924,0.310163628227) (0.997026816672,0.314454457109)
    (0.997118440819,0.318756608885) (0.997208036704,0.323070099115) (0.997295628748,0.327394943043) (0.997381241446,0.331731155599)
    (0.997464899364,0.336078751391) (0.997546627135,0.340437744709) (0.997626449454,0.344808149515) (0.997704391074,0.349189979448)
    (0.997780476801,0.353583247818) (0.997854731488,0.357987967605) (0.997927180034,0.362404151456) (0.997997847373,0.366831811687)
    (0.998066758474,0.371270960277) (0.998133938335,0.375721608869) (0.998199411976,0.380183768768) (0.998263204436,0.384657450943)
    (0.998325340765,0.389142666020) (0.998385846021,0.393639424287) (0.998444745265,0.398147735690) (0.998502063552,0.402667609835)
    (0.998557825928,0.407199055983) (0.998612057426,0.411742083058) (0.998664783056,0.416296699638) (0.998716027801,0.420862913963)
    (0.998765816615,0.425440733929) (0.998814174410,0.430030167094) (0.998861126055,0.434631220675) (0.998906696370,0.439243901550)
    (0.998950910118,0.443868216260) (0.998993791998,0.448504171009) (0.999035366644,0.453151771668) (0.999075658612,0.457811023771)
    (0.999114692380,0.462481932525) (0.999152492338,0.467164502806) (0.999189082781,0.471858739161) (0.999224487909,0.476564645815)
    (0.999258731813,0.481282226671) (0.999291838473,0.486011485311) (0.999323831750,0.490752425003) (0.999354735383,0.495505048701)
    (0.999384572978,0.500269359050) (0.999413368004,0.505045358388) (0.999441143789,0.509833048752) (0.999467923509,0.514632431880)
    (0.999493730186,0.519443509217) (0.999518586679,0.524266281919) (0.999542515680,0.529100750854) (0.999565539706,0.533946916612)
    (0.999587681093,0.538804779507) (0.999608961991,0.543674339583) (0.999629404360,0.548555596619) (0.999649029957,0.553448550134)
    (0.999667860339,0.558353199392) (0.999685916850,0.563269543410) (0.999703220620,0.568197580963) (0.999719792555,0.573137310589)
    (0.999735653336,0.578088730597) (0.999750823409,0.583051839072) (0.999765322984,0.588026633882) (0.999779172025,0.593013112687)
    (0.999792390248,0.598011272941) (0.999804997115,0.603021111904) (0.999817011828,0.608042626645) (0.999828453324,0.613075814055)
    (0.999839340274,0.618120670847) (0.999849691071,0.623177193568) (0.999859523833,0.628245378609) (0.999868856392,0.633325222206)
    (0.999877706295,0.638416720454) (0.999886090797,0.643519869311) (0.999894026857,0.648634664609) (0.999901531135,0.653761102060)
    (0.999908619987,0.658899177263) (0.999915309463,0.664048885717) (0.999921615303,0.669210222825) (0.999927552933,0.674383183904)
    (0.999933137462,0.679567764192) (0.999938383682,0.684763958858) (0.999943306062,0.689971763010) (0.999947918744,0.695191171704)
    (0.999952235546,0.700422179949) (0.999956269956,0.705664782720) (0.999960035133,0.710918974965) (0.999963543900,0.716184751611)
    (0.999966808749,0.721462107574) (0.999969841836,0.726751037769) (0.999972654980,0.732051537115) (0.999975259665,0.737363600544)
    (0.999977667035,0.742687223011) (0.999979887897,0.748022399499) (0.999981932721,0.753369125029) (0.999983811639,0.758727394668)
    (0.999985534444,0.764097203532) (0.999987110595,0.769478546800) (0.999988549212,0.774871419716) (0.999989859085,0.780275817599)
    (0.999991048666,0.785691735849) (0.999992126079,0.791119169952) (0.999993099118,0.796558115490) (0.999993975250,0.802008568142)
    (0.999994761618,0.807470523697) (0.999995465043,0.812943978051) (0.999996092028,0.818428927221) (0.999996648763,0.823925367343)
    (0.999997141123,0.829433294682) (0.999997574682,0.834952705634) (0.999997954706,0.840483596728) (0.999998286166,0.846025964634)
    (0.999998573742,0.851579806165) (0.999998821823,0.857145118278) (0.999999034521,0.862721898079) (0.999999215669,0.868310142822)
    (0.999999368832,0.873909849914) (0.999999497314,0.879521016915) (0.999999604163,0.885143641539) (0.999999692179,0.890777721651)
    (0.999999763923,0.896423255272) (0.999999821723,0.902080240575) (0.999999867686,0.907748675881) (0.999999903702,0.913428559665)
    (0.999999931459,0.919119890542) (0.999999952447,0.924822667275) (0.999999967973,0.930536888763) (0.999999979169,0.936262554039)
    (0.999999987001,0.941999662266) (0.999999992287,0.947748212729) (0.999999995698,0.953508204829) (0.999999997783,0.959279638074)
    (0.999999998969,0.965062512073) (0.999999999584,0.970856826524) (0.999999999863,0.976662581207) (0.999999999967,0.982479775970)
    (0.999999999996,0.988308410719) (1.000000000000,0.994148485402) (1.000000000000,1.000000000000)
}

\def\PathDataFive{%
    (-1.000000000000,-1.000000000000) (-0.993851506673,-1.000000000000) (-0.987715033860,-1.000000000000) (-0.981590581563,-1.000000000000)
    (-0.975478149781,-0.999999999998) (-0.969377738517,-0.999999999994) (-0.963289347776,-0.999999999982) (-0.957212977564,-0.999999999954)
    (-0.951148627892,-0.999999999899) (-0.945096298780,-0.999999999796) (-0.939055990252,-0.999999999618) (-0.933027702342,-0.999999999327)
    (-0.927011435095,-0.999999998873) (-0.921007188569,-0.999999998188) (-0.915014962836,-0.999999997191) (-0.909034757986,-0.999999995775)
    (-0.903066574126,-0.999999993814) (-0.897110411386,-0.999999991154) (-0.891166269917,-0.999999987609) (-0.885234149896,-0.999999982962)
    (-0.879314051527,-0.999999976960) (-0.873405975043,-0.999999969309) (-0.867509920708,-0.999999959669) (-0.861625888820,-0.999999947656)
    (-0.855753879712,-0.999999932832) (-0.849893893753,-0.999999914704) (-0.844045931353,-0.999999892720) (-0.838209992960,-0.999999866264)
    (-0.832386079067,-0.999999834653) (-0.826574190209,-0.999999797129) (-0.820774326968,-0.999999752863) (-0.814986489970,-0.999999700941)
    (-0.809210679892,-0.999999640366) (-0.803446897459,-0.999999570053) (-0.797695143446,-0.999999488821) (-0.791955418679,-0.999999395393)
    (-0.786227724036,-0.999999288390) (-0.780512060449,-0.999999166324) (-0.774808428900,-0.999999027599) (-0.769116830427,-0.999998870502)
    (-0.763437266120,-0.999998693201) (-0.757769737126,-0.999998493739) (-0.752114244641,-0.999998270031) (-0.746470789918,-0.999998019861)
    (-0.740839374261,-0.999997740873) (-0.735219999030,-0.999997430573) (-0.729612665633,-0.999997086320) (-0.724017375532,-0.999996705324)
    (-0.718434130238,-0.999996284641) (-0.712862931313,-0.999995821172) (-0.707303780366,-0.999995311653) (-0.701756679051,-0.999994752658)
    (-0.696221629071,-0.999994140589) (-0.690698632169,-0.999993471677) (-0.685187690131,-0.999992741976) (-0.679688804785,-0.999991947360)
    (-0.674201977992,-0.999991083519) (-0.668727211652,-0.999990145956) (-0.663264507697,-0.999989129984) (-0.657813868091,-0.999988030722)
    (-0.652375294823,-0.999986843092) (-0.646948789909,-0.999985561815) (-0.641534355387,-0.999984181410) (-0.636131993315,-0.999982696192)
    (-0.630741705767,-0.999981100264) (-0.625363494828,-0.999979387521) (-0.619997362595,-0.999977551641) (-0.614643311170,-0.999975586087)
    (-0.609301342658,-0.999973484105) (-0.603971459161,-0.999971238718) (-0.598653662778,-0.999968842727) (-0.593347955599,-0.999966288706)
    (-0.588054339700,-0.999963569006) (-0.582772817142,-0.999960675745) (-0.577503389963,-0.999957600814) (-0.572246060177,-0.999954335870)
    (-0.567000829768,-0.999950872337) (-0.561767700686,-0.999947201406) (-0.556546674843,-0.999943314029) (-0.551337754108,-0.999939200925)
    (-0.546140940303,-0.999934852572) (-0.540956235198,-0.999930259211) (-0.535783640506,-0.999925410843) (-0.530623157879,-0.999920297229)
    (-0.525474788903,-0.999914907892) (-0.520338535093,-0.999909232111) (-0.515214397888,-0.999903258926) (-0.510102378649,-0.999896977136)
    (-0.505002478648,-0.999890375299) (-0.499914699070,-0.999883441732) (-0.494839041003,-0.999876164514) (-0.489775505436,-0.999868531480)
    (-0.484724093255,-0.999860530229) (-0.479684805233,-0.999852148119) (-0.474657642030,-0.999843372273) (-0.469642604188,-0.999834189573)
    (-0.464639692123,-0.999824586669) (-0.459648906124,-0.999814549974) (-0.454670246343,-0.999804065667) (-0.449703712797,-0.999793119696)
    (-0.444749305358,-0.999781697777) (-0.439807023748,-0.999769785399) (-0.434876867541,-0.999757367820) (-0.429958836149,-0.999744430076)
    (-0.425052928825,-0.999730956978) (-0.420159144655,-0.999716933113) (-0.415277482553,-0.999702342851) (-0.410407941260,-0.999687170345)
    (-0.405550519334,-0.999671399530) (-0.400705215153,-0.999655014130) (-0.395872026903,-0.999637997660) (-0.391050952581,-0.999620333425)
    (-0.386241989986,-0.999602004527) (-0.381445136716,-0.999582993864) (-0.376660390167,-0.999563284138) (-0.371887747526,-0.999542857851)
    (-0.367127205767,-0.999521697313) (-0.362378761650,-0.999499784645) (-0.357642411716,-0.999477101779) (-0.352918152282,-0.999453630465)
    (-0.348205979442,-0.999429352272) (-0.343505889059,-0.999404248591) (-0.338817876764,-0.999378300641) (-0.334141937952,-0.999351489469)
    (-0.329478067783,-0.999323795958) (-0.324826261172,-0.999295200824) (-0.320186512793,-0.999265684629) (-0.315558817073,-0.999235227777)
    (-0.310943168189,-0.999203810518) (-0.306339560067,-0.999171412959) (-0.301747986381,-0.999138015060) (-0.297168440548,-0.999103596641)
    (-0.292600915727,-0.999068137386) (-0.288045404818,-0.999031616849) (-0.283501900458,-0.998994014454) (-0.278970395023,-0.998955309502)
    (-0.274450880622,-0.998915481172) (-0.269943349099,-0.998874508530) (-0.265447792029,-0.998832370529) (-0.260964200720,-0.998789046015)
    (-0.256492566208,-0.998744513730) (-0.252032879259,-0.998698752318) (-0.247585130368,-0.998651740327) (-0.243149309754,-0.998603456215)
    (-0.238725407368,-0.998553878352) (-0.234313412881,-0.998502985028) (-0.229913315695,-0.998450754454) (-0.225525104935,-0.998397164766)
    (-0.221148769450,-0.998342194031) (-0.216784297819,-0.998285820252) (-0.212431678340,-0.998228021369) (-0.208090899042,-0.998168775266)
    (-0.203761947676,-0.998108059772) (-0.199444811720,-0.998045852671) (-0.195139478381,-0.997982131699) (-0.190845934590,-0.997916874553)
    (-0.186564167008,-0.997850058894) (-0.182294162023,-0.997781662349) (-0.178035905754,-0.997711662519) (-0.173789384051,-0.997640036980)
    (-0.169554582495,-0.997566763288) (-0.165331486399,-0.997491818980) (-0.161120080812,-0.997415181585) (-0.156920350519,-0.997336828622)
    (-0.152732280041,-0.997256737603) (-0.148555853638,-0.997174886044) (-0.144391055311,-0.997091251459) (-0.140237868803,-0.997005811373)
    (-0.136096277602,-0.996918543319) (-0.131966264940,-0.996829424845) (-0.127847813800,-0.996738433518) (-0.123740906912,-0.996645546923)
    (-0.119645526762,-0.996550742674) (-0.115561655587,-0.996453998410) (-0.111489275384,-0.996355291804) (-0.107428367907,-0.996254600563)
    (-0.103378914674,-0.996151902433) (-0.099340896965,-0.996047175202) (-0.095314295828,-0.995940396702) (-0.091299092082,-0.995831544814)
    (-0.087295266316,-0.995720597469) (-0.083302798895,-0.995607532655) (-0.079321669963,-0.995492328414) (-0.075351859443,-0.995374962850)
    (-0.071393347043,-0.995255414130) (-0.067446112257,-0.995133660487) (-0.063510134369,-0.995009680223) (-0.059585392456,-0.994883451709)
    (-0.055671865389,-0.994754953393) (-0.051769531842,-0.994624163798) (-0.047878370287,-0.994491061528) (-0.043998359004,-0.994355625265)
    (-0.040129476081,-0.994217833779) (-0.036271699419,-0.994077665924) (-0.032425006731,-0.993935100642) (-0.028589375553,-0.993790116969)
    (-0.024764783241,-0.993642694032) (-0.020951206975,-0.993492811052) (-0.017148623767,-0.993340447349) (-0.013357010458,-0.993185582342)
    (-0.009576343726,-0.993028195549) (-0.005806600090,-0.992868266595) (-0.002047755908,-0.992705775206) (0.001700212613,-0.992540701216)
    (0.005437329417,-0.992373024567) (0.009163618595,-0.992202725311) (0.012879104383,-0.992029783611) (0.016583811155,-0.991854179743)
    (0.020277763421,-0.991675894096) (0.023960985826,-0.991494907178) (0.027633503141,-0.991311199612) (0.031295340267,-0.991124752137)
    (0.034946522223,-0.990935545617) (0.038587074150,-0.990743561031) (0.042217021303,-0.990548779484) (0.045836389049,-0.990351182201)
    (0.049445202863,-0.990150750533) (0.053043488324,-0.989947465956) (0.056631271116,-0.989741310068) (0.060208577016,-0.989532264599)
    (0.063775431899,-0.989320311402) (0.067331861729,-0.989105432461) (0.070877892560,-0.988887609887) (0.074413550528,-0.988666825920)
    (0.077938861852,-0.988443062933) (0.081453852825,-0.988216303426) (0.084958549820,-0.987986530034) (0.088452979275,-0.987753725519)
    (0.091937167701,-0.987517872778) (0.095411141669,-0.987278954841) (0.098874927814,-0.987036954868) (0.102328552827,-0.986791856154)
    (0.105772043457,-0.986543642126) (0.109205426502,-0.986292296345) (0.112628728809,-0.986037802505) (0.116041977271,-0.985780144435)
    (0.119445198823,-0.985519306097) (0.126221669133,-0.984988025128) (0.132958355957,-0.984443833970) (0.139655475998,-0.983886609127)
    (0.146313246357,-0.983316229240) (0.152931884439,-0.982732575078) (0.159511607866,-0.982135529528) (0.166052634392,-0.981524977586)
    (0.172555181816,-0.980900806346) (0.179019467904,-0.980262904983) (0.185445710308,-0.979611164741) (0.191834126487,-0.978945478913)
    (0.198184933637,-0.978265742826) (0.204498348616,-0.977571853815) (0.210774587871,-0.976863711206) (0.217013867377,-0.976141216293)
    (0.223216402565,-0.975404272313) (0.229382408264,-0.974652784423) (0.235512098636,-0.973886659670) (0.241605687121,-0.973105806968)
    (0.247663386377,-0.972310137071) (0.253685408230,-0.971499562540) (0.259671963618,-0.970673997720) (0.265623262546,-0.969833358704)
    (0.271539514034,-0.968977563306) (0.277420926072,-0.968106531033) (0.283267705580,-0.967220183046) (0.289080058362,-0.966318442136)
    (0.294858189071,-0.965401232686) (0.300602301167,-0.964468480644) (0.306312596887,-0.963520113486) (0.311989277203,-0.962556060184)
    (0.317632541800,-0.961576251175) (0.323242589037,-0.960580618327) (0.328819615924,-0.959569094904) (0.334363818093,-0.958541615535)
    (0.339875389775,-0.957498116179) (0.345354523773,-0.956438534093) (0.350801411442,-0.955362807799) (0.356216242669,-0.954270877049)
    (0.361599205854,-0.953162682793) (0.366950487891,-0.952038167147) (0.372270274151,-0.950897273360) (0.377558748472,-0.949739945780)
    (0.382816093139,-0.948566129822) (0.388042488875,-0.947375771937) (0.393238114831,-0.946168819581) (0.398403148571,-0.944945221180)
    (0.403537766071,-0.943704926099) (0.408642141701,-0.942447884616) (0.413716448228,-0.941174047885) (0.418760856803,-0.939883367908)
    (0.423775536961,-0.938575797505) (0.428760656612,-0.937251290285) (0.433716382043,-0.935909800612) (0.438642877913,-0.934551283581)
    (0.443540307253,-0.933175694989) (0.448408831462,-0.931782991301) (0.453248610313,-0.930373129626) (0.458059801948,-0.928946067691)
    (0.462842562882,-0.927501763807) (0.467597048007,-0.926040176849) (0.472323410591,-0.924561266225) (0.477021802284,-0.923064991850)
    (0.481692373122,-0.921551314120) (0.486335271529,-0.920020193889) (0.490950644326,-0.918471592440) (0.495538636732,-0.916905471461)
    (0.500099392375,-0.915321793022) (0.504633053293,-0.913720519548) (0.509139759944,-0.912101613800) (0.513619651215,-0.910465038843)
    (0.518072864426,-0.908810758033) (0.522499535337,-0.907138734986) (0.526899798163,-0.905448933558) (0.531273785575,-0.903741317825)
    (0.535621628714,-0.902015852058) (0.539943457197,-0.900272500703) (0.544239399130,-0.898511228360) (0.548509581113,-0.896731999761)
    (0.552754128256,-0.894934779749) (0.556973164182,-0.893119533261) (0.561166811044,-0.891286225304) (0.565335189531,-0.889434820936)
    (0.569478418883,-0.887565285249) (0.573596616897,-0.885677583347) (0.577689899942,-0.883771680329) (0.581758382967,-0.881847541269)
    (0.585802179516,-0.879905131200) (0.589821401738,-0.877944415092) (0.593816160396,-0.875965357838) (0.597786564883,-0.873967924234)
    (0.601732723231,-0.871952078963) (0.605654742123,-0.869917786578) (0.609552726905,-0.867865011482) (0.613426781601,-0.865793717917)
    (0.617277008921,-0.863703869941) (0.621103510274,-0.861595431418) (0.624906385783,-0.859468365997) (0.628685734293,-0.857322637099)
    (0.632441653386,-0.855158207901) (0.636174239394,-0.852975041319) (0.639883587410,-0.850773099995) (0.643569791298,-0.848552346279)
    (0.647232943710,-0.846312742217) (0.650873136095,-0.844054249535) (0.654490458714,-0.841776829624) (0.658085000649,-0.839480443525)
    (0.661656849818,-0.837165051917) (0.665206092988,-0.834830615100) (0.668732815783,-0.832477092983) (0.672237102704,-0.830104445070)
    (0.675719037132,-0.827712630444) (0.679178701349,-0.825301607756) (0.682616176544,-0.822871335210) (0.686031542830,-0.820421770549)
    (0.689424879252,-0.817952871045) (0.692796263804,-0.815464593481) (0.696145773438,-0.812956894141) (0.699473484075,-0.810429728795)
    (0.702779470623,-0.807883052690) (0.706063806984,-0.805316820530) (0.709326566068,-0.802730986471) (0.712567819804,-0.800125504102)
    (0.715787639156,-0.797500326438) (0.718986094131,-0.794855405900) (0.722163253792,-0.792190694312) (0.725319186271,-0.789506142879)
    (0.728453958784,-0.786801702180) (0.731567637637,-0.784077322156) (0.734660288241,-0.781332952095) (0.737731975127,-0.778568540619)
    (0.740782761953,-0.775784035676) (0.743812711521,-0.772979384525) (0.746821885784,-0.770154533723) (0.749810345863,-0.767309429114)
    (0.752778152055,-0.764444015816) (0.755725363848,-0.761558238212) (0.758652039933,-0.758652039933) (0.761558238212,-0.755725363848)
    (0.764444015816,-0.752778152055) (0.767309429114,-0.749810345863) (0.770154533723,-0.746821885784) (0.772979384525,-0.743812711521)
    (0.775784035676,-0.740782761953) (0.778568540619,-0.737731975127) (0.781332952095,-0.734660288241) (0.784077322156,-0.731567637637)
    (0.786801702180,-0.728453958784) (0.789506142879,-0.725319186271) (0.792190694312,-0.722163253792) (0.794855405900,-0.718986094131)
    (0.797500326438,-0.715787639156) (0.800125504102,-0.712567819804) (0.802730986471,-0.709326566068) (0.805316820530,-0.706063806984)
    (0.807883052690,-0.702779470623) (0.810429728795,-0.699473484075) (0.812956894141,-0.696145773438) (0.815464593481,-0.692796263804)
    (0.817952871045,-0.689424879252) (0.820421770549,-0.686031542830) (0.822871335210,-0.682616176544) (0.825301607756,-0.679178701349)
    (0.827712630444,-0.675719037132) (0.830104445070,-0.672237102704) (0.832477092983,-0.668732815783) (0.834830615100,-0.665206092988)
    (0.837165051917,-0.661656849818) (0.839480443525,-0.658085000649) (0.841776829624,-0.654490458714) (0.844054249535,-0.650873136095)
    (0.846312742217,-0.647232943710) (0.848552346279,-0.643569791298) (0.850773099995,-0.639883587410) (0.852975041319,-0.636174239394)
    (0.855158207901,-0.632441653386) (0.857322637099,-0.628685734293) (0.859468365997,-0.624906385783) (0.861595431418,-0.621103510274)
    (0.863703869941,-0.617277008921) (0.865793717917,-0.613426781601) (0.867865011482,-0.609552726905) (0.869917786578,-0.605654742123)
    (0.871952078963,-0.601732723231) (0.873967924234,-0.597786564883) (0.875965357838,-0.593816160396) (0.877944415092,-0.589821401738)
    (0.879905131200,-0.585802179516) (0.881847541269,-0.581758382967) (0.883771680329,-0.577689899942) (0.885677583347,-0.573596616897)
    (0.887565285249,-0.569478418883) (0.889434820936,-0.565335189531) (0.891286225304,-0.561166811044) (0.893119533261,-0.556973164182)
    (0.894934779749,-0.552754128256) (0.896731999761,-0.548509581113) (0.898511228360,-0.544239399130) (0.900272500703,-0.539943457197)
    (0.902015852058,-0.535621628714) (0.903741317825,-0.531273785575) (0.905448933558,-0.526899798163) (0.907138734986,-0.522499535337)
    (0.908810758033,-0.518072864426) (0.910465038843,-0.513619651215) (0.912101613800,-0.509139759944) (0.913720519548,-0.504633053293)
    (0.915321793022,-0.500099392375) (0.916905471461,-0.495538636732) (0.918471592440,-0.490950644326) (0.920020193889,-0.486335271529)
    (0.921551314120,-0.481692373122) (0.923064991850,-0.477021802284) (0.924561266225,-0.472323410591) (0.926040176849,-0.467597048007)
    (0.927501763807,-0.462842562882) (0.928946067691,-0.458059801948) (0.930373129626,-0.453248610313) (0.931782991301,-0.448408831462)
    (0.933175694989,-0.443540307253) (0.934551283581,-0.438642877913) (0.935909800612,-0.433716382043) (0.937251290285,-0.428760656612)
    (0.938575797505,-0.423775536961) (0.939883367908,-0.418760856803) (0.941174047885,-0.413716448228) (0.942447884616,-0.408642141701)
    (0.943704926099,-0.403537766071) (0.944945221180,-0.398403148571) (0.946168819581,-0.393238114831) (0.947375771937,-0.388042488875)
    (0.948566129822,-0.382816093139) (0.949739945780,-0.377558748472) (0.950897273360,-0.372270274151) (0.952038167147,-0.366950487891)
    (0.953162682793,-0.361599205854) (0.954270877049,-0.356216242669) (0.955362807799,-0.350801411442) (0.956438534093,-0.345354523773)
    (0.957498116179,-0.339875389775) (0.958541615535,-0.334363818093) (0.959569094904,-0.328819615924) (0.960580618327,-0.323242589037)
    (0.961576251175,-0.317632541800) (0.962556060184,-0.311989277203) (0.963520113486,-0.306312596887) (0.964468480644,-0.300602301167)
    (0.965401232686,-0.294858189071) (0.966318442136,-0.289080058362) (0.967220183046,-0.283267705580) (0.968106531033,-0.277420926072)
    (0.968977563306,-0.271539514034) (0.969833358704,-0.265623262546) (0.970673997720,-0.259671963618) (0.971499562540,-0.253685408230)
    (0.972310137071,-0.247663386377) (0.973105806968,-0.241605687121) (0.973886659670,-0.235512098636) (0.974652784423,-0.229382408264)
    (0.975404272313,-0.223216402565) (0.976141216293,-0.217013867377) (0.976863711206,-0.210774587871) (0.977571853815,-0.204498348616)
    (0.978265742826,-0.198184933637) (0.978945478913,-0.191834126487) (0.979611164741,-0.185445710308) (0.980262904983,-0.179019467904)
    (0.980900806346,-0.172555181816) (0.981524977586,-0.166052634392) (0.982135529528,-0.159511607866) (0.982732575078,-0.152931884439)
    (0.983316229240,-0.146313246357) (0.983886609127,-0.139655475998) (0.984443833970,-0.132958355957) (0.984988025128,-0.126221669133)
    (0.985519306097,-0.119445198823) (0.985780144435,-0.116041977271) (0.986037802505,-0.112628728809) (0.986292296345,-0.109205426502)
    (0.986543642126,-0.105772043457) (0.986791856154,-0.102328552827) (0.987036954868,-0.098874927814) (0.987278954841,-0.095411141669)
    (0.987517872778,-0.091937167701) (0.987753725519,-0.088452979275) (0.987986530034,-0.084958549820) (0.988216303426,-0.081453852825)
    (0.988443062933,-0.077938861852) (0.988666825920,-0.074413550528) (0.988887609887,-0.070877892560) (0.989105432461,-0.067331861729)
    (0.989320311402,-0.063775431899) (0.989532264599,-0.060208577016) (0.989741310068,-0.056631271116) (0.989947465956,-0.053043488324)
    (0.990150750533,-0.049445202863) (0.990351182201,-0.045836389049) (0.990548779484,-0.042217021303) (0.990743561031,-0.038587074150)
    (0.990935545617,-0.034946522223) (0.991124752137,-0.031295340267) (0.991311199612,-0.027633503141) (0.991494907178,-0.023960985826)
    (0.991675894096,-0.020277763421) (0.991854179743,-0.016583811155) (0.992029783611,-0.012879104383) (0.992202725311,-0.009163618595)
    (0.992373024567,-0.005437329417) (0.992540701216,-0.001700212613) (0.992705775206,0.002047755908) (0.992868266595,0.005806600090)
    (0.993028195549,0.009576343726) (0.993185582342,0.013357010458) (0.993340447349,0.017148623767) (0.993492811052,0.020951206975)
    (0.993642694032,0.024764783241) (0.993790116969,0.028589375553) (0.993935100642,0.032425006731) (0.994077665924,0.036271699419)
    (0.994217833779,0.040129476081) (0.994355625265,0.043998359004) (0.994491061528,0.047878370287) (0.994624163798,0.051769531842)
    (0.994754953393,0.055671865389) (0.994883451709,0.059585392456) (0.995009680223,0.063510134369) (0.995133660487,0.067446112257)
    (0.995255414130,0.071393347043) (0.995374962850,0.075351859443) (0.995492328414,0.079321669963) (0.995607532655,0.083302798895)
    (0.995720597469,0.087295266316) (0.995831544814,0.091299092082) (0.995940396702,0.095314295828) (0.996047175202,0.099340896965)
    (0.996151902433,0.103378914674) (0.996254600563,0.107428367907) (0.996355291804,0.111489275384) (0.996453998410,0.115561655587)
    (0.996550742674,0.119645526762) (0.996645546923,0.123740906912) (0.996738433518,0.127847813800) (0.996829424845,0.131966264940)
    (0.996918543319,0.136096277602) (0.997005811373,0.140237868803) (0.997091251459,0.144391055311) (0.997174886044,0.148555853638)
    (0.997256737603,0.152732280041) (0.997336828622,0.156920350519) (0.997415181585,0.161120080812) (0.997491818980,0.165331486399)
    (0.997566763288,0.169554582495) (0.997640036980,0.173789384051) (0.997711662519,0.178035905754) (0.997781662349,0.182294162023)
    (0.997850058894,0.186564167008) (0.997916874553,0.190845934590) (0.997982131699,0.195139478381) (0.998045852671,0.199444811720)
    (0.998108059772,0.203761947676) (0.998168775266,0.208090899042) (0.998228021369,0.212431678340) (0.998285820252,0.216784297819)
    (0.998342194031,0.221148769450) (0.998397164766,0.225525104935) (0.998450754454,0.229913315695) (0.998502985028,0.234313412881)
    (0.998553878352,0.238725407368) (0.998603456215,0.243149309754) (0.998651740327,0.247585130368) (0.998698752318,0.252032879259)
    (0.998744513730,0.256492566208) (0.998789046015,0.260964200720) (0.998832370529,0.265447792029) (0.998874508530,0.269943349099)
    (0.998915481172,0.274450880622) (0.998955309502,0.278970395023) (0.998994014454,0.283501900458) (0.999031616849,0.288045404818)
    (0.999068137386,0.292600915727) (0.999103596641,0.297168440548) (0.999138015060,0.301747986381) (0.999171412959,0.306339560067)
    (0.999203810518,0.310943168189) (0.999235227777,0.315558817073) (0.999265684629,0.320186512793) (0.999295200824,0.324826261172)
    (0.999323795958,0.329478067783) (0.999351489469,0.334141937952) (0.999378300641,0.338817876764) (0.999404248591,0.343505889059)
    (0.999429352272,0.348205979442) (0.999453630465,0.352918152282) (0.999477101779,0.357642411716) (0.999499784645,0.362378761650)
    (0.999521697313,0.367127205767) (0.999542857851,0.371887747526) (0.999563284138,0.376660390167) (0.999582993864,0.381445136716)
    (0.999602004527,0.386241989986) (0.999620333425,0.391050952581) (0.999637997660,0.395872026903) (0.999655014130,0.400705215153)
    (0.999671399530,0.405550519334) (0.999687170345,0.410407941260) (0.999702342851,0.415277482553) (0.999716933113,0.420159144655)
    (0.999730956978,0.425052928825) (0.999744430076,0.429958836149) (0.999757367820,0.434876867541) (0.999769785399,0.439807023748)
    (0.999781697777,0.444749305358) (0.999793119696,0.449703712797) (0.999804065667,0.454670246343) (0.999814549974,0.459648906124)
    (0.999824586669,0.464639692123) (0.999834189573,0.469642604188) (0.999843372273,0.474657642030) (0.999852148119,0.479684805233)
    (0.999860530229,0.484724093255) (0.999868531480,0.489775505436) (0.999876164514,0.494839041003) (0.999883441732,0.499914699070)
    (0.999890375299,0.505002478648) (0.999896977136,0.510102378649) (0.999903258926,0.515214397888) (0.999909232111,0.520338535093)
    (0.999914907892,0.525474788903) (0.999920297229,0.530623157879) (0.999925410843,0.535783640506) (0.999930259211,0.540956235198)
    (0.999934852572,0.546140940303) (0.999939200925,0.551337754108) (0.999943314029,0.556546674843) (0.999947201406,0.561767700686)
    (0.999950872337,0.567000829768) (0.999954335870,0.572246060177) (0.999957600814,0.577503389963) (0.999960675745,0.582772817142)
    (0.999963569006,0.588054339700) (0.999966288706,0.593347955599) (0.999968842727,0.598653662778) (0.999971238718,0.603971459161)
    (0.999973484105,0.609301342658) (0.999975586087,0.614643311170) (0.999977551641,0.619997362595) (0.999979387521,0.625363494828)
    (0.999981100264,0.630741705767) (0.999982696192,0.636131993315) (0.999984181410,0.641534355387) (0.999985561815,0.646948789909)
    (0.999986843092,0.652375294823) (0.999988030722,0.657813868091) (0.999989129984,0.663264507697) (0.999990145956,0.668727211652)
    (0.999991083519,0.674201977992) (0.999991947360,0.679688804785) (0.999992741976,0.685187690131) (0.999993471677,0.690698632169)
    (0.999994140589,0.696221629071) (0.999994752658,0.701756679051) (0.999995311653,0.707303780366) (0.999995821172,0.712862931313)
    (0.999996284641,0.718434130238) (0.999996705324,0.724017375532) (0.999997086320,0.729612665633) (0.999997430573,0.735219999030)
    (0.999997740873,0.740839374261) (0.999998019861,0.746470789918) (0.999998270031,0.752114244641) (0.999998493739,0.757769737126)
    (0.999998693201,0.763437266120) (0.999998870502,0.769116830427) (0.999999027599,0.774808428900) (0.999999166324,0.780512060449)
    (0.999999288390,0.786227724036) (0.999999395393,0.791955418679) (0.999999488821,0.797695143446) (0.999999570053,0.803446897459)
    (0.999999640366,0.809210679892) (0.999999700941,0.814986489970) (0.999999752863,0.820774326968) (0.999999797129,0.826574190209)
    (0.999999834653,0.832386079067) (0.999999866264,0.838209992960) (0.999999892720,0.844045931353) (0.999999914704,0.849893893753)
    (0.999999932832,0.855753879712) (0.999999947656,0.861625888820) (0.999999959669,0.867509920708) (0.999999969309,0.873405975043)
    (0.999999976960,0.879314051527) (0.999999982962,0.885234149896) (0.999999987609,0.891166269917) (0.999999991154,0.897110411386)
    (0.999999993814,0.903066574126) (0.999999995775,0.909034757986) (0.999999997191,0.915014962836) (0.999999998188,0.921007188569)
    (0.999999998873,0.927011435095) (0.999999999327,0.933027702342) (0.999999999618,0.939055990252) (0.999999999796,0.945096298780)
    (0.999999999899,0.951148627892) (0.999999999954,0.957212977564) (0.999999999982,0.963289347776) (0.999999999994,0.969377738517)
    (0.999999999998,0.975478149781) (1.000000000000,0.981590581563) (1.000000000000,0.987715033860) (1.000000000000,0.993851506673)
    (1.000000000000,1.000000000000)
}

\begin{figure}[t]
\centering
\begin{tikzpicture}[
  x=3.55cm,y=3.55cm,
  >=Latex,
  line cap=round,
  line join=round,
  curve1/.style={curveone,very thick},
  curve2/.style={curvetwo,very thick},
  curve3/.style={curvethree,very thick},
  curve4/.style={curvefour,very thick},
  curve5/.style={curvefive,very thick}
]
  \draw[black!55,thick] (-1,-1) rectangle (1,1);
  \draw[->,black!65] (-1.08,0)--(1.10,0) node[right] {$t$};
  \draw[->,black!65] (0,-1.08)--(0,1.10) node[above] {$s$};
  \foreach \x in {-1,1} {
    \draw[black!65] (\x,-.018)--(\x,.018);
    \node[below=2pt] at (\x,0) {$\x$};
  }
  \foreach \y in {-1,1} {
    \draw[black!65] (-.018,\y)--(.018,\y);
    \node[left=2pt] at (0,\y) {$\y$};
  }
  \node[below left=2pt] at (0,0) {$0$};

  \draw[curve5] plot coordinates {\PathDataFive};
  \draw[curve4] plot coordinates {\PathDataFour};
  \draw[curve3] plot coordinates {\PathDataThree};
  \draw[curve2] plot coordinates {\PathDataTwo};
  \draw[curve1] (-1,-1)--(1,1);

  \fill[curve2] (0.474790098372,-0.474790098372) circle[radius=.014];
  \fill[curve3] (0.625147328770,-0.625147328770) circle[radius=.014];
  \fill[curve4] (0.706719212873,-0.706719212873) circle[radius=.014];
  \fill[curve5] (0.758652039933,-0.758652039933) circle[radius=.014];

  \draw[curve1] (1.18,.68)--(1.40,.68);
  \node[anchor=west] at (1.46,.68) {$n=1$};
  \draw[curve2] (1.18,.43)--(1.40,.43);
  \node[anchor=west] at (1.46,.43) {$n=2$};
  \draw[curve3] (1.18,.18)--(1.40,.18);
  \node[anchor=west] at (1.46,.18) {$n=3$};
  \draw[curve4] (1.18,-.07)--(1.40,-.07);
  \node[anchor=west] at (1.46,-.07) {$n=4$};
  \draw[curve5] (1.18,-.32)--(1.40,-.32);
  \node[anchor=west] at (1.46,-.32) {$n=5$};
  \fill[black!65] (1.29,-.57) circle[radius=.014];
  \node[anchor=west] at (1.46,-.57) {turning point};
\end{tikzpicture}
\caption{Closed calibrated paths $s=T_n(t)$ joining $(-1,-1)$ to $(1,1)$.
The marked
points are $P_n=(\sqrt{a_n},-\sqrt{a_n})$.}
\label{fig:numerical-optimal-paths}
\end{figure}
\endgroup

%% file: functional_extremizers_2026-09-29.tex
\begin{figure}[t]
\centering
\begingroup
\definecolor{feqcolor1}{RGB}{31,119,180}
\definecolor{feqcolor2}{RGB}{230,126,34}
\definecolor{feqcolor3}{RGB}{44,160,44}
\definecolor{feqcolor4}{RGB}{214,39,40}
\definecolor{feqcolor5}{RGB}{148,103,189}
\expandafter\def\csname feqplus1\endcsname{%
(0.0000000000,0.0000000000) (0.0129903811,0.0000843750) (0.0281458256,0.0003960938) (0.0411362067,0.0008460938)
(0.0541265877,0.0014648438) (0.0671169688,0.0022523438) (0.0822724134,0.0033843750) (0.0952627944,0.0045375000)
(0.1082531755,0.0058593750) (0.1212435565,0.0073500000) (0.1342339376,0.0090093750) (0.1472243186,0.0108375000)
(0.1623797632,0.0131835938) (0.1775352078,0.0157593750) (0.1905255888,0.0181500000) (0.2035159699,0.0207093750)
(0.2165063509,0.0234375000) (0.2316617955,0.0268335938) (0.2446521766,0.0299273438) (0.2576425576,0.0331898437)
(0.2706329387,0.0366210938) (0.2836233197,0.0402210938) (0.2966137008,0.0439898438) (0.3096040819,0.0479273437)
(0.3247595264,0.0527343750) (0.3377499075,0.0570375000) (0.3529053520,0.0622710938) (0.3658957331,0.0669398437)
(0.3788861142,0.0717773438) (0.3918764952,0.0767835938) (0.4048668763,0.0819585938) (0.4178572573,0.0873023437)
(0.4330127019,0.0937500000) (0.4460030829,0.0994593750) (0.4589934640,0.1053375000) (0.4741489086,0.1124085938)
(0.4871392896,0.1186523438) (0.5001296707,0.1250648437) (0.5131200517,0.1316460937) (0.5282754963,0.1395375000)
(0.5412658774,0.1464843750) (0.5542562584,0.1536000000) (0.5672466395,0.1608843750) (0.5802370205,0.1683375000)
(0.5953924651,0.1772460938) (0.6083828462,0.1850648438) (0.6213732272,0.1930523438) (0.6365286718,0.2025843750)
(0.6495190528,0.2109375000) (0.6625094339,0.2194593750) (0.6754998150,0.2281500000) (0.6906552595,0.2385023438)
(0.7036456406,0.2475585938) (0.7166360216,0.2567835938) (0.7296264027,0.2661773438) (0.7447818473,0.2773500000)
(0.7577722283,0.2871093750) (0.7707626094,0.2970375000) (0.7859180539,0.3088335937) (0.7989084350,0.3191273437)
(0.8118988160,0.3295898438) (0.8248891971,0.3402210937) (0.8378795782,0.3510210938) (0.8530350227,0.3638343750)
(0.8660254038,0.3750000000) (0.8790157848,0.3863343750) (0.8920061659,0.3978375000) (0.9049965470,0.4095093750)
(0.9201519915,0.4233398438) (0.9353074361,0.4374000000) (0.9482978171,0.4496343750) (0.9612881982,0.4620375000)
(0.9742785793,0.4746093750) (0.9872689603,0.4873500000) (1.0024244049,0.5024273437) (1.0154147859,0.5155335938)
(1.0284051670,0.5288085938) (1.0413955481,0.5422523438) (1.0565509926,0.5581500000) (1.0695413737,0.5719593750)
(1.0825317547,0.5859375000) (1.0955221358,0.6000843750) (1.1085125168,0.6144000000) (1.1236679614,0.6313148438)
(1.1366583425,0.6459960938) (1.1496487235,0.6608460938) (1.1626391046,0.6758648437) (1.1756294856,0.6910523437)
(1.1907849302,0.7089843750) (1.2037753113,0.7245375000) (1.2167656923,0.7402593750) (1.2319211369,0.7588148437)
(1.2449115179,0.7749023438) (1.2579018990,0.7911585938) (1.2730573436,0.8103375000) (1.2860477246,0.8269593750)
(1.2990381057,0.8437500000) (1.3120284867,0.8607093750) (1.3250188678,0.8778375000) (1.3380092488,0.8951343750)
(1.3531646934,0.9155273438) (1.3683201380,0.9361500000) (1.3813105190,0.9540093750) (1.3943009001,0.9720375000)
(1.4072912811,0.9902343750) (1.4202816622,1.0086000000) (1.4354371068,1.0302398437) (1.4484274878,1.0489710938)
(1.4614178689,1.0678710938) (1.4744082499,1.0869398438) (1.4895636945,1.1094000000) (1.5025540756,1.1288343750)
(1.5155444566,1.1484375000) (1.5285348377,1.1682093750) (1.5436902822,1.1914898437) (1.5566806633,1.2116273438)
(1.5696710444,1.2319335938) (1.5826614254,1.2524085938) (1.5956518065,1.2730523438) (1.6086421875,1.2938648437)
(1.6237976321,1.3183593750) (1.6367880132,1.3395375000) (1.6519434577,1.3644585937) (1.6649338388,1.3860023438)
(1.6779242198,1.4077148438) (1.6909146009,1.4295960938) (1.7039049819,1.4516460937) (1.7168953630,1.4738648438)
(1.7320508076,1.5000000000) (1.7450411886,1.5225843750) (1.7601966332,1.5491460938) (1.7731870142,1.5720960937)
(1.7861773953,1.5952148438) (1.7991677764,1.6185023438) (1.8121581574,1.6419585938) (1.8251485385,1.6655835937)
(1.8403039830,1.6933593750) (1.8532943641,1.7173500000) (1.8662847452,1.7415093750) (1.8814401897,1.7699085937)
(1.8944305708,1.7944335938) (1.9095860153,1.8232593750) (1.9225763964,1.8481500000) (1.9355667775,1.8732093750)
(1.9485571585,1.8984375000) (1.9615475396,1.9238343750) (1.9767029841,1.9536773438) (1.9896933652,1.9794398437)
(2.0026837463,2.0053710938) (2.0156741273,2.0314710938) (2.0286645084,2.0577398438) (2.0416548894,2.0841773437)
(2.0568103340,2.1152343750) (2.0698007150,2.1420375000) (2.0827910961,2.1690093750) (2.0979465407,2.2006898437)
(2.1109369217,2.2280273438) (2.1239273028,2.2555335938) (2.1369176838,2.2832085938) (2.1499080649,2.3110523438)
(2.1650635095,2.3437500000) (2.1780538905,2.3719593750) (2.1910442716,2.4003375000) (2.2061997161,2.4336585937)
(2.2191900972,2.4624023438) (2.2321804783,2.4913148438) (2.2473359228,2.5252593750) (2.2603263039,2.5545375000)
(2.2733166849,2.5839843750) (2.2863070660,2.6136000000) (2.2992974470,2.6433843750) (2.3122878281,2.6733375000)
(2.3274432727,2.7084960938) (2.3404336537,2.7388148438) (2.3534240348,2.7693023438) (2.3664144158,2.7999585938)
(2.3815698604,2.8359375000) (2.3945602415,2.8669593750) (2.4075506225,2.8981500000) (2.4227060671,2.9347523437)
(2.4356964481,2.9663085938) (2.4508518927,3.0033375000) (2.4638422738,3.0352593750) (2.4898230359,3.0996093750)
(2.5158037980,3.1646343750) (2.5287941791,3.1974000000) (2.5439496236,3.2358398438) (2.5569400047,3.2689710938)
(2.5720954492,3.3078375000) (2.5980762114,3.3750000000) (2.6240569735,3.4428375000) (2.6370473545,3.4770093750)
(2.6522027991,3.5170898438) (2.6781835612,3.5863335938) (2.6911739423,3.6212085937) (2.7063293868,3.6621093750)
(2.7214848314,3.7032398438) (2.7344752124,3.7386773438) (2.7604559746,3.8100585938) (2.7756114191,3.8520093750)
(2.7886018002,3.8881500000) (2.8145825623,3.9609375000) (2.8405633244,4.0344000000) (2.8535537055,4.0713843750)
(2.8687091500,4.1147460938) (2.8816995311,4.1520960937) (2.8968549757,4.1958843750) (2.9228357378,4.2714843750)
(2.9509815634,4.3541460937) (2.9769623255,4.4311523438) (3.0051081511,4.5153375000) (3.0310889132,4.5937500000)
(3.0570696754,4.6728375000) (3.0852155010,4.7592773438) (3.1133613266,4.8465093750) (3.1393420887,4.9277343750)
(3.1653228508,5.0096343750) (3.1934686765,5.0991210938) (3.2194494386,5.1824273438) (3.2475952642,5.2734375000)
(3.2757410898,5.3652398438) (3.3017218519,5.4506835938) (3.3298676776,5.5440093750) (3.3558484397,5.6308593750)
(3.3818292018,5.7183843750) (3.4099750274,5.8139648438) (3.4359557895,5.9028960938) (3.4641016151,6.0000000000)
}
\expandafter\def\csname feqminus1\endcsname{%
(0.0000000000,0.0000000000) (0.0129903811,0.0000843750) (0.0281458256,0.0003960938) (0.0411362067,0.0008460938)
(0.0541265877,0.0014648438) (0.0671169688,0.0022523438) (0.0822724134,0.0033843750) (0.0952627944,0.0045375000)
(0.1082531755,0.0058593750) (0.1212435565,0.0073500000) (0.1342339376,0.0090093750) (0.1472243186,0.0108375000)
(0.1623797632,0.0131835938) (0.1775352078,0.0157593750) (0.1905255888,0.0181500000) (0.2035159699,0.0207093750)
(0.2165063509,0.0234375000) (0.2316617955,0.0268335938) (0.2446521766,0.0299273438) (0.2576425576,0.0331898437)
(0.2706329387,0.0366210938) (0.2836233197,0.0402210938) (0.2966137008,0.0439898438) (0.3096040819,0.0479273437)
(0.3247595264,0.0527343750) (0.3377499075,0.0570375000) (0.3529053520,0.0622710938) (0.3658957331,0.0669398437)
(0.3788861142,0.0717773438) (0.3918764952,0.0767835938) (0.4048668763,0.0819585938) (0.4178572573,0.0873023437)
(0.4330127019,0.0937500000) (0.4460030829,0.0994593750) (0.4589934640,0.1053375000) (0.4741489086,0.1124085938)
(0.4871392896,0.1186523438) (0.5001296707,0.1250648437) (0.5131200517,0.1316460937) (0.5282754963,0.1395375000)
(0.5412658774,0.1464843750) (0.5542562584,0.1536000000) (0.5672466395,0.1608843750) (0.5802370205,0.1683375000)
(0.5953924651,0.1772460938) (0.6083828462,0.1850648438) (0.6213732272,0.1930523438) (0.6365286718,0.2025843750)
(0.6495190528,0.2109375000) (0.6625094339,0.2194593750) (0.6754998150,0.2281500000) (0.6906552595,0.2385023438)
(0.7036456406,0.2475585938) (0.7166360216,0.2567835938) (0.7296264027,0.2661773438) (0.7447818473,0.2773500000)
(0.7577722283,0.2871093750) (0.7707626094,0.2970375000) (0.7859180539,0.3088335937) (0.7989084350,0.3191273437)
(0.8118988160,0.3295898438) (0.8248891971,0.3402210937) (0.8378795782,0.3510210938) (0.8530350227,0.3638343750)
(0.8660254038,0.3750000000) (0.8790157848,0.3863343750) (0.8920061659,0.3978375000) (0.9049965470,0.4095093750)
(0.9201519915,0.4233398438) (0.9353074361,0.4374000000) (0.9482978171,0.4496343750) (0.9612881982,0.4620375000)
(0.9742785793,0.4746093750) (0.9872689603,0.4873500000) (1.0024244049,0.5024273437) (1.0154147859,0.5155335938)
(1.0284051670,0.5288085938) (1.0413955481,0.5422523438) (1.0565509926,0.5581500000) (1.0695413737,0.5719593750)
(1.0825317547,0.5859375000) (1.0955221358,0.6000843750) (1.1085125168,0.6144000000) (1.1236679614,0.6313148438)
(1.1366583425,0.6459960938) (1.1496487235,0.6608460938) (1.1626391046,0.6758648437) (1.1756294856,0.6910523437)
(1.1907849302,0.7089843750) (1.2037753113,0.7245375000) (1.2167656923,0.7402593750) (1.2319211369,0.7588148437)
(1.2449115179,0.7749023438) (1.2579018990,0.7911585938) (1.2730573436,0.8103375000) (1.2860477246,0.8269593750)
(1.2990381057,0.8437500000) (1.3120284867,0.8607093750) (1.3250188678,0.8778375000) (1.3380092488,0.8951343750)
(1.3531646934,0.9155273438) (1.3683201380,0.9361500000) (1.3813105190,0.9540093750) (1.3943009001,0.9720375000)
(1.4072912811,0.9902343750) (1.4202816622,1.0086000000) (1.4354371068,1.0302398437) (1.4484274878,1.0489710938)
(1.4614178689,1.0678710938) (1.4744082499,1.0869398438) (1.4895636945,1.1094000000) (1.5025540756,1.1288343750)
(1.5155444566,1.1484375000) (1.5285348377,1.1682093750) (1.5436902822,1.1914898437) (1.5566806633,1.2116273438)
(1.5696710444,1.2319335938) (1.5826614254,1.2524085938) (1.5956518065,1.2730523438) (1.6086421875,1.2938648437)
(1.6237976321,1.3183593750) (1.6367880132,1.3395375000) (1.6519434577,1.3644585937) (1.6649338388,1.3860023438)
(1.6779242198,1.4077148438) (1.6909146009,1.4295960938) (1.7039049819,1.4516460937) (1.7168953630,1.4738648438)
(1.7320508076,1.5000000000) (1.7450411886,1.5225843750) (1.7601966332,1.5491460938) (1.7731870142,1.5720960937)
(1.7861773953,1.5952148438) (1.7991677764,1.6185023438) (1.8121581574,1.6419585938) (1.8251485385,1.6655835937)
(1.8403039830,1.6933593750) (1.8532943641,1.7173500000) (1.8662847452,1.7415093750) (1.8814401897,1.7699085937)
(1.8944305708,1.7944335938) (1.9095860153,1.8232593750) (1.9225763964,1.8481500000) (1.9355667775,1.8732093750)
(1.9485571585,1.8984375000) (1.9615475396,1.9238343750) (1.9767029841,1.9536773438) (1.9896933652,1.9794398437)
(2.0026837463,2.0053710938) (2.0156741273,2.0314710938) (2.0286645084,2.0577398438) (2.0416548894,2.0841773437)
(2.0568103340,2.1152343750) (2.0698007150,2.1420375000) (2.0827910961,2.1690093750) (2.0979465407,2.2006898437)
(2.1109369217,2.2280273438) (2.1239273028,2.2555335938) (2.1369176838,2.2832085938) (2.1499080649,2.3110523438)
(2.1650635095,2.3437500000) (2.1780538905,2.3719593750) (2.1910442716,2.4003375000) (2.2061997161,2.4336585937)
(2.2191900972,2.4624023438) (2.2321804783,2.4913148438) (2.2473359228,2.5252593750) (2.2603263039,2.5545375000)
(2.2733166849,2.5839843750) (2.2863070660,2.6136000000) (2.2992974470,2.6433843750) (2.3122878281,2.6733375000)
(2.3274432727,2.7084960938) (2.3404336537,2.7388148438) (2.3534240348,2.7693023438) (2.3664144158,2.7999585938)
(2.3815698604,2.8359375000) (2.3945602415,2.8669593750) (2.4075506225,2.8981500000) (2.4227060671,2.9347523437)
(2.4356964481,2.9663085938) (2.4508518927,3.0033375000) (2.4638422738,3.0352593750) (2.4898230359,3.0996093750)
(2.5158037980,3.1646343750) (2.5287941791,3.1974000000) (2.5439496236,3.2358398438) (2.5569400047,3.2689710938)
(2.5720954492,3.3078375000) (2.5980762114,3.3750000000) (2.6240569735,3.4428375000) (2.6370473545,3.4770093750)
(2.6522027991,3.5170898438) (2.6781835612,3.5863335938) (2.6911739423,3.6212085937) (2.7063293868,3.6621093750)
(2.7214848314,3.7032398438) (2.7344752124,3.7386773438) (2.7604559746,3.8100585938) (2.7756114191,3.8520093750)
(2.7886018002,3.8881500000) (2.8145825623,3.9609375000) (2.8405633244,4.0344000000) (2.8535537055,4.0713843750)
(2.8687091500,4.1147460938) (2.8816995311,4.1520960937) (2.8968549757,4.1958843750) (2.9228357378,4.2714843750)
(2.9509815634,4.3541460937) (2.9769623255,4.4311523438) (3.0051081511,4.5153375000) (3.0310889132,4.5937500000)
(3.0570696754,4.6728375000) (3.0852155010,4.7592773438) (3.1133613266,4.8465093750) (3.1393420887,4.9277343750)
(3.1653228508,5.0096343750) (3.1934686765,5.0991210938) (3.2194494386,5.1824273438) (3.2475952642,5.2734375000)
(3.2757410898,5.3652398438) (3.3017218519,5.4506835938) (3.3298676776,5.5440093750) (3.3558484397,5.6308593750)
(3.3818292018,5.7183843750) (3.4099750274,5.8139648438) (3.4359557895,5.9028960938) (3.4641016151,6.0000000000)
}
\expandafter\def\csname feqplus2\endcsname{%
(0.0000000000,0.0000000000) (0.0835064296,0.0638085938) (0.1176482220,0.0900375000) (0.1493072277,0.1144710938)
(0.1801508757,0.1383960938) (0.2105622301,0.1621148437) (0.2414985081,0.1863843750) (0.2708299430,0.2095335938)
(0.2998294794,0.2325585937) (0.3301510359,0.2567835938) (0.3597526381,0.2805843750) (0.3904675408,0.3054398437)
(0.4201189707,0.3295898438) (0.4507020499,0.3546585938) (0.4799194693,0.3787593750) (0.5099115139,0.4036523438)
(0.5406651993,0.4293375000) (0.5721681235,0.4558148438) (0.6044084866,0.4830843750) (0.6348136671,0.5089593750)
(0.6658293211,0.5355093750) (0.6974475533,0.5627343750) (0.7296609016,0.5906343750) (0.7597065782,0.6168023437)
(0.7930412420,0.6459960938) (0.8241039186,0.6733500000) (0.8556458139,0.7012710937) (0.8876626438,0.7297593750)
(0.9201503724,0.7588148437) (0.9531052004,0.7884375000) (0.9865235543,0.8186273438) (1.0204020751,0.8493843750)
(1.0547376075,0.8807085938) (1.0895271892,0.9126000000) (1.1247680405,0.9450585938) (1.1604575542,0.9780843750)
(1.1965932861,1.0116773438) (1.2331729457,1.0458375000) (1.2701943872,1.0805648437) (1.3042319142,1.1126273438)
(1.3420912857,1.1484375000) (1.3803869433,1.1848148438) (1.4191172300,1.2217593750) (1.4547024330,1.2558375000)
(1.4942582474,1.2938648437) (1.5305915961,1.3289273438) (1.5709679665,1.3680375000) (1.6080452244,1.4040843750)
(1.6492377837,1.4442773438) (1.6870552321,1.4813085938) (1.7290601494,1.5225843750) (1.7676145279,1.5606000000)
(1.8104284470,1.6029585937) (1.8497169018,1.6419585938) (1.8893540197,1.6814273438) (1.9293392711,1.7213648437)
(1.9737245510,1.7658375000) (2.0144393062,1.8067593750) (2.0555007584,1.8481500000) (2.0969085038,1.8900093750)
(2.1428565431,1.9365960937) (2.1849903065,1.9794398437) (2.2274692736,2.0227523438) (2.2702931403,2.0665335938)
(2.3177974152,2.1152343750) (2.3613446725,2.1600000000) (2.4096440643,2.2097835938) (2.4539136392,2.2555335938)
(2.5030070567,2.3064000000) (2.5479980336,2.3531343750) (2.5933322546,2.4003375000) (2.6390095578,2.4480093750)
(2.6896506740,2.5009898438) (2.7360479801,2.5496460937) (2.7827879463,2.5987710937) (2.8298704598,2.6483648438)
(2.8820567449,2.7034593750) (2.9298582860,2.7540375000) (2.9780020880,2.8050843750) (3.0264880770,2.8566000000)
(3.0802178130,2.9138085938) (3.1294221904,2.9663085938) (3.1839420234,3.0246000000) (3.2338645688,3.0780843750)
(3.2891742916,3.1374585937) (3.3398148465,3.1919273437) (3.3907972712,3.2468648437) (3.4421215477,3.3022710938)
(3.4989730754,3.3637593750) (3.5510152047,3.4201500000) (3.6561249466,3.5343375000) (3.7626020285,3.6504000000)
(3.8163533470,3.7091343750) (3.8758746603,3.7742835937) (3.9851551214,3.8941898438) (4.0403082518,3.9548460938)
(4.1013716716,4.0221093750) (4.2134563140,4.1458593750) (4.3326178183,4.2778148437) (4.4475077912,4.4054085938)
(4.5696156065,4.5414000000) (4.6932331166,4.6794585938) (4.8183607705,4.8195843750) (4.9389343856,4.9549593750)
(5.0670117966,5.0991210938) (5.1966008025,5.2453500000) (5.3277019279,5.3936460938) (5.4539664588,5.5368023438)
(5.5880213740,5.6891343750) (5.7235900216,5.8435335938) (5.8606729640,6.0000000000)
}
\expandafter\def\csname feqminus2\endcsname{%
(0.0000000000,0.0000000000) (0.7632701495,0.0000093750) (0.7663846843,0.0001898437) (0.7704035227,0.0006000000)
(0.7742698074,0.0011343750) (0.7784146210,0.0018375000) (0.7827752840,0.0027093750) (0.7873084972,0.0037500000)
(0.7919822369,0.0049593750) (0.7959663139,0.0060960937) (0.8008368844,0.0076148438) (0.8049590180,0.0090093750)
(0.8091298744,0.0105210938) (0.8133422921,0.0121500000) (0.8184432468,0.0142593750) (0.8227259543,0.0161460937)
(0.8270324557,0.0181500000) (0.8313583832,0.0202710938) (0.8374398803,0.0234375000) (0.8426694483,0.0263343750)
(0.8487827190,0.0299273438) (0.8540275589,0.0331898437) (0.8601462141,0.0372093750) (0.8653860794,0.0408375000)
(0.8714888797,0.0452835938) (0.8767073487,0.0492773438) (0.8827769538,0.0541500000) (0.8879605434,0.0585093750)
(0.8939826480,0.0638085938) (0.8991201900,0.0685335937) (0.9050829049,0.0742593750) (0.9101650913,0.0793500000)
(0.9152190211,0.0846093750) (0.9202431689,0.0900375000) (0.9260651638,0.0965835938) (0.9310201066,0.1023773437)
(0.9359411813,0.1083398437) (0.9408273028,0.1144710938) (0.9464822734,0.1218375000) (0.9512893445,0.1283343750)
(0.9560585691,0.1350000000) (0.9607891852,0.1418343750) (0.9662585137,0.1500210938) (0.9709031667,0.1572210937)
(0.9755072061,0.1645898437) (0.9800701142,0.1721273437) (0.9845914228,0.1798335938) (0.9890707104,0.1877085938)
(0.9935075998,0.1957523438) (0.9979017552,0.2039648438) (1.0029738654,0.2137593750) (1.0072744626,0.2223375000)
(1.0115315114,0.2310843750) (1.0157448260,0.2400000000) (1.0199142516,0.2490843750) (1.0240396630,0.2583375000)
(1.0281209631,0.2677593750) (1.0321580811,0.2773500000) (1.0361509716,0.2871093750) (1.0400996129,0.2970375000)
(1.0440040062,0.3071343750) (1.0478641741,0.3174000000) (1.0516801599,0.3278343750) (1.0554520262,0.3384375000)
(1.0591798541,0.3492093750) (1.0628637426,0.3601500000) (1.0665038071,0.3712593750) (1.0701001790,0.3825375000)
(1.0736530049,0.3939843750) (1.0771624459,0.4056000000) (1.0806286764,0.4173843750) (1.0840518841,0.4293375000)
(1.0874322690,0.4414593750) (1.0907700427,0.4537500000) (1.0940654279,0.4662093750) (1.0973186580,0.4788375000)
(1.1005299764,0.4916343750) (1.1036996358,0.5046000000) (1.1068278979,0.5177343750) (1.1099150331,0.5310375000)
(1.1129613197,0.5445093750) (1.1164640744,0.5604398437) (1.1194228458,0.5742773438) (1.1228243136,0.5906343750)
(1.1256969816,0.6048375000) (1.1289988467,0.6216210938) (1.1317868846,0.6361898437) (1.1349908981,0.6534000000)
(1.1381429685,0.6708398438) (1.1412436283,0.6885093750) (1.1438608273,0.7038375000) (1.1468674484,0.7219335937)
(1.1498242194,0.7402593750) (1.1527316978,0.7588148437) (1.1551850157,0.7749023438) (1.1580024467,0.7938843750)
(1.1607722039,0.8130960937) (1.1634948610,0.8325375000) (1.1657915164,0.8493843750) (1.1684282340,0.8692523437)
(1.1710195054,0.8893500000) (1.1735659134,0.9096773438) (1.1760680417,0.9302343750) (1.1785264748,0.9510210938)
(1.1809417976,0.9720375000) (1.1833145952,0.9932835937) (1.1856454520,1.0147593750) (1.1879349521,1.0364648438)
(1.1901836782,1.0584000000) (1.1923922118,1.0805648437) (1.1966910193,1.1255835937) (1.1987824466,1.1484375000)
(1.2011262916,1.1748375000) (1.2051115073,1.2217593750) (1.2092212401,1.2730523438) (1.2129122594,1.3218773437)
(1.2167169686,1.3752093750) (1.2203717575,1.4295960938) (1.2238818536,1.4850375000) (1.2272523810,1.5415335937)
(1.2304883553,1.5990843750) (1.2335946788,1.6576898438) (1.2367705760,1.7213648437) (1.2396239829,1.7821500000)
(1.2425405736,1.8481500000) (1.2453312551,1.9153500000) (1.2480012108,1.9837500000) (1.2505554603,2.0533500000)
(1.2529988586,2.1241500000) (1.2553360975,2.1961500000) (1.2577081574,2.2739648437) (1.2598405650,2.3484398437)
(1.2620046632,2.4288843750) (1.2640688427,2.5106835937) (1.2661506966,2.5987710937) (1.2680235771,2.6833593750)
(1.2699126439,2.7744000000) (1.2717097794,2.8669593750) (1.2734195908,2.9610375000) (1.2750464701,3.0566343750)
(1.2766783805,3.1591898437) (1.2782272262,3.2634375000) (1.2796974461,3.3693773438) (1.2810932495,3.4770093750)
(1.2824865039,3.5921343750) (1.2838061220,3.7091343750) (1.2850562590,3.8280093750) (1.2862984116,3.9548460938)
(1.2874725586,4.0837500000) (1.2885827200,4.2147210938) (1.2896812319,4.3541460937) (1.2907177402,4.4958398437)
(1.2916960622,4.6398023438) (1.2926605166,4.7927343750) (1.2935690976,4.9481460938) (1.2944253680,5.1060375000)
(1.2952666896,5.2734375000) (1.2968332213,5.6235960937) (1.2982717518,6.0000000000)
}
\expandafter\def\csname feqplus3\endcsname{%
(0.0000000000,0.0000000000) (0.1262489203,0.1134375000) (0.2088127276,0.1877085938) (0.2837315169,0.2552343750)
(0.3525305942,0.3174000000) (0.4223196496,0.3806460938) (0.4892045365,0.4414593750) (0.5536554731,0.5002593750)
(0.6193868138,0.5604398437) (0.6885896140,0.6240375000) (0.7557228756,0.6859710937) (0.8229032184,0.7481835938)
(0.8897688669,0.8103375000) (0.9590416145,0.8749710938) (1.0275339213,0.9391148437) (1.0949308958,1.0024593750)
(1.1642835781,1.0678710938) (1.2390211589,1.1386148438) (1.3123358874,1.2082593750) (1.3875669495,1.2799710938)
(1.4609872363,1.3501898437) (1.5361271955,1.4222835938) (1.6129789282,1.4962523438) (1.6915352123,1.5720960937)
(1.7717894577,1.6498148438) (1.8537356625,1.7294085938) (1.9331467593,1.8067593750) (2.0183770519,1.8900093750)
(2.1008994741,1.9708335937) (2.1849321319,2.0533500000) (2.2704718428,2.1375585938) (2.3575157217,2.2234593750)
(2.4460611577,2.3110523438) (2.5408864704,2.4050835938) (2.6325069124,2.4961500000) (2.7256224130,2.5889085937)
(2.8202312608,2.6833593750) (2.9163319220,2.7795023438) (3.0139230253,2.8773375000) (3.1130033484,2.9768648437)
(3.2135718053,3.0780843750) (3.3156274348,3.1809960938) (3.4191693897,3.2856000000) (3.5241969272,3.3918960938)
(3.6307094000,3.4998843750) (3.7387062476,3.6095648438) (3.8481869897,3.7209375000) (3.9591512182,3.8340023438)
(4.0715985918,3.9487593750) (4.1855288297,4.0652085938) (4.3009417062,4.1833500000) (4.4178370464,4.3031835938)
(4.5362147212,4.4247093750) (4.6560746434,4.5479273437) (4.7709933845,4.6662210937) (4.8937396792,4.7927343750)
(5.0179681734,4.9209398438) (5.1370256025,5.0439585938) (5.2641406367,5.1754593750) (5.3927380743,5.3086523438)
(5.5228180302,5.4435375000) (5.6543806386,5.5801148437) (5.7874260514,5.7183843750) (5.9219544365,5.8583460938)
(6.0579659760,6.0000000000)
}
\expandafter\def\csname feqminus3\endcsname{%
(0.0000000000,0.0000000000) (0.8983720574,0.0000000000) (0.9000452724,0.0002343750) (0.9019379918,0.0006773438)
(0.9041597750,0.0013500000) (0.9066226765,0.0022523438) (0.9092716296,0.0033843750) (0.9116606415,0.0045375000)
(0.9141392080,0.0058593750) (0.9166922863,0.0073500000) (0.9193073917,0.0090093750) (0.9219739673,0.0108375000)
(0.9246829534,0.0128343750) (0.9274264827,0.0150000000) (0.9301976570,0.0173343750) (0.9325236668,0.0194085938)
(0.9353301759,0.0220523437) (0.9376783263,0.0243843750) (0.9405035011,0.0273375000) (0.9428612104,0.0299273438)
(0.9456913227,0.0331898437) (0.9480481615,0.0360375000) (0.9504016287,0.0390023438) (0.9527500591,0.0420843750)
(0.9550919108,0.0452835938) (0.9574257540,0.0486000000) (0.9597502625,0.0520335937) (0.9620642052,0.0555843750)
(0.9643664384,0.0592523438) (0.9666559002,0.0630375000) (0.9689316036,0.0669398437) (0.9711926321,0.0709593750)
(0.9734381347,0.0750960937) (0.9756673213,0.0793500000) (0.9778794593,0.0837210938) (0.9800738695,0.0882093750)
(0.9822499232,0.0928148438) (0.9844070392,0.0975375000) (0.9882404393,0.1063335937) (0.9924224167,0.1165523438)
(0.9961139787,0.1261500000) (1.0001329608,0.1372593750) (1.0036737591,0.1476585938) (1.0075215289,0.1596585937)
(1.0109055573,0.1708593750) (1.0145768402,0.1837500000) (1.0178004989,0.1957523438) (1.0212924543,0.2095335938)
(1.0243541396,0.2223375000) (1.0273390301,0.2355210937) (1.0305657198,0.2506148437) (1.0333892201,0.2646000000)
(1.0364375477,0.2805843750) (1.0399731205,0.3003843750) (1.0416913061,0.3105375000) (1.0436544725,0.3225960937)
(1.0469173580,0.3438023438) (1.0485008114,0.3546585938) (1.0503082633,0.3675375000) (1.0533081602,0.3901500000)
(1.0564202156,0.4154085937) (1.0591686943,0.4394273437) (1.0620152283,0.4662093750) (1.0647293785,0.4937835937)
(1.0673152037,0.5221500000) (1.0697768542,0.5513085938) (1.0721185445,0.5812593750) (1.0743445285,0.6120023438)
(1.0764590772,0.6435375000) (1.0784664583,0.6758648437) (1.0805132726,0.7115648437) (1.0823115978,0.7455375000)
(1.0841429787,0.7830093750) (1.0858701570,0.8214000000) (1.0874981234,0.8607093750) (1.0890317487,0.9009375000)
(1.0905756107,0.9450585938) (1.0919287315,0.9871898437) (1.0932895744,1.0333500000) (1.0945635627,1.0805648437)
(1.0958324383,1.1320898437) (1.0969426940,1.1814843750) (1.0980476681,1.2353343750) (1.0990759916,1.2903843750)
(1.1000327084,1.3466343750) (1.1009760925,1.4077148438) (1.1018494144,1.4701500000) (1.1026119001,1.5301500000)
(1.1033633801,1.5952148438) (1.1040587807,1.6616343750) (1.1047022572,1.7294085938) (1.1059107204,1.8774023438)
(1.1069624487,2.0358375000) (1.1078953990,2.2097835938) (1.1087141751,2.4003375000) (1.1094256313,2.6086523438)
(1.1100504548,2.8410960938) (1.1105810958,3.0942210938) (1.1110364033,3.3750000000) (1.1114221364,3.6855843750)
(1.1117552160,4.0405523438) (1.1120325456,4.4376000000) (1.1122631910,4.8870375000) (1.1126113948,6.0000000000)
}
\expandafter\def\csname feqplus4\endcsname{%
(0.0000000000,0.0000000000) (0.2703086172,0.2552343750) (0.4335098622,0.4095093750) (0.5856148961,0.5535843750)
(0.7250221146,0.6859710937) (0.7987629771,0.7561500000) (0.8701307718,0.8241773437) (0.9414217425,0.8922398437)
(1.0122813121,0.9600000000) (1.0823729104,1.0271343750) (1.1547145062,1.0965375000) (1.2258567266,1.1649023438)
(1.2990299906,1.2353343750) (1.3742256071,1.3078335937) (1.4514353273,1.3824000000) (1.5268337830,1.4553375000)
(1.6040456804,1.5301500000) (1.6830650616,1.6068375000) (1.7598025676,1.6814273438) (1.8423308233,1.7617710937)
(1.9223926385,1.8398343750) (2.0040678991,1.9195898438) (2.0829294783,1.9967085938) (2.1677365472,2.0797593750)
(2.2495592558,2.1600000000) (2.3328186348,2.2417593750) (2.4175123878,2.3250375000) (2.5036383855,2.4098343750)
(2.5911946565,2.4961500000) (2.6851648222,2.5889085937) (2.7756555277,2.6783460937) (2.8675713640,2.7693023438)
(2.9609109073,2.8617773438) (3.0556728481,2.9557710938) (3.1518559837,3.0512835938) (3.2494592109,3.1483148438)
(3.3484815197,3.2468648437) (3.4489219868,3.3469335938) (3.5507797696,3.4485210938) (3.6540541008,3.5516273437)
(3.7528910285,3.6504000000) (3.8589178217,3.7564593750) (3.9663592922,3.8640375000) (4.0752149213,3.9731343750)
(4.1793210892,4.0775648437) (4.2909251807,4.1896148438) (4.4039421944,4.3031835938) (4.6277409894,4.5283593750)
(4.8568785902,4.7592773438) (5.0913534414,4.9959375000) (5.3311644035,5.2383398438) (5.5763106947,5.4864843750)
(5.8267918404,5.7403710938) (6.0826076294,6.0000000000)
}
\expandafter\def\csname feqminus4\endcsname{%
(0.0000000000,0.0000000000) (0.9441856825,0.0000093750) (0.9450536378,0.0002343750) (0.9460194100,0.0006000000)
(0.9471626681,0.0011343750) (0.9484420448,0.0018375000) (0.9498299442,0.0027093750) (0.9510547077,0.0035648438)
(0.9536527938,0.0056273437) (0.9561229144,0.0078843750) (0.9586855119,0.0105210938) (0.9613160381,0.0135375000)
(0.9639945334,0.0169335937) (0.9667044486,0.0207093750) (0.9691283381,0.0243843750) (0.9715576514,0.0283593750)
(0.9739850928,0.0326343750) (0.9764042785,0.0372093750) (0.9788096039,0.0420843750) (0.9811961381,0.0472593750)
(0.9832655141,0.0520335937) (0.9856055194,0.0577710938) (0.9876285818,0.0630375000) (0.9899099985,0.0693375000)
(0.9918774037,0.0750960937) (0.9940907971,0.0819585938) (0.9959952781,0.0882093750) (0.9978684022,0.0946898438)
(0.9997091175,0.1014000000) (1.0017719450,0.1093500000) (1.0035403111,0.1165523438) (1.0052738311,0.1239843750)
(1.0069719716,0.1316460937) (1.0086343022,0.1395375000) (1.0102604876,0.1476585938) (1.0118502812,0.1560093750)
(1.0134035190,0.1645898437) (1.0157702769,0.1785375000) (1.0182490381,0.1944000000) (1.0206206850,0.2109375000)
(1.0228862293,0.2281500000) (1.0248709867,0.2445210938) (1.0269374841,0.2630273438) (1.0289030352,0.2822085938)
(1.0307700023,0.3020648437) (1.0325409788,0.3225960937) (1.0342187506,0.3438023438) (1.0358062606,0.3656835938)
(1.0373065756,0.3882398438) (1.0387228567,0.4114710938) (1.0401660764,0.4374000000) (1.0414177028,0.4620375000)
(1.0426902590,0.4894898438) (1.0438801792,0.5177343750) (1.0449916120,0.5467710937) (1.0460286521,0.5766000000)
(1.0469953200,0.6072210937) (1.0478955446,0.6386343750) (1.0487950923,0.6733500000) (1.0496266316,0.7089843750)
(1.0503946076,0.7455375000) (1.0511032807,0.7830093750) (1.0518013824,0.8241773437) (1.0524408249,0.8664000000)
(1.0530260959,0.9096773438) (1.0535954301,0.9570023438) (1.0541127305,1.0055273437) (1.0550340712,1.1094000000)
(1.0557461977,1.2116273438) (1.0563522154,1.3218773437) (1.0568777684,1.4442773438) (1.0573262169,1.5797835938)
(1.0577030625,1.7294085938) (1.0580220901,1.8984375000) (1.0582862246,2.0886000000) (1.0585043930,2.3064000000)
(1.0588086700,2.8153500000) (1.0590003566,3.5056148438) (1.0591109014,4.4828648437) (1.0591687501,6.0000000000)
}
\expandafter\def\csname feqplus5\endcsname{%
(0.0000000000,0.0000000000) (0.3307735877,0.3191273437) (0.5117277130,0.4937835937) (0.6769613479,0.6534000000)
(0.8250468577,0.7966148438) (0.9783282584,0.9450585938) (1.1278973854,1.0901343750) (1.2738488192,1.2319335938)
(1.4210720164,1.3752093750) (1.5761201245,1.5263648438) (1.7348892042,1.6814273438) (1.8925517827,1.8356835937)
(2.0524356456,1.9923843750) (2.2140097084,2.1510093750) (2.3767621738,2.3110523438) (2.5450861114,2.4768375000)
(2.7139230030,2.6433843750) (2.8931928176,2.8204898438) (3.0779986349,3.0033375000) (3.2683312954,3.1919273437)
(3.4585105440,3.3806273437) (3.6538914518,3.5747460937) (3.8544685209,3.7742835937) (4.0602370116,3.9792398438)
(4.2649142673,4.1833500000) (4.4809015792,4.3989843750) (4.6954910434,4.6134585938) (4.9149589751,4.8330375000)
(5.1393034676,5.0577210937) (5.3685229405,5.2875093750) (5.6026161013,5.5224023438) (5.8415819118,5.7624000000)
(6.0779589654,6.0000000000)
}
\expandafter\def\csname feqminus5\endcsname{%
(0.0000000000,0.0000000000) (0.9647527483,0.0000000000) (0.9657031054,0.0003960938) (0.9670048504,0.0012398438)
(0.9684161021,0.0024000000) (0.9698069532,0.0037500000) (0.9712947102,0.0054000000) (0.9726581703,0.0070898438)
(0.9742707503,0.0093023438) (0.9757182118,0.0114843750) (0.9771909763,0.0138960937) (0.9784680712,0.0161460937)
(0.9799696918,0.0189843750) (0.9812628533,0.0216000000) (0.9827743098,0.0248648437) (0.9840691016,0.0278460938)
(0.9853606317,0.0309960938) (0.9866467085,0.0343148437) (0.9879253533,0.0378023438) (0.9891947785,0.0414585938)
(0.9916996666,0.0492773438) (0.9941502516,0.0577710938) (0.9963411829,0.0661500000) (0.9986640045,0.0759375000)
(1.0007267388,0.0855023438) (1.0027222078,0.0956343750) (1.0046477224,0.1063335937) (1.0065013175,0.1176000000)
(1.0082816673,0.1294335938) (1.0099880112,0.1418343750) (1.0116200881,0.1548023437) (1.0131780779,0.1683375000)
(1.0146625484,0.1824398437) (1.0160744075,0.1971093750) (1.0175332268,0.2137593750) (1.0187974539,0.2296148437)
(1.0200990020,0.2475585938) (1.0212229800,0.2646000000) (1.0223762207,0.2838375000) (1.0233688546,0.3020648437)
(1.0243841029,0.3225960937) (1.0253317388,0.3438023438) (1.0262148620,0.3656835938) (1.0270366060,0.3882398438)
(1.0278001118,0.4114710938) (1.0285085035,0.4353773438) (1.0292173117,0.4620375000) (1.0298689089,0.4894898438)
(1.0304670448,0.5177343750) (1.0315173235,0.5766000000) (1.0324262029,0.6410835938) (1.0332273879,0.7141500000)
(1.0339170281,0.7966148438) (1.0344814279,0.8864648438) (1.0348756528,0.9690210937) (1.0352147612,1.0615523438)
(1.0354923138,1.1616000000) (1.0357294868,1.2765093750) (1.0360767441,1.5491460938) (1.0363005052,1.9153500000)
(1.0364216350,2.3719593750) (1.0364891001,3.0405960938) (1.0365320673,6.0000000000)
}

\begin{tikzpicture}[x=2.05cm,y=1.45cm,>=Latex,
    line cap=round,line join=round]

\begin{scope}[shift={(0,0)}]
  \node[anchor=south] at (1.30,3.36) {$\varphi_n(\rho)$};
  \foreach \yy in {1,2,3} {
    \draw[black!13, very thin] (0,\yy)--(2.60,\yy);
    \draw[black!55, thin] (-.035,\yy)--(.035,\yy);
    \node[anchor=east,font=\small] at (-.08,\yy) {$\yy$};
  }
  \foreach \xx in {1,2} {
    \draw[black!55, thin] (\xx,-.035)--(\xx,.035);
    \node[anchor=north,font=\small] at (\xx,-.08) {$\xx$};
  }
  \draw[black!65,thin,->] (-.025,0)--(2.72,0)
      node[right,text=black] {$\rho$};
  \draw[black!65,thin,->] (0,-.025)--(0,3.23);
  \node[anchor=north east,font=\small] at (-.025,-.075) {$0$};
  \begin{scope}
    \clip (0,-.013) rectangle (2.60,3.075);
    \foreach \nn in {1,2,3,4,5} {
      \draw[feqcolor\nn,line width=1.15pt]
        plot coordinates {\csname feqplus\nn\endcsname};
    }
  \end{scope}
\end{scope}

\begin{scope}[shift={(3.65,0)}]
  \node[anchor=south] at (1.30,3.36) {$(\mathcal L\varphi_n)(\rho)$};
  \foreach \yy in {1,2,3} {
    \draw[black!13, very thin] (0,\yy)--(2.60,\yy);
    \draw[black!55, thin] (-.035,\yy)--(.035,\yy);
    \node[anchor=east,font=\small] at (-.08,\yy) {$\yy$};
  }
  \foreach \xx in {1,2} {
    \draw[black!55, thin] (\xx,-.035)--(\xx,.035);
    \node[anchor=north,font=\small] at (\xx,-.08) {$\xx$};
  }
  \draw[black!65,thin,->] (-.025,0)--(2.72,0)
      node[right,text=black] {$\rho$};
  \draw[black!65,thin,->] (0,-.025)--(0,3.23);
  \node[anchor=north east,font=\small] at (-.025,-.075) {$0$};
  \begin{scope}
    \clip (0,-.013) rectangle (2.60,3.075);
    \foreach \nn in {1,2,3,4,5} {
      \draw[feqcolor\nn,line width=1.15pt]
        plot coordinates {\csname feqminus\nn\endcsname};
    }
  \end{scope}
\end{scope}

\foreach \nn/\xx in {1/.48,2/1.68,3/2.88,4/4.08,5/5.28} {
  \draw[feqcolor\nn,line width=1.15pt] (\xx,-.61)--++(.33,0);
  \node[anchor=west,font=\small] at (\xx+.42,-.61) {$n=\nn$};
}
\end{tikzpicture}
\endgroup

\caption{Numerically computed equality functions for $n=1,\ldots,5$
in the canonical balanced normalization, with horizontal sections
$r_n(t)B_2^n$. The horizontal variable is $\rho=|x|$.
Left: the functions associated with $r_n(t)$.
Right: their Legendre duals, associated with $r_n(-t)$.
The blue curve is $\rho^2/2$ in both panels.
The higher-dimensional duals are zero on an interval and tend to
infinity at the finite radius of their domain; only the displayed
height range is shown.}
\label{fig:functional-equality-families}
\end{figure}
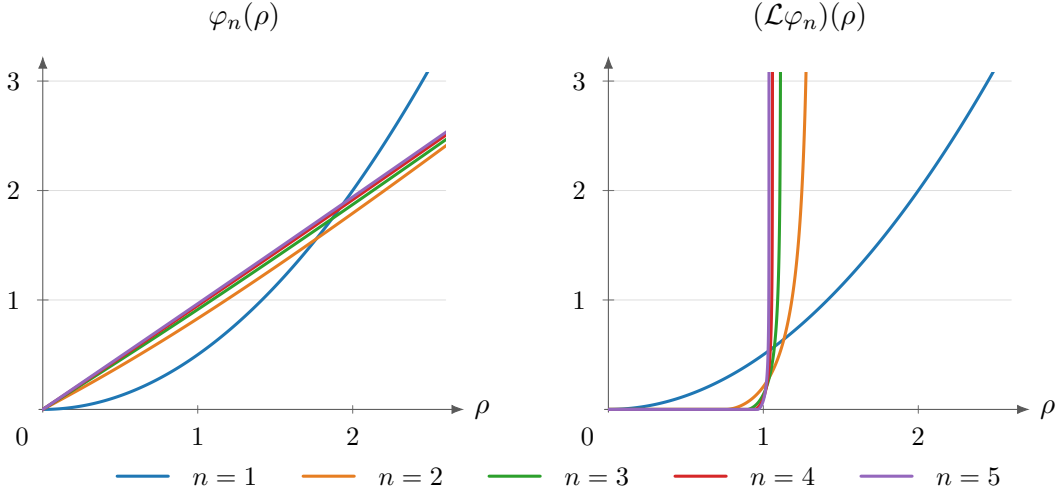

%% file: sharp_santalo_paper_2026-09-29_1623.bbl
\begin{thebibliography}{99}

\bibitem{Arnold1989}
V.~I. Arnold,
\emph{Mathematical Methods of Classical Mechanics},
second edition, Graduate Texts in Mathematics, vol.~60,
Springer-Verlag, New York, 1989.

\bibitem{ArtsteinICM}
S. Artstein-Avidan,
A duality inspired journey in convex geometric analysis,
in \emph{Proceedings of the International Congress of Mathematicians 2026,
Volume~4: Invited Lectures (Sections 5--8)},
Society for Industrial and Applied Mathematics, Philadelphia, 2026,
pp.~619--635, doi:10.1137/25M1803334.

\bibitem{ArtsteinAvidanGiannopoulosMilmanbook}
S. Artstein-Avidan, A. Giannopoulos, and V.~D. Milman,
\emph{Asymptotic Geometric Analysis, Part I},
Mathematical Surveys and Monographs, vol.~202,
American Mathematical Society, Providence, RI, 2015,
doi:10.1090/surv/202.

\bibitem{ArtsteinAvidanMilman2009}
S. Artstein-Avidan and V.~D. Milman,
The concept of duality in convex analysis, and the characterization of the
Legendre transform,
\emph{Ann. of Math. (2)} \textbf{169} (2009), no.~2, 661--674,
{doi:10.4007/annals.2009.169.661}.

\bibitem{ArtsteinAvidanMilman2011}
S. Artstein-Avidan and V.~D. Milman,
Hidden structures in the class of convex functions and a new duality
transform,
\emph{J. Eur. Math. Soc. (JEMS)} \textbf{13} (2011), no.~4, 975--1004,
{doi:10.4171/JEMS/273}.

\bibitem{ArtsteinKlartagMilman2004}
S. Artstein-Avidan, B. Klartag, and V.~D. Milman,
The Santal\'o point of a function, and a functional form of the Santal\'o
inequality,
\emph{Mathematika} \textbf{51} (2004), no.~1--2, 33--48,
{doi:10.1112/S0025579300015497}.

\bibitem{Artstein-Slomka}
S. Artstein-Avidan and B. A. Slomka,
\emph{A note on Santal\'o inequality for the polarity transform and its reverse},
Proc. Amer. Math. Soc. \textbf{143} (2015), no. 4, 1693-1704.

 
\bibitem{BallThesis1986}
K.~M. Ball,
\emph{Isometric Problems in $\ell_p$ and Sections of Convex Sets},
Ph.D. thesis, University of Cambridge, 1986.

\bibitem{Bobkov2010}
S.~G. Bobkov,
Convex bodies and norms associated to convex measures,
\emph{Probab. Theory Related Fields} \textbf{147} (2010), no.~1--2,
303--332.

 
 
 

\bibitem{FradeliziMeyer2007}
M. Fradelizi and M. Meyer,
Some functional forms of Blaschke--Santal\'o inequality,
\emph{Math. Z.} \textbf{256} (2007), no.~2, 379--395.

\bibitem{GSS}
S. Gilboa, A. Segal and B. A. Slomka,
\emph{The scaled polarity transform and related inequalities},
J. Fixed Point Theory Appl. \textbf{27} (2025), no. 3, Paper No. 73.


\bibitem{Lehec}  
J.~Lehec,
\newblock A direct proof of the functional Santaló inequality,
\newblock \emph{C. R. Math. Acad. Sci. Paris} \textbf{347} (2009), no.~1--2, 55--58.

\bibitem{MeyerPajor1989}
M. Meyer and A. Pajor,
On Santal\'o's inequality,
in \emph{Geometric Aspects of Functional Analysis (1987--88)},
Lecture Notes in Mathematics, vol.~1376, Springer, Berlin, 1989,
pp.~261--263.

\bibitem{MeyerPajor1990}
M. Meyer and A. Pajor,
On the Blaschke-Santal\'o inequality,
\emph{Arch. Math. (Basel)} \textbf{55} (1990), 82--93.

 
  
\bibitem{Rotem2014Santalo}
L. Rotem,
A sharp Blaschke-Santal\'o inequality for $\alpha$-concave functions,
\emph{Geom. Dedicata} \textbf{172} (2014), 217--228.

\bibitem{Santalo1949}
L.~A. Santal\'o,
Un invariante af\'in para los cuerpos convexos del espacio de $n$
dimensiones,
\emph{Portugal. Math.} \textbf{8} (1949), no.~4, 155--161.

\bibitem{Schneider2014}
R. Schneider,
\emph{Convex Bodies: The Brunn-Minkowski Theory},
second expanded edition, Encyclopedia of Mathematics and its Applications,
vol.~151, Cambridge University Press, Cambridge, 2014.



 

\end{thebibliography}
